\documentclass[12pt,a4paper]{article}
\usepackage{amsmath}
\usepackage{amsthm}
\usepackage{amsfonts}
\usepackage{amssymb}
\usepackage{stmaryrd}
\usepackage{latexsym}
\usepackage{eucal}

\theoremstyle{theorem}
\newtheorem{theorem}{Theorem}[section]
\newtheorem{lemma}[theorem]{Lemma}
\newtheorem{corollary}[theorem]{Corollary}
\newtheorem{proposition}[theorem]{Proposition}

\newtheorem{definition}[theorem]{Definition}
\newtheorem{procedure}[theorem]{Procedure}
\theoremstyle{definition}
\newtheorem{remark}[theorem]{Remark}
\newtheorem{example}[theorem]{Example}
\renewenvironment{proof}[1][\proofname]{\noindent \textbf{Proof. }}
{%
  \qed\endtrivlist
}

\author{V. A. Shcherbacov
\\ Institute of Mathematics and Computer Science\\
Academy of Sciences of Moldova \\
 Academiei str. 5,
Chi\c{s}in\u{a}u,  MD$-$2028,
 Moldova \\
 E-mail: \emph{scerb@math.md}\quad}

\title{A-nuclei and A-centers of a quasigroup}               % Here insert the title
\date{}               % this must be left empty
\begin{document}

\tolerance1500

\maketitle

\begin{abstract}
  A-nuclei (groups of regular permutations) of a quasigroup are studied. A quasigroup is A-nuclear if and only if it is group isotope. Any quasigroup with permutation medial or  paramedial identity is an  abelian group isotope.  Definition   of A-center of a quasigroup is given. A quasigroup is A-central if and only if it is abelian group isotope.  If a quasigroup is central in Belyavskaya-Smith sense, then it is A-central. Conditions when  A-nucleus define  normal  congruence of a quasigroup  are established, conditions normality of nuclei of some inverse quasigroups are given. Notice, definition of A-nucleus of  a loop and A-center of a loop coincides, in fact, with corresponding standard definition.


MSC \: {20N05}.
\end{abstract}

\footnote{\textsf{
The research to this paper was sponsored by Special Projects Office, Special and Extension Programs of the
Central European University. %%% Corporation. Grant Application Number: 201011/R.
}}

\def\Ibf{{\text{\bf P}}}
\def\Lbf{{\text{\bf L}}}
\def\Rbf{{\text{\bf R}}}
\def\Tbf{{\text{\bf T}}}

\tableofcontents

\section{Introduction}

\subsection{Historical notes}

G.N. Garrison \cite{GARRISON} was the first who  has defined   concept of a quasigroup nucleus.
A nucleus (middle, left, or right) "measures" how far is a quasigroup from a group.
Unfortunately, if a quasigroup $(Q, \cdot)$ has a non-trivial Garrison's nucleus, then this quasigroup is a right loop, or a left loop, either is a loop \cite[p. 17]{HOP}.
Therefore many authors tried to diffuse (to generalize) concept of nucleus on "proper" quasigroup case.

Various definitions of groupoid, quasigroup and loop nuclei  were given and researched by A. Sade, R.H. Bruck,  V.D. Belousov, P.I. Gramma, A.A. Gva\-ra\-miya,  M.D. Kitaroag\u a, G.B. Be\-lyav\-s\-kaya,  H.O. Pflugfelder  and many  other mathematicians  \cite{BRUCK_46, SADE_58, Kertez_Sade_59, RHB, GRAMMA_PI, VD, GVARAM_72, BELAS, BEL_RUSS_75, MDK, VDB_GVAR, LUMPOV_90, gbb, gbb1, HOP, SCERB_03, KEED_SCERB}.

V.D. Belousov discovered  connections of quasigroup nuclei  with the gro\-ups of regular mappings of quasigroups \cite{vdb0, VD, 1a,  kepka71, ks3, KEED_SCERB, kepka_05}.

    Belousov studied   autotopisms of the form $(\alpha , \varepsilon , \alpha )$  of a quasigroup $(Q,\circ )$. In this paper we study   autotopisms of the form $(\alpha , \varepsilon , \gamma )$, i.e. we use Kepka's generalization  \cite{kepka71, kepka75}. By this approach "usual" nuclei of quasigroups are obtained as some   orbits by the action of components of  A-nuclei on the set $Q$. The study of A-nuclei of $(\alpha, \beta, \gamma)$- and $(r; s; t)$-inverse quasigroups was initiated in \cite{ks3, KEED_SCERB}.

In the paper we develop Belousov approach to parastrophes of quasigroup nuclei \cite{vdb_63} using concept of middle translation \cite{BELAS}.
Also we make an attempt to extend some  results   of  A. Drapal  and  P. Jedli\u cka \cite{DR_04, DRAP_JEDL} about  normality of nuclei.

\subsection{Quasigroups, identity elements, translations}

For convenience of readers we start from some  definitions which it is possible to find in \cite{VD, 1a, HOP,
JDH_2007, SCERB_03}.

\begin{definition} \label{GROU}
{\it A binary groupoid $(G, A)$ is understood to be a non-empty set $G$ together with a binary operation $A$.}
\end{definition}

As usual the product of mappings is their consecutive realization.  We shall use the following (left) order of
multiplication of maps: if $\mu, \nu$ are some maps, then $(\mu\nu) (x)= \mu (\nu (x))$.

\begin{definition} \label{RIGHT_QUSIGROUP} \label{LEFT_QUSIGROUP}
A groupoid $(Q,\circ )$ is called a \textit{right  quasigroup} (a \textit{left  quasigroup}) if, for all $a,b
\in Q$, there exists a unique solution $x \in Q$ to the equation $x\circ a = b$ ($a\circ x = b$), i.e. in this
case any right (left) translation of the groupoid $(Q,\circ )$ is a bijective map  of the set $Q$.
\end{definition}

In this case any right (left) translation of the groupoid $(Q,\circ )$ is a bijective map  of the set $Q$.

\begin{definition} \label{def2}
A groupoid $(Q,\circ )$ is called a \textit{quasigroup} if for any fixed pair of elements $a, b
\in Q$ there exist a unique solution $x \in Q$ to the equation $x\circ a = b$ and a unique solution $y \in Q$ to the equation $a\circ y = b$.
\end{definition}

\begin{definition} An element $f$ is a left  identity element of a quasigroup $(Q,\cdot)$ means that $f \cdot x = x$ for all $x\in
Q$.

An element $e$ is a right  identity element of a quasigroup $(Q,\cdot)$ means that $x \cdot e = x$ for all $x\in
Q$.

An element $s$ is a middle  identity element of a quasigroup $(Q,\cdot)$ means that $x \cdot x = s$ for all $x\in
Q$.

An element $e$ is an identity element of a quasigroup $(Q,\cdot)$, if  $x\cdot e = e \cdot x = x$ for all $x\in Q$.
\end{definition}

\label{loop_DEFF}
\begin{definition} A quasigroup $(Q,\cdot)$ with a left  identity  element $f \in Q$  is called a {\it left loop}.

A quasigroup $(Q,\cdot)$ with a right  identity  element $e \in Q$  is called a {\it right  loop}.

A quasigroup $(Q,\cdot)$ with a middle  identity  element $s \in Q$  is called an {\it unipotent quasigroup}.

A quasigroup $(Q,\cdot)$ with an  identity  element $e \in Q$  is called a {\it loop}.
\end{definition}

Define in a quasigroup $(Q, \cdot)$ the following mappings: $f : x\mapsto f(x)$, where $f(x)\cdot x = x$; $e : x\mapsto e(x)$, where $x\cdot e(x) = x$;
 $s : x\mapsto s(x)$, where  $s(x) = x\cdot x$.

\begin{remark} \label{LEFT_RIGHT_LOOP}
In a left loop   $f(Q) = 1$,  in a right loop   $e(Q) = 1$, in an unipotent quasigroup   $s(Q) = 1$, where $1$ is a fixed element of the set  $Q$.
\end{remark}

\begin{definition} A quasigroup $(Q,\cdot)$ with  identity   $x\cdot x = x$   is called an {\it idempotent quasigroup}.
\end{definition}

\begin{remark} \label{IDEMPOT_QUAS}
In an idempotent quasigroup the mappings $e, f, s$ are identity permutations of the set $Q$. Moreover, if one from these three mappings is identity permutation, then all other also are identity permutations.
\end{remark}

In a loop $(Q,\cdot)$, as usual \cite{VD}, an element $b$ such that $b\cdot a = 1$ is called left inverse element to the element $a$, $b = {}^{-1}(a)$; an element $c$ such that $a\cdot c = 1$ is called right  inverse element to the element $a$, $c = (a)^{-1}$.

\begin{definition} \label{def1}  A binary  groupoid $(Q, A)$ with a binary operation $A$ such that in the equality $A(x_1,
x_2) = x_3$ knowledge of any $2$ elements of $x_1, x_2,x_3$ uniquely specifies the remaining one is called a
\textit{binary quasigroup}  \cite{2, MUFANG}. \end{definition}

From Definition \ref{def1} it follows that  with any quasigroup $(Q, A)$ it possible to associate else
$(3!-1)=5$ quasigroups, so-called parastrophes of quasigroup $(Q, A)$:
 $A(x_1, x_2) = x_3 \Leftrightarrow A^{(12)}(x_2, x_1) = x_3
\Leftrightarrow {A}^{(13)}(x_3, x_2) = x_1 \Leftrightarrow {A}^{(23)}(x_1, x_3) = x_2 \Leftrightarrow
{A}^{(123)}(x_2, x_3) = x_1 \Leftrightarrow {A}^{(132)}(x_3, x_1) = x_2.$

\medskip

We shall denote:

the operation of $(12)$-parastrophe of a quasigroup $(Q, \cdot)$  by $\ast$;

the operation of $(13)$-parastrophe of a quasigroup $(Q, \cdot)$  by $/$;

the operation of $(23)$-parastrophe of a quasigroup $(Q, \cdot)$  by $\backslash$.

\smallskip

We have defined left and right translations of a groupoid and, therefore, of  a quasigroup. But for  quasigroups
it is possible to define  and the third kind of translations, namely the map $P_a : Q\longrightarrow Q$, $x\cdot
P_{a}x = a$  for all $x \in Q$ \cite{BELAS}.

\smallskip

In the following table connections between different kinds of translations in different parastrophes
of a quasigroup  $(Q,\cdot)$  are given \cite{DUPLAK, SCERB_91}. In fact this table  is in  \cite{BELAS}.

\hfill Table 1 \label{Table_0}
\[
\begin{array}{|c||c| c| c| c| c| c|}
\hline
  & \varepsilon  & (12) & (13) & (23) & (123) & (132)\\
\hline\hline
R  & R & L & R^{-1} & P & P^{-1} & L^{-1} \\
\hline
L  & L & R & P^{-1} & L^{-1} & R^{-1} & P \\
\hline
P  & P & P^{-1} & L^{-1} & R & L & R^{-1} \\
\hline
R^{-1} & R^{-1} & L^{-1} & R & P^{-1} &  P & L \\
\hline
L^{-1}  & L^{-1} & R^{-1} & P & L & R & P^{-1} \\
\hline
P^{-1}  & P^{-1} & P & L & R^{-1} & L^{-1} & R \\
\hline
\end{array}
\]

With any quasigroup $(Q,\cdot)$ we can associate the sets of all left translations ($\mathbb {L}$), right
translations ($\mathbb {R}$) and  middle translations ($\mathbb {P}$). Denote the groups generated by all left, right
and middle translations of a quasigroup $(Q,\cdot)$ as $LM(Q,\cdot)$, $RM(Q,\cdot)$ and $PM(Q,\cdot)$,
respectively.

The group generated by all left and right translations of a quasigroup $(Q,\cdot)$ is called (following articles
of A.A. Albert \cite{A1, A2})  multiplication group of a quasigroup. This group usually is denoted by $M(Q,\cdot)$.
By $FM(Q,\cdot)$ we shall denote  a group generated by the sets $\mathbb{L,R, P}$ of a quasigroup $(Q,\cdot)$.

Connections between different kinds of local identity elements (and of course of identity elements) in different parastrophes of a quasigroup
$(Q,\cdot)$ are given in the following table  \cite{SC_89_1, SCERB_91}.

\medskip

\hfill Table 2 \label{TABLE_two}
\[
\begin{array}{|c||c| c| c| c| c| c|}
\hline
  & \varepsilon  & (12)  & (13) & (23)   & (123) & (132) \\
\hline\hline
 {f}  &  {f} &  {e} &  {s} &  {f} &  {e} &  {s} \\
\hline
 {e}  &  {e} &  {f} &  {e} &  {s} &  {s} &  {f} \\
\hline
 {s}  &  {s} &  {s} &  {f} &  {e} &  {f} &  {e} \\
\hline
\end{array}
\]

In Table 2, for example, ${s}^{(123)} =  {f}^{(\cdot)}.$

\begin{lemma}
Parastrophic image of a loop is a loop, or an unipotent left loop, either an unipotent right loop. Parastrophic image of an idempotent quasigroup is an idempotent quasigroup.
\end{lemma}
\begin{proof}
It is possible to use Table 2.
\end{proof}

In this paper an algebra (or algebraic structure) is a set $A$ together with a collection of operations defined on $A$.
T. Evans \cite{EVANS_51}  defined  a binary quasigroup as an algebra $(Q, \cdot, /, \backslash)$ with three
binary operations. He has used  the following identities:
\begin{equation}
x\cdot(x \backslash y) = y, \label{(1)}
\end{equation}
\begin{equation}
(y / x)\cdot x = y, \label{(2)}
\end{equation}
\begin{equation}
x\backslash (x \cdot y) = y, \label{(3)}
\end{equation}
\begin{equation}
(y \cdot x)/ x = y. \label{(4)}
\end{equation}

  \begin{definition}    An algebra $(Q, \cdot, \backslash, /)$ with identities (\ref{(1)}) -- (\ref{(4)}) is called a \textit{quasigroup} \cite{EVANS_51, BIRKHOFF, BURRIS, VD, 1a,
HOP}. \label{def3}
\end{definition}

In any quasigroup $(Q, \cdot,
\backslash, /)$ the following identities are true  \cite{BIRKHOFF}:
\begin{equation}
(x/y)\backslash x = y, \label{(53)}
\end{equation}
\begin{equation}
x/(y\backslash x) = y. \label{(63)}
\end{equation}

\subsection{Isotopism}

\begin{definition}
A binary  groupoid $(G, \circ)$ is an isotope
of a binary groupoid $(G, \cdot)$ (in other words $(G, \circ)$ is an isotopic image of $(G, \cdot)$), if there exist permutations $\mu_1, \mu_2, \mu_3$ of the set
$G$ such that
$ x_1 \circ  x_2 = \mu^{-1}_3
(\mu_1 x_1  \cdot \mu_2 x_2)$ for all $x_1,  x_2 \in G$  \cite{VD, 1a, HOP}.
\end{definition}

 We can write this fact and in the
form $(G,\circ) = (G,\cdot)T$, where $T = (\mu_1, \mu_2, \mu_3)$.

\begin{definition}
If  $(\circ) = (\cdot)$, then the  triple  $(\mu _{1},\mu _{2},\mu _{3})$
is called an autotopy of groupoid $(Q,\cdot)$.
\end{definition}

 The set of all autotopisms of a groupoid $(Q, \cdot)$ forms a group relative to usual component-wise multiplications of autotopisms \cite{VD, 1a, HOP}. We shall denote this group  as $Avt(Q,\cdot)$.

The last component of an autotopy of a groupoid is called a \textit{quasiautomorphism}.

\begin{lemma} \label{TRANSLATIONS}  \begin{enumerate}
\item If $(Q, \circ) = (Q,\cdot)(\alpha, \varepsilon, \varepsilon)$,  then $L^{\circ}_x = L^{\cdot}_{\alpha x}$,  $R^{\circ}_x  = R^{\, \cdot}_x \alpha  $, $P^{\circ}_x = P^{\cdot}_x \alpha $ for all $x\in Q$.
\item If $(Q, \circ) = (Q,\cdot)(\varepsilon, \beta, \varepsilon)$,  then $L^{\circ}_x   = L^{\cdot}_{x} \beta $,  $R^{\,\circ}_x = R^{\,\cdot}_{\beta x}  $, $(P^{\circ}_x)^{-1}  = (P^{\, \cdot}_x )^{-1} \beta $ for all $x\in Q$.
\item If  $(Q, \circ) = (Q,\cdot)(\varepsilon, \varepsilon, \gamma)$,  then  $L^{\circ}_x = \gamma^{-1} L^{\cdot}_x$, $R^{\circ}_x = \gamma^{-1} R^{\,\cdot}_x$,  $P^{\circ}_{x} = P^{\,\cdot}_{\gamma x}$ for all $x\in Q$.
  \end{enumerate}
\end{lemma}
\begin{proof}
Case 1.
We can re-write equality   $(Q, \circ) = (Q,\cdot)(\alpha, \varepsilon, \varepsilon)$ in the form  $x\circ y = \alpha x \cdot y = z$ for all $x, y \in Q$. Therefore    $L^{\circ}_x  y = L^{\cdot}_{\alpha x}y$,  $R^{\circ}_y x = R^{\,\cdot}_y \alpha x $, $P^{\circ}_z x = P^{\,\cdot}_z \alpha x$.

Case 2.
We can re-write equality   $(Q, \circ) = (Q,\cdot)(\varepsilon, \beta, \varepsilon)$ in the form  $x\circ y = x \cdot \beta y = z$ for all $x, y \in Q$. Therefore $L^{\circ}_x  y = L^{\cdot}_{x} \beta y$,  $R^{\circ}_y x = R^{\cdot}_{\beta y}  x $, $(P^{\circ}_z)^{-1} y = (P^{\cdot}_z )^{-1} \beta y$.

Case 3.
We can re-write equality   $(Q, \circ) = (Q,\cdot)(\varepsilon, \varepsilon, \gamma)$ in the form  $x\circ y = \gamma^{-1}(x \cdot y)$ for all $x, y \in Q$. Therefore   $L^{\,\circ}_x = \gamma^{-1} L^{\,\cdot}_x$, $R^{\,\circ}_y = \gamma^{-1} R^{\,\cdot}_y$. If  $x\circ y = \gamma^{-1}(x \cdot y) = z$, then   $P^{\,\circ}_z  x = y$,  $P^{\,\cdot}_{\gamma \, z} x = y $, $P^{\,\circ}_z = P^{\,\cdot}_{\gamma \, z}$.
\end{proof}

\begin{corollary} \label{ONE_COMPONT_ISOTOPY}
In conditions of Lemma \ref{TRANSLATIONS}
\begin{enumerate}
\item
\begin{enumerate}
\item[(i)] if $\alpha = (R^{\,\cdot}_a)^{-1}$,  then $R^{\,\circ}_a = \varepsilon$,  $(Q, \circ)$ is a right loop.
\item[(ii)] if $\alpha = (P^{\,\cdot}_a)^{-1}$, then  $P^{\,\circ}_a = \varepsilon$, $(Q, \circ)$ is an unipotent quasigroup.
  \end{enumerate}
  \item
\begin{enumerate}
\item[(i)] if $\beta = (L^{\,\cdot}_b)^{-1}$, then  $L^{\circ}_b = \varepsilon$, $(Q, \circ)$ is a left loop.
\item[(ii)] if $\beta = P^{\,\cdot}_b$, then  $P^{\,\circ}_b = \varepsilon$, $(Q, \circ)$ is an unipotent quasigroup.
  \end{enumerate}
  \item
\begin{enumerate}
\item[(i)] if $\gamma = L^{\cdot}_c$, then  $L^{\circ}_c = \varepsilon$, $(Q, \circ)$ is a left loop.
\item[(ii)] if $\gamma = R^{\,\cdot}_c$, then  $R^{\,\circ}_c = \varepsilon$, $(Q, \circ)$ is a right loop.
  \end{enumerate}
\end{enumerate}
\end{corollary}
\begin{proof}
The proof follows from Lemma \ref{TRANSLATIONS}.
\end{proof}

\begin{lemma}\label{LP_ISOTOPII}
\label{parastrophic_loop}
If $(Q,\circ) = (Q,\cdot) (\alpha, \beta, \varepsilon)$ and $(Q,\circ)$ is a loop, then
there exist elements $a, b \in Q$ such that $\alpha = R^{-1}_a$, $\beta = L^{-1}_b$ \cite{VD, 1a}.

If $(Q,\circ) = (Q,\cdot) (\varepsilon, \beta, \gamma)$ and $(Q,\circ)$ is a right unipotent loop, then
there exist elements $a, b \in Q$ such that $\beta = P_{a}$, $\gamma = R_{b}$.

If $(Q,\circ) = (Q,\cdot) (\alpha, \varepsilon, \gamma)$ and $(Q,\circ)$ is a left  unipotent loop, then
there exist elements $a, b \in Q$ such that $\alpha = P^{-1}_{a}$, $\gamma = L_b$.
\end{lemma}
\begin{proof}  Let $x \circ y = \alpha x \cdot \beta y$. If $x=1$, then we have $1\circ y = y = \alpha 1
\cdot \beta y$. Therefore $L_{\alpha 1}\beta = \varepsilon$, $\beta = L^{-1}_{\alpha 1}$. If we take $y = 1$,
then we have $x\circ 1 = x = \alpha x \cdot \beta 1$, $R_{\beta 1}\alpha = \varepsilon$, $\alpha = R^{-1}_{\beta
1}$.

Let $x \circ y = \gamma^{-1}( x \cdot \beta y)$.
If $y = 1$, then we have $x = \gamma^{-1}(x\cdot \beta 1)$. Therefore $\gamma x = R_{\beta 1} x$.
If $x = y$, then $1 = \gamma^{-1}(x \cdot \beta x)$, $\gamma 1 = x \cdot \beta x$, $P_{\gamma 1} x = \beta x$. Denote
$\beta 1$ by $b$, $\gamma 1$ by $a$.

Let $x \circ y = \gamma^{-1}( \alpha x \cdot  y)$.
If $x = 1$, then we have $y = \gamma^{-1}(\alpha 1\cdot  y)$. Therefore $\gamma y = L_{\alpha 1} y$.
If $x = y$, then $1 = \gamma^{-1}(\alpha  \cdot  x)$, $\gamma 1 = \alpha x \cdot x$, $P^{-1}_{\gamma 1} x = \alpha x$. Denote
$\alpha 1$ by $b$, $\gamma 1$ by $a$.
\end{proof}

\begin{definition} \label{LP_ISOt_def} Isotopism of the form $(R^{-1}_a, L^{-1}_b, \varepsilon)$ where $L_b, R_a$ are left and right translations of the quasigroup  $(Q,\cdot)$ is called LP-isotopism (loop isotopism) \cite{VD, 1a}.
\end{definition}

\begin{theorem} \label{LP_ISOT_AND_ANALOGS}
Any LP-isotope of a quasigroup $(Q,\cdot)$ is a loop \cite{VD, 1a}.

If $(Q,\circ) = (Q,\cdot) (\varepsilon, P_{a}, R_{b})$, where $(Q,\cdot)$ is a quasigroup,  then  $(Q,\circ)$ is a unipotent right  loop.

If $(Q,\circ) = (Q,\cdot) (P^{-1}_{a}, \varepsilon, L_b)$, where $(Q,\cdot)$ is a quasigroup, then  $(Q,\circ)$ is a unipotent left  loop.
\end{theorem}

\begin{proof} Case 1. Prove that quasigroup $(Q,\circ)$, where $x\circ y = R^{-1}_a x \cdot  L^{-1}_b y$, is a loop.
Let $1 = b\cdot a.$ If we take $x=1$, then  $1\circ y = R^{-1}_a ba \cdot  L^{-1}_b y = R^{-1}_a R_a b \cdot
L^{-1}_b y = b\cdot L^{-1}_b y = L_bL^{-1}_b y = y$.

If we take $y=1$, then we have $x \circ 1 = R^{-1}_a x
\cdot L^{-1}_b ba = R^{-1}_a x \cdot a = R_a R^{-1}_a x = x$. Element $1$ is the identity element of the
quasigroup $(Q,\circ)$.

Case 2. Prove that quasigroup $(Q,\circ)$, where $x\circ y = R^{-1}_b ( x \cdot P_a  y)$, is a right unipotent loop.
Let $1 = R^{-1}_b a = a \slash b.$ If we take $y = 1$, then  $x\circ 1 = R^{-1}_b (x \cdot P_a R^{-1}_b a) = R^{-1}_b (x\cdot
b) = x$, since $P_a R^{-1}_b a = b$. Indeed, if $P_a R^{-1}_b a = b$, then $R^{-1}_b a \cdot b = a$, $a=a$. Also we can use identity (\ref{(2)}). Therefore element $1$ is the right identity element of quasigroup $(Q,\circ)$.

If we take $x=y$, then we have $x\circ x = R^{-1}_b ( x \cdot P_a  x) = R^{-1}_b a = 1$ since by definition of middle translation  $x \cdot P_a  x= a$,  $(Q,\circ)$ is an unipotent quasigroup.

Case 3. Prove that quasigroup $(Q,\circ)$, where $x\circ y = L^{-1}_b ( P^{-1}_{a} x \cdot  y)$, is a left  unipotent loop. Let $1 = L^{-1}_b a = b \backslash a.$
If we take $x = 1$, then  $1\circ y = L^{-1}_b ( P^{-1}_{a} L^{-1}_b a  \cdot  y) = L^{-1}_b (b   \cdot  y) = y$, since $P^{-1}_{a} L^{-1}_b a = b$. Indeed, if $P^{-1}_{a} L^{-1}_b a = b$, then $b \cdot L^{-1}_b a = a$, $a=a$.  Therefore element $1$ is the left  identity element of quasigroup $(Q,\circ)$.

If we take $x=y$, then we have $x\circ x = L^{-1}_b ( P^{-1}_{a} x \cdot  x) = L^{-1}_b a = 1$ since by definition of middle translation  $P^{-1}_a x \cdot x= a$,  $(Q,\circ)$ is an unipotent quasigroup.
 \end{proof}
\begin{remark}
Instead of direct proof of Theorem  \ref{LP_ISOT_AND_ANALOGS} it is possible to use Corollary \ref{ONE_COMPONT_ISOTOPY}.
\end{remark}

There exists  well known connection between translations of a quasigroup and translations of its LP-isotopes \cite{12, IHRINGER}.  In the following lemma we extend this connection.

\begin{lemma} \label{LOOP_TRANSLATIONS}
If $x\circ y = R^{-1}_a x \cdot L^{-1}_b y$, then $L_x^{\circ}  =
L_{R^{-1}_a x} L^{-1}_b $, $R_y^{\circ}  = R_{L^{-1}_b y}R^{-1}_a $, $P^{\circ}_z  = L_b P_z  R^{-1}_a $.

 If $x\circ y = R^{-1}_b ( x \cdot P_a  y)$, then  $L^{\circ}_x  = R^{-1}_b L_x P_a $,  $R^{\circ}_y  = R^{-1}_b R_{P_a  y}$,  $P^{\circ}_z  = P^{-1}_a P_{R_b z}$.

 If $x \circ y =  L^{-1}_{b}(P^{-1}_{a} x \cdot  y)$, then  $L^{\circ}_x  =  L^{-1}_{b}L_{P^{-1}_{a} x}$, $R^{\circ}_y  =  L^{-1}_{b}  R_y P^{-1}_{a} $,  $P^{\circ}_z = P_{b\cdot z} P^{-1}_{a}$.
\end{lemma}
\begin{proof}
Case 1. If $x\circ y = R^{-1}_a x \cdot L^{-1}_b y = z$, then  $P^{\circ}_z x =y$, $P_z  R^{-1}_a x = L^{-1}_b y$, $L_b P_z  R^{-1}_a x =  y$, $P^{\circ}_z x = L_b P_z  R^{-1}_a x$.

Case 2.
 If $x\circ y = R^{-1}_b ( x \cdot P_a  y) = z$, then  $R^{\circ}_y x = R^{-1}_b R_{P_a  y} x $,  $P^{\circ}_z x = y$, $ R^{-1}_b ( x \cdot P_a  y) = z$,   $P_{R_b z} x  =  P_a y$, $P^{-1}_a P_{R_b z} x  =   y$, $P^{\circ}_z x = P^{-1}_a P_{R_b z} x$, $L^{\circ}_x y = R^{-1}_b L_x P_a  y$.

Case 3.
 If $x \circ y =  L^{-1}_{b}(P^{-1}_{a} x \cdot  y) = z$, then $P^{\circ}_z x = y$, $P^{-1}_{a} x \cdot  y = L_{b} z$, $P_{L_{b} z} P^{-1}_{a} x  =  y $, $P^{\circ}_z = P_{b\cdot z} P^{-1}_{a}$,  $L^{\circ}_x y =  L^{-1}_{b}L_{P^{-1}_{a} x} y$, $R^{\circ}_y  x  =  L^{-1}_{b}  R_y P^{-1}_{a}  x $.
\end{proof}

We shall use the following well known fact.
\begin{lemma} \label{PARASTR_ISOT_GR_ISOT}
If a quasigroup $(Q, \cdot)$ is a group isotope, i.e. $(Q, \cdot) \sim (Q,+)$, where $(Q,+)$ is a group, then any parastrophe of this quasigroup also is a group isotope \cite{SOH_95_I}.
\end{lemma}
\begin{proof}
The proof is based on the fact that any parastrophe of a  group  is an isotope of this group, i.e.  $(Q, +)^{\sigma} \sim (Q, +)$ \cite[Lemma 5.1]{VD}, \cite[p. 53]{1a}.
 If $(Q, \cdot) \sim (Q,+)$, then $(Q, \cdot)^{\sigma} \sim (Q,+)^{\sigma}\sim (Q,+)$.
\end{proof}

\subsection{Autotopy. Leakh Theorem}

\begin{lemma} \label{ISOM_AUTOT_GR} If quasigroups $(Q, \circ)$ and $(Q, \cdot)$ are isotopic
with isotopy $T$, i.e. $(Q,\circ) = (Q, \cdot)T$, then $ Avt(Q, \circ) = T^{-1} Avt  (Q, \cdot) T$ (\cite{1a},  Lemma 1.4).
\end{lemma}

Lemma \ref{ISOM_AUTOT_GR} allows (up to isomorphism) to reduce the study of autotopy group of a quasigroup to the study autotopy group of an LP-isotope of this quasigroup, i.e. to the study of autotopy group of a loop.

\begin{corollary} \label{COROLL_VDB_LEMMA}
If $H$ is a subgroup of the group $Avt  (Q, \cdot)$, then $T^{-1}HT$ is a subgroup of the group $ Avt(Q, \circ)$.
\end{corollary}
\begin{proof}
The proof follows from Lemma \ref{ISOM_AUTOT_GR} and standard algebraic facts \cite{KM}.
\end{proof}
\begin{lemma} \label{TWO_AUTOTOPY_COMPONENTS} If $T$ is a quasigroup  autotopy, then any its two components define the third component uniquely.
\end{lemma}
\begin{proof} If $(\alpha_1, \beta, \gamma)$ and $(\alpha_2, \beta, \gamma)$ are autotopies, then
$(\alpha^{-1}_2, \beta^{-1}, \gamma^{-1})$ is an autotopy and $(\alpha_1 \alpha^{-1}_2, \beta \beta^{-1}, \gamma
\gamma^{-1}) = (\alpha_1 \alpha^{-1}_2, \varepsilon, \varepsilon)$ is an autotopy too. We can re-write the last
 autotopy in such form: $\alpha_1 \alpha^{-1}_2 x \cdot y = x\cdot y$, then $\alpha_1 = \alpha_2$.

If $(\varepsilon, \varepsilon, \gamma_1 \gamma_2)$ is an autotopy, then we have
$x\cdot y = \gamma_1\gamma^{-1}_2(x\cdot y)$. If we put in the last equality
$y=e(x)$, then we obtain $x=\gamma_1 \gamma^{-1}_2 x$ for all $x\in Q$, i.e.
$\gamma_1 = \gamma_2$.
\end{proof}

\begin{corollary} \label{COROLL_OF_LEAKH_TH}
If two components of a quasigroup autotopy are identity mappings, the the third component also is an identity mapping.
\end{corollary}
\begin{proof}
The proof follows from Lemma \ref{TWO_AUTOTOPY_COMPONENTS}.
\end{proof}

There exists \cite{IVL} more strong result than Lemma \ref{TWO_AUTOTOPY_COMPONENTS}.
I.V. Leakh proves this result using geometrical (net theory) approach in more general case than in the following

\begin{theorem} Leakh Theorem. \label{DEF_OF_COMP_BY_AUTOT_AND_EL} Any  autotopy $T = (\alpha_1, \alpha_2, \alpha_3)$  of a quasigroup $(Q, \circ)$ is uniquely defined   by any autotopy component $\alpha_i$, $i\in \{1, 2, 3 \}$,  and by element $b = \alpha_j a$, where $a$ is any fixed element of set $Q$,  $i\neq j$   \cite{IVL}.
\end{theorem}
\begin{proof} Case 1. $i = 1, j = 2$. If we have autotopies $(\alpha, \beta_1, \gamma_1)$ and $(\alpha, \beta_2, \gamma_2)$ such that $\beta_1 a = \beta_2 a = b$, then we have $\alpha x \circ \beta_1 a = \gamma_1(x\circ a)$ and
$\alpha x \circ \beta_2 a = \gamma_2(x\circ a)$. Since the left sides of the last equalities are equal, then we have
$\gamma_1(x\circ a) = \gamma_2(x\circ a)$, $\gamma_1R_a x = \gamma_2 R_a x$, $\gamma_1 = \gamma_2$ and by Lemma  \ref{TWO_AUTOTOPY_COMPONENTS} $\beta_1 = \beta_2$.

Case 2. $i = 1, j = 3$. Suppose there exist  autotopies $(\alpha, \beta_1, \gamma_1)$ and $(\alpha, \beta_2, \gamma_2)$ such that $\gamma_1 a = \gamma_2 a = b$ for some fixed element $a\in Q$. Since $(Q, \circ)$ is a quasigroup,  then for any element $x\in Q$ there exists a unique element  $x^{\prime} \in Q$ such that $x \circ x^{\prime} = a$. Using the concept of middle quasigroup translation  we can re-write the last equality in the form $P_a x = x^{\prime}$ and say that $P_a$ is a permutation of the set $Q$.

For all pairs $x, x^{\prime}$   we have $\alpha x \circ \beta_1 x^{\prime} = \gamma_1(x\circ x^{\prime}) = b$ and
$\alpha x \circ \beta_2 x^{\prime} = \gamma_2(x\circ x^{\prime}) = b$. Since the right  sides of the last equalities are equal  we have
$\alpha x \circ \beta_1 x^{\prime} = \alpha x \circ \beta_2 x^{\prime}$, $\beta_1 x^{\prime} = \beta_2 x^{\prime}$ for all $x^{\prime} \in Q$.  The variable $x^{\prime}$ takes all values from the set $Q$ since $P_a$ is a permutation of the set $Q$. Therefore $\beta_1 = \beta_2$ and by Lemma  \ref{TWO_AUTOTOPY_COMPONENTS} $\gamma_1 = \gamma_2$.

All other cases are proved in the similar way with Cases 1 and  2.
\end{proof}

\begin{lemma} \label{FORM_OF_AUTOTOPY}
Any autotopy $(\alpha, \beta, \gamma)$ of a loop $(Q, \cdot)$ has the form  $(R^{-1}_{\beta 1} \gamma, L^{-1}_{\alpha 1} \gamma, \gamma)$, where $\alpha 1 \cdot \beta 1 = 1$.
\end{lemma}
\begin{proof}
If we put in equality $\alpha x \cdot \beta y = \gamma(xy)$ $x=y=1$, then $\alpha 1 \cdot \beta 1 = 1$.
If $x=1$, then $\alpha 1 \cdot \beta y = \gamma y$, $\beta = L^{-1}_{\alpha 1} \gamma$.
If $y=1$, then $\alpha x \cdot \beta 1 = \gamma x$, $\alpha  = R^{-1}_{\beta 1} \gamma$.
\end{proof}

\begin{theorem}\label{AUTOTOPY_ORDER}
The order of autotopy group of a finite quasigroup $Q$ of order $n$ is a divisor of the number $n! \cdot n$.
\end{theorem}
\begin{proof}
The proof follows from Lemma \ref{ISOM_AUTOT_GR} (we can prove loop case),  Lemma \ref{TWO_AUTOTOPY_COMPONENTS} and Lemma \ref{FORM_OF_AUTOTOPY} (we can take the second and third components of loop autotopy).
\end{proof}

Also by the proving of Theorem \ref{AUTOTOPY_ORDER} it is possible to use Leakh Theorem (Theorem \ref{DEF_OF_COMP_BY_AUTOT_AND_EL}).

\begin{example}
The order of autotopy group of the group $Z_2 \times Z_2$  is equal to $4\cdot 4\cdot 6 = 4!\cdot 4$, i.e. in this case autotopy group is equal to the upper bound.
\end{example}

\begin{remark}
There exist quasigroups (loops) with identity autotopy group \cite{DER_DER_DUD}. In this case autotopy group is of minimal order.
\end{remark}

\subsection{Isostrophism}

 Isostrophy (synonymous with isostrophism) of a quasigroup  is a transformation  that is a combination
 of parastrophy and isotopy, i.e. isostrophe image of a quasigroup $(Q,A)$ is parastrophe image of its isotopic
image or, vice versa, isostrophe image of a quasigroup $(Q,A)$ is isotopic image of its parastrophe (Definition
\ref{D2.3}). Therefore, there exists a possibility to define isostrophy by at least in two ways.

If $T = (\alpha_1, \alpha_2, \alpha_3)$ is an isotopy, $\sigma$ is a parastrophy of a quasigroup $(Q, A)$, then
we shall denote by $T^{\sigma}$ the triple $(\alpha_{\sigma^{-1} 1}, \alpha_{\sigma^{-1} 2}, \alpha_{\sigma^{-1}
3})$.

In order to give some properties of quasigroup isostrophisms  we need the following

\begin{lemma} \label{L2.1} ${(AT)}^\sigma  = {A}^\sigma T^\sigma,$ $(T_1T_2)^{\sigma} = T_1^{\sigma}
T_2^{\sigma}$ \cite{VD, 2}.
\end{lemma}

\begin{remark}
It is possible to define action of a parastrophy $\sigma$  on an isotopy  $T = (\alpha_1, \alpha_2, \alpha_3)$ also in the following (more standard) way: $T^{\sigma} = (\alpha_{\sigma 1}, \alpha_{\sigma 2}, \alpha_{\sigma
3})$. In this case Lemma \ref{L2.1} also is valid.
\end{remark}

Recall, if   $(Q, A)$ is a binary  groupoid, then  $x_3 = A(x_1, x_2)$ (Definition \ref{def1}).

\begin{definition}\label{D2.3}
A quasigroup $(Q,B)$ is an isostrophic image of a quasigroup $(Q,A)$ if there exists a collection of
permutations $(\sigma, (\alpha_1, \alpha_2, \alpha_3)) = (\sigma, T)$, where $\sigma \in S_3 $, $T = (\alpha_1,
\alpha_2, \alpha_3)$ and    $\alpha_1, \alpha_2, \alpha_3$ are permutations of the set $Q$ such that  $  B(x_1,
x_2) = A(x_1, x_2)(\sigma, T) = (A^{\sigma}(x_1, x_2)) T = \alpha^{-1} _3{A}(\alpha _1 x_{\sigma^{-1} 1}, \alpha _2x_{\sigma^{-1} 2}) $ for
all $x_{1}, x_{2}\in Q$ \cite{SCERB_08_1, SCERB_09_JGLTA}.
\end{definition}

\begin{remark}
There exist  formally other but in some sense equivalent definitions of isostrophy and autostrophy. See, for example,  \cite{ks3, KEED_SCERB}.
\end{remark}

A collection of permutations $(\sigma, (\alpha_1, \alpha_2, \alpha_3)) =  (\sigma, T)$ will be called an {\it
isostrophy} of a quasigroup $(Q,A)$.

Often  an isostrophy $(\sigma, T)$ is called $\sigma$-isostrophy T or isostrophy of type $\sigma$. We can
re-write the equality from Definition \ref{D2.3} in the form ${A}^{\sigma} T = B$, where $T = (\alpha_1,
\alpha_2, \alpha_3)$. It is clear that  $\varepsilon$-isostrophy is called an \textit{isotopy}.

Probably R. Artzy was the first who has given algebraic definition of isostrophy \cite{RA_63}. For  $n$-ary quasigroups concept
of isostrophy is studied in \cite{2}. Isostrophy has clear geometrical (net) motivation \cite{HOP, BS_83}. See
also \cite{IVL}.

Let $(\sigma, T)$ and $(\tau, S)$ be some isostrophisms, $T = (\alpha_1, \alpha_2, \alpha_3)$, $S = (\beta_1, \beta_2, \beta_3)$.  Then
\begin{equation}
(\sigma, T) (\tau, S) = (\sigma\tau,T^{\tau}S)
\end{equation}

We can write the equality   $B=A(\sigma\tau,T^{\tau}S)$ in more details: (the right record of maps)
\begin{equation}\label{composition_of_isostrophisms}
\begin{split}
B(x_1, x_2, x_3) =
 A(x_{1 (\tau^{-1}\sigma^{-1})} \alpha_{1\tau^{-1}}\beta_1, x_{2 (\tau^{-1}\sigma^{-1})}\alpha_{2 \tau^{-1}}\beta_2, x_{3 ( \tau^{-1}\sigma^{-1}) }\alpha_{3 \tau^{-1}}\beta_3).
\end{split}
\end{equation}
If we shall use the left record of maps, then
  \begin{equation} \label{composition_of_isostrophisms_LEFT_REC}
 B(x_1, x_2, x_3) = A(\beta_1 \alpha_{\tau^{-1}1} x_{\sigma^{-1} (\tau^{-1} 1)},\beta_2 \alpha_{\tau^{-1}2} x_{\sigma^{-1} (\tau^{-1} 2)}, \beta_3 \alpha_{\tau^{-1}3}x_{\sigma^{-1} (\tau^{-1} 3)}).
 \end{equation}

\begin{lemma} \label{inver_AUTOSTROPHY}
If $(\sigma, T) = (\sigma, (\alpha_1, \alpha_2, \alpha_3))$ is an isostrophy, then
\begin{equation}\label{INVERSE_ISOSTROPHISM}
(\sigma, T)^{-1} = (\sigma^{-1},(T^{-1})^{\sigma^{-1}}) = (\sigma^{-1}, (\alpha^{-1}_{\sigma 1}, \alpha^{-1}_{\sigma 2}, \alpha^{-1}_{\sigma 3}))
\end{equation}
\end{lemma}
\begin{proof}
Indeed,
$(\sigma, T)^{-1}  = (\sigma^{-1},(T^{-1})^{\sigma^{-1}})$,    $ (T^{-1})^{\sigma^{-1}} = (\alpha^{-1}_1, \alpha^{-1}_2, \alpha^{-1}_3)^{\sigma^{-1}} = (\alpha^{-1}_{\sigma 1}, \alpha^{-1}_{\sigma 2}, \alpha^{-1}_{\sigma 3})$.
 Then
\begin{equation*}
\begin{split}
(\sigma, T)(\sigma^{-1},(T^{-1})^{\sigma^{-1}}) = (\sigma \sigma^{-1}, T^{\sigma^{-1}} (T^{-1})^{\sigma^{-1}}) = (\varepsilon, \varepsilon), &\\
(\sigma^{-1},(T^{-1})^{\sigma^{-1}}) (\sigma, T) = (\sigma^{-1} \sigma, T^{-1} T) = (\varepsilon, \varepsilon).
\end{split}
\end{equation*}
\end{proof}

If $B=A$, then  isostrophism is called an \textit{autostrophism (autostrophy)}. Denote by $Aus\,(Q,A)$ the group of all autostrophisms of a quasigroup $(Q, A)$.

We can generalize Lemma \ref{ISOM_AUTOT_GR} and similarly prove the following

\begin{lemma} \label{ISOSTR_AUTOT_GR1} If quasigroups $(Q, \circ)$ and $(Q, \cdot)$ are isostrophic
with an isostrophy $T$, i.e. $(Q,\circ) = (Q, \cdot)T$, then $ Aus(Q, \circ) = T^{-1} Aus  (Q, \cdot) T$.
\end{lemma}
\begin{proof}
The proof repeats the proof of Lemma 1.4 from \cite{1a}.
Let $S\in Aus(Q, \circ)$. Then $(Q, \cdot)T = (Q, \cdot)TS$, $(Q, \cdot) = (Q, \cdot)TST^{-1}$,
\begin{equation} \label{AUTOTOP_EQU}
T Aus(Q, \circ) T^{-1} \subseteq Aus (Q, \cdot).
\end{equation}
 If $(Q,\circ) = (Q, \cdot)T$, then $(Q,\circ) T^{-1} = (Q, \cdot)$, expression (\ref{AUTOTOP_EQU}) takes the form $T^{-1} Aus(Q, \cdot) T \subseteq Aus (Q, \circ)$,
  \begin{equation} \label{AUTOTOP_EQU_1}
T Aus(Q, \circ) T^{-1} \supseteq Aus (Q, \cdot).
\end{equation}
Comparing (\ref{AUTOTOP_EQU}) and (\ref{AUTOTOP_EQU_1}) we obtain $ Aus(Q, \circ) = T^{-1} Aus  (Q, \cdot) T$.
\end{proof}

\begin{corollary} \label{ISOSTR_AUTOT_GR} If quasigroups $(Q, \circ)$ and $(Q, \cdot)$ are isostrophic
with isostrophy $T$, i.e. $(Q,\circ) = (Q, \cdot)T$, then $ Avt(Q, \circ) = T^{-1} Avt (Q, \cdot) T$.
\end{corollary}
\begin{proof}
The proof repeats the proof of Lemma \ref{ISOSTR_AUTOT_GR1}. We only notice, if $S \in Avt(Q, \circ)$, then $T^{-1} S T \in Avt(Q, \cdot)$.
\end{proof}

\begin{corollary} \label{AUTOSTR_AUTOT_GR} If $T$ is an autostrophy of a quasigroup  $(Q, \cdot)$, i.e. $(Q,\cdot) = (Q, \cdot)T$, then $ Avt(Q, \cdot) = T^{-1} Avt (Q, \cdot) T$.
\end{corollary}
\begin{proof}
The proof follows from Lemma \ref{ISOSTR_AUTOT_GR1}.
\end{proof}

\subsection{Group action}

We shall denote by $S_Q$ the symmetric group of all bijections (permutations) of the set $Q$.

We recall some definitions from \cite{FR, KM}.
\begin{definition}
A group $G$ acts on a set $M$ if for any pair of elements $(g,m)$, $g\in G, m\in M,$ an element $(gm)\in M$ is
defined. Moreover, $g_1(g_2 (m)) = (g_1g_2)m $ and $em=m$ for all $m\in M$, $g_1, g_2\in G$. Here $e$ is the
identity element of the group $G$.
\end{definition}

The set $Gm=\{gm \, | \, g\in G\}$ is called an orbit of element $m$.
For every $m$ in $M$, we define the stabilizer subgroup of $m$  as the set of all elements in $G$ that fix $m$:
$G_m=\{g \, | \, gm = m \} $

The orbits of any two elements of the set
$M$ coincide or are not intersected. Then the set $M$ is divided into a set of non-intersected orbits. In other
words, if we define on the set $M$ a binary relation $\sim$ as:
$$m_1 \sim m_2 \textrm{\ if \ and \ only \ if \
there \ exists \ } g\in G \textrm{\ such \ that \ }m_2 = g m_1,$$ then $\sim$ is an equivalence relation on the
set $M$.

Every orbit is an invariant subset of $M$ on which $G$ acts transitively. The action of $G$ on $M$ is transitive if and only if all elements are equivalent, meaning that there is only one orbit.

A partition $\theta$ of the set $M$ on disjoint subsets $\theta(x)$, $x\in M$ is called a partition on blocks relatively the group
$G$, if for any $\theta(a)$ and any $g\in G$ there exists a subset $\theta(b)$ such that $g \theta(a)  = \theta(b)$.
It is obviously that there exist trivial partition of the set $M$, namely, partition into one-element blocks
and partition into unique block.

If there does not exist a partition of the set $M$ into non-trivial blocks, then the group $G$ is called
primitive.

\begin{definition} \label{action}
The action of $G$ on $M$ is called:
\begin{enumerate}
\item faithful (or effective) if for any two distinct $g,  h \in G$ there exists an $x \in M$ such that $g(x) \neq  h(x)$; or equivalently, if for any $g\neq e \in G$ there exists an $x \in M$ such that $g(x) \neq x$. Intuitively, different elements of $G$ induce different permutations of $M$;

\item free (or semiregular) if for any two distinct $g, h \in G$ and all $x \in M$ we have $g(x) \neq h(x)$; or equivalently, if $g(x) = x$ for some $x$ then $g = e$;

\item regular (or simply transitive) if it is both transitive and free; this is equivalent to saying that for any two $x, y$ in $M$ there exists precisely one $g$ in $G$ such that $g(x) = y$. In this case, $M$ is known as a principal homogeneous space for $G$ or as a $G$-torsor  \cite{WIKI_3}.
\end{enumerate}
\end{definition}

Denote  the property of a set of permutations  $\{p_1, p_2, \dots , p_m\}$ of an $m$-element set $Q$ \lq\lq $p_i
p^{-1}_j$ ($i\neq j$) leaves no variable unchanged\rq\rq \cite{MANN} as the $\tau$-property.  An  $m$-tuple of
permutations $T$ can also have the $\tau$-property. We shall call the $m$-tuple $T$ as a $\tau$-m-tuple.

In \cite{MANN}, in fact, Mann proves the following
\begin{theorem}\label{THEOR_1}
A set $P = \{p_1, p_2, \dots , p_m\}$ of $m$ permutations of a finite set $Q$ of order $m$  defines Cayley table of a quasigroup if and only if $P$ has the
$\tau$-property.
\end{theorem}

A permutation $\alpha$ of a finite non-empty set $Q$ which  leaves no elements of the set $Q$ unchanged will be
called a \textit{fixed point free permutation}.

\begin{definition} \label{ACTION_OF_SET_OF_PERM}
A set $M$  of  maps
of the set $Q$ into itself is called {\it simply transitive} (more precise, the set $M$ acts on the set $Q$ simply
transitively) if for any pair of elements $x, y$  of the set $Q$ there exists a unique element  $\mu_j$ of the set
$M$ such that $\mu_j(x) = y$.
\end{definition}

In Definition \ref{ACTION_OF_SET_OF_PERM} we do not suppose  that the set $M$ is a group. Notice  that concepts from Definition \ref{action} are suitable for the set  $M$.

\begin{theorem} \label{TERMANN}
A set $T = \{p_1, p_2, \dots , p_n\}$ of $n$ permutations of a finite set $Q$ of order $n$ is simply
transitive if and only if the set $T$ has the $\tau$-property \cite{MANN}.
\end{theorem}

\begin{lemma}
If $(Q, \cdot)$ is a quasigroup, then the sets $\mathbb{L,R, P}$ of all left, right, middle translations of quasigroup $(Q, \cdot)$ are simply transitive sets of permutations of the set $Q$.
\end{lemma}
\begin{proof}
 Indeed, if $a, b$ are fixed elements of the set $Q$, then there exists an unique element $x\in Q$ such that $x\cdot a = b$,    there exists an unique element $y\in Q$ such that $a \cdot y = b$, there exists an unique element $z\in Q$ such that $a\cdot b = z$. We establish the following bijections between elements of the sets $Q$ and  elements of the sets $\mathbb{L}(Q, \cdot), \mathbb{R}(Q, \cdot), \mathbb{P}(Q, \cdot)$: $\varphi_1: x \leftrightarrow L_x $, $\varphi_2: x \leftrightarrow R_x $, $\varphi_3: x \leftrightarrow P_x $. Therefore  there exists an unique element $L_x\in {\cal L}(Q, \cdot)$ such that $L_x a = b$,  there exists an unique element $R_y\in {\cal R}(Q, \cdot)$ such that $R_y a = b$, there exists an unique element $P_z\in {\cal P}(Q, \cdot)$ such that $P_z a = b$.
\end{proof}

\begin{definition} \label{CENTRALIZER_DEF}
The centralizer of a set   $S$ of a group $G$ (written as $C_G(S)$) is the set of elements of $G$ which commute with any element of $S$; in other words, $C_G(S) = \{x \in  G \, \mid \,  xa = ax \textrm{  for all } a\in S \} $ \cite{KM, HALL}.
\end{definition}

\begin{definition}
The normalizer of a set $S$ in a group $G$, written as $N_G(S)$,   is defined as $N_G(S) = \{x \in  G \, \mid \, xS = Sx\}$.
\end{definition}

The following theorem is called NC-theorem \cite{WIKI_5}.
\begin{theorem}
$C(S)$ is always a normal subgroup of $N(S)$ \cite{DRAP_JEDL, WIKI_5}.
\end{theorem}
\begin{proof}
The proof is taken from \cite{WIKI_5}.
If $c$ is in $C(S)$ and $n$ is in $N(S)$, we have to show that $n^{-1}cn $ is in $C(S)$.
To that end, pick $s$ in $S$ and let $t = nsn^{-1}$.
Then $t$ is in $S$, so therefore $ct = tc$. Then note that $ns = tn$; and $n^{-1}t = sn^{-1}$. So
$(n^{-1}cn)s = (n^{-1}c)tn = n^{-1}(tc)n = (sn^{-1})cn = s(n^{-1}cn)$
which is what we needed.
\end{proof}

\section{Garrison's nuclei and  A-nuclei}

\subsection{Definitions of  nuclei and  A-nuclei}

We recall standard Garrison's \cite{GARRISON} definition of quasigroup nuclei.
\begin{definition}\label{GARRISON_NUCL_DEF}
Let $(Q,\circ)$ be a quasigroup. Then $N_l = \{a\in Q \, \mid \, (a\circ x)\circ y = a\circ (x\circ y)\}$,  $N_r = \{a\in Q \, \mid \, (x\circ y)\circ a = x\circ (y\circ a)\}$
 and $N_m = \{a\in Q \, \mid \, (x\circ a)\circ y = x\circ (a\circ y)\}$
are  respectively its  left, right and middle nuclei \cite{GARRISON, VD, HOP}.
\end{definition}
\begin{remark}
Garrison names an element  of a middle quasigroup  nucleus as a  \textit{center element} \cite{GARRISON}.
\end{remark}

\begin{definition} \label{BRUCK_NUCLEUS}  Let $(Q,\cdot)$ be a quasigroup.  Nucleus is given by $N= N_l\cap N_r\cap N_m$ \cite{HOP}.
\end{definition}

The importance of Garrison's quasigroup nuclei is in the fact that $N_l, N_r$ and $N_m$ all are subgroups
of a quasigroup $(Q, \cdot)$ \cite{GARRISON}. %%(Theorem \ref{middle_nuclei}).

The weakness of Garrrison's definition is in the fact that, if a quasigroup $(Q,\cdot)$  has  a non-trivial left nucleus, then $(Q,\cdot)$ is a left loop, i.e. $(Q,\cdot)$ has a left identity element; if a quasigroup $(Q,\cdot)$  has  a non-trivial right nucleus, then $(Q,\cdot)$ is a right loop, i.e. $(Q,\cdot)$ has a right identity element;
if a quasigroup $(Q,\cdot)$  has  a non-trivial middle nucleus, then $(Q,\cdot)$ is a loop, i.e. $(Q,\cdot)$ has an  identity element (\cite{GARRISON}; \cite{HOP}, I.3.4. Theorem).

It is well known connection between autotopies and nuclei \cite{vdb0, vdb_63, VD, kepka71}.

 Namely  the set of autotopies of the form $(L_a, \varepsilon, L_a)$ of a quasigroup $(Q,\circ)$  corresponds to  left nucleus of $(Q,\circ)$ and vice versa.

 Similarly, the set of autotopies of the form $(\varepsilon, R_a,  R_a)$ of a quasigroup $(Q,\circ)$  corresponds to  right  nucleus of $(Q,\circ)$, the set of autotopies of the form $(R_a, L^{-1}_a, \varepsilon)$ of a quasigroup $(Q,\circ)$  corresponds to middle  nucleus of $(Q,\circ)$.

It is easy to see that from Garrison definition of left nucleus of a loop $(Q, \cdot)$ it follows that $R^{-1}_{xy}R_yR_x a = a$ for all $x, y \in Q$ and all $a\in N_l$. Permutations of the form $R^{-1}_{xy}R_yR_x$ generate right multiplication group of $(Q, \cdot)$. It is clear that any element of left nucleus is invariant relatively to any element of the group $\left< R^{-1}_{xy}R_yR_x \, | \, x, y \in Q \right>$.

Similarly  from Garrison definition of right  nucleus of a loop $(Q, \cdot)$ it follows that $L^{-1}_{xy}L_xL_y a = a$ for all $x, y \in Q$ and all $a\in N_l$. Permutations of the form $L^{-1}_{xy}L_xL_y$ generate right multiplication group of $(Q, \cdot)$. It is clear that any element of left nucleus is invariant relatively to any element of the group $\left< L^{-1}_{xy} L_x L_y \, | \, x, y \in Q \right>$.

For middle nucleus situation is slightly other and of course Garrison was right when called elements of middle loop nucleus as central elements.

P.I.~Gramma \cite{GRAMMA_PI}, M.D.~Kitaroag\u a \cite{MDK},  and G.B.~Belyavskaya \cite{gbb1, GBB2, Bel_94, gbb} generalized on quasigroup case concepts of nuclei and center using namely this nuclear property.  G.B.~Be\-lyav\-s\-kaya obtained in this direction the most general results.

In \cite{KEED_SCERB} the following definition is given.

\begin{definition} \label{A_NUCLEI} The set of all autotopisms of the form $(\alpha , \varepsilon , \gamma )$
 of a quasigroup $(Q,\circ )$, where $\varepsilon $ is the identity mapping, is called the {\it left
  autotopy nucleus} (left $A$-nucleus) of  quasigroup $(Q,\circ)$.

  Similarly, the sets  of
  autotopisms of the forms $(\alpha ,\beta
  ,\varepsilon )$ and $(\varepsilon ,\beta ,\gamma )$ form the \textit{middle and right $A$-nuclei} of $(Q,\circ)$.
We shall denote these three sets  of mappings by $N^A_l$, $N^A_m$ and $N^A_r$ respectively.
\end{definition}

Using Definition \ref{A_NUCLEI} we can say that to the elements of left Kitaroag\u a nucleus of a quasigroup $(Q, \cdot)$ correspond to autotopies of the form $(L_aL^{-1}_h, \varepsilon, L_aL^{-1}_h)$, where $a\in Q$  (in fact $a\in N_l$ in Kitaroag\u a sense), $h$ is a fixed element of the set $Q$ \cite{MDK}.

  Autotopies of the form $(L_aL^{-1}_h, \varepsilon, L_{a\cdot e(h)} L^{-1}_h)$, where $a\in Q$ (in fact $a\in N_l$ in Belyavskaya  sense) correspond to the elements of left Belyavskaya nucleus \cite{gbb1}.

\begin{remark}
Often, by a generalization of some objects or concepts, we not only win in generality, but also lose some important properties of generalized objects or concepts.

The weakness of Definition \ref{A_NUCLEI} is in the fact that it is not easy  to define A-nucleus similarly to quasigroup nucleus as an intersection of left, right and middle  quasigroup A-nuclei.
\end{remark}

\begin{lemma} \cite{KEED_SCERB}.\label{COMPONENTS}
\begin{enumerate}
\item The first components of the autotopisms of any subgroup $K(N^A_l)$ of $N^A_l$ themselves
 form a group $K_1(N^A_l)$.
\item The third  components of the autotopisms of any subgroup $K(N^A_l)$ of $N^A_l$ themselves
 form a group $K_3(N^A_l)$.
\item The first components of the autotopisms of any subgroup $K(N^A_m)$ of $N^A_m$ themselves
 form a group $K_1(N^A_m)$.
\item The second  components of the autotopisms of any subgroup $K(N^A_m)$ of $N^A_m$ themselves
 form a group $K_2(N^A_l)$.
\item The second  components of the autotopisms of any subgroup $K(N^A_r)$ of $N^A_r$ themselves
 form a group $K_2(N^A_r)$.
\item The third components of the autotopisms of any subgroup $K(N^A_r)$ of $N^A_r$ themselves
 form a group $K_3(N^A_r)$.
\end{enumerate}
 \end{lemma}
\begin{proof} Case 1.  Let $(\alpha _1, \varepsilon , \gamma _1), (\alpha _2, \varepsilon , \gamma _2)\in K(N^A_l)$. Then $\alpha _1 x \circ y =
\gamma _1 (x\circ y)$ and $\alpha _2 x \circ y =  \gamma _2 (x\circ y)$ for all $x, y \in Q$. From the first    equation, $\alpha _1(\alpha _2 x) \circ y = \gamma _1 (\alpha _2 x \circ y) = \gamma _1\gamma _2(x\circ y)$ by virtue of the second  equation.
 Thus, $(\alpha _1\alpha _2, \varepsilon , \gamma _1\gamma _2)\in K(N^A_l)$.

 Let
 $x = \alpha ^{-1}_1 u$. Then, $\alpha _1 x \circ y =
\gamma _1 (x\circ y) \Rightarrow u\circ y = \gamma _1 (\alpha ^{-1}_1 u \circ y) \Rightarrow \alpha ^{-1}_1 u \circ y =
\gamma ^{-1}_1 (u\circ y)$ so $(\alpha ^{-1}_1, \varepsilon , \gamma ^{-1}_1)\in K(N^A_l)$.

 Clearly, $(\varepsilon , \varepsilon , \varepsilon )\in K(N^A_l)$. Hence, $\varepsilon $ is a first component and, if $\alpha _1$
 and $\alpha _2$ are first components so are $\alpha _1\alpha _2$ and $\alpha ^{-1}_1$.

Other cases are proved in similar way with Case 1.
\end{proof}

 From Lemma \ref{COMPONENTS} it follows that the first, second and third components of $N^A_l$, $N^A_m$ and $N^A_r$
each form groups. For brevity, we shall denote these nine groups by $_1N^A_l$, $_2N^A_l$, $_3N^A_l$, $_1N^A_m$,
$_2N^A_m$, $_3N^A_m$, $_1N^A_r$, $_2N^A_r$ and $_3N^A_r$.

Next two lemmas demonstrate  that A-nuclei have some advantages in comparison with Garrison's nuclei.

\begin{lemma}\label{NORMALITY_OF_A_NUCL}
In any quasigroup $Q$ its left, right, middle A-nucleus is normal subgroup of the group Avt(Q).
\end{lemma}
\begin{proof}
Let $(\mu, \varepsilon, \nu) \in N^A_l$, $(\alpha, \beta, \gamma) \in Avt(Q)$. Then
\begin{equation*}
\begin{split}
(\alpha^{-1}, \beta^{-1}, \gamma^{-1})(\mu, \varepsilon, \nu)(\alpha, \beta, \gamma) = (\alpha^{-1}\mu
\alpha,\varepsilon ,\gamma^{-1}\nu \gamma)\in N^A_l
\end{split}
\end{equation*}
The  proofs for right and middle nucleus are similar.
\end{proof}

\begin{lemma} \label{ISOTOPIC_COMPONENTS}
If  quasigroups $(Q, \cdot)$ and $(Q, \circ)$ are isotopic, then these quasigroups   have isomorphic autotopy nuclei and isomorphic components of the autotopy nuclei, i.e. $N_l^A(Q, \cdot) \cong N_l^A(Q, \circ)$, ${}_1N_l^A(Q, \cdot) \cong {}_1N_l^A(Q, \circ)$, ${}_3N_l^A(Q, \cdot) \cong {}_3N_l^A(Q, \circ)$, and so on.
\end{lemma}
\begin{proof}
We can use Lemmas \ref{ISOM_AUTOT_GR} and \ref{COMPONENTS}.
\end{proof}

\begin{corollary} If a quasigroup $(Q, \cdot)$ is an  isotope of a group $(Q, +)$, then all its A-nuclei and components of A-nuclei are isomorphic with the group $(Q, +)$.
\end{corollary}
\begin{proof}
It is clear that in any group all its A-nuclei and components of A-nuclei are isomorphic with the group $(Q, +)$.
Further we can apply Lemma \ref{ISOTOPIC_COMPONENTS}.
\end{proof}

\begin{lemma} \label{ISOTOPIC_COMPONENTS_PERM}
Isomorphic components of A-nuclei of a quasigroup $(Q, \cdot)$ act on the set $Q$ in such manner that the numbers and lengths of orbits by these actions are equal.
\end{lemma}
\begin{proof}
In Lemma \ref{ISOTOPIC_COMPONENTS} it is established  that A-nuclei and the same components of A-nuclei are isomorphic in pairs. Notice all A-nuclei of a  quasigroup $(Q, \cdot)$ are subgroups of  the group $S_Q\times S_Q\times S_Q$. Components of these A-nuclei are subgroups of the group $S_Q$.

If two components of A-nuclei, say $B$ and $C$ are  isomorphic, then this isomorphism is an isomorphism of permutation groups which  act on the set $Q$. Isomorphic permutation groups are called \textit{similar} \cite[p. 111]{KM}.

Let $\psi B = C$, where $B, C$ are  isomorphic components of A-nuclei, $\psi$ is an isomorphism. Then  we can establish  a bijection $\varphi$ of the set $Q$ such that $\psi(g) (\varphi (m)) = \varphi (g(m))$ for all $m\in Q, g\in B$ \cite[p. 111]{KM}.
\end{proof}

\subsection{A-nuclei and isosotrophy}

In this section we find   connections between components of A-nuclei of a quasigroup $(Q,\cdot)$  and its isostrophic images of the form $(Q, \circ) = (Q,\cdot) ((\sigma)(\alpha, \beta, \gamma))$, where $\sigma \in S_3$, $\alpha, \beta, \gamma \in S_Q$.

We omit the symbol of autotopy nuclei (the symbol $A$) and put on its place the symbols of binary operations "$\circ$" and "$\cdot$" respectively. Denote isostrophy  $(\sigma, (\alpha, \beta, \gamma))$ by $(\sigma, T)$ for all cases.

\begin{lemma} \label{ISOSTR_AUTOT_GR_1vsp} If quasigroup $(Q, \circ)$ is  isostrophic image of quasigroup $(Q, \cdot)$
with an isostrophy $S = ((12), T)$, i.e. $(Q, \circ) = (Q, \cdot)S$,   then $ N_l(Q, \circ) = S^{-1} N_r(Q, \cdot)S$, $_1N_l(Q, \circ) = \alpha^{-1}\,_2N_r(Q, \cdot) \alpha$, $\,{}_3N_l(Q, \circ) = \gamma^{-1} \, {}_3N_r(Q, \cdot)\gamma$.
\end{lemma}
\begin{proof}
The proof repeats the proof of Lemma \ref{ISOSTR_AUTOT_GR1}.
Let $K\in N_l(Q, \circ)$. Then $(Q, \cdot)S = (Q, \cdot)SK$, $(Q, \cdot) = (Q, \cdot)SKS^{-1}$,
\begin{equation} \label{EQUT_NUCLEI}
S N_l(Q, \circ) S^{-1} \subseteq N_r(Q, \cdot),
\end{equation}
since  \[
((12), (\alpha, \beta, \gamma))\,  (\varepsilon, (\,_1N_l^{\circ}, \varepsilon, \,_3N_l^{\circ}))\: ((12), (\beta^{-1}, \alpha^{-1}, \gamma^{-1})) = (\varepsilon, (\varepsilon, \alpha\,_1N_l^{\circ}\alpha^{-1}, \gamma \,_3N_l^{\circ} \gamma^{-1})).
\]

If $(Q,\circ) = (Q, \cdot)S$, then $(Q,\circ) S^{-1} = (Q, \cdot)$, expression (\ref{EQUT_NUCLEI}) takes the form
\begin{equation} \label{VKLYUCH_YADER}
S^{-1} N_r(Q, \cdot) S \subseteq N_l(Q, \circ).
\end{equation}
Indeed, \[
((12), (\beta^{-1}, \alpha^{-1}, \gamma^{-1}))\,  (\varepsilon, (\varepsilon, \,_2N_r^{\cdot},  \,_3N_r^{\cdot}))\: ((12), (\alpha, \beta, \gamma)) = (\varepsilon, (\alpha^{-1}\,{}_2N_r^{\cdot}\alpha, \; \varepsilon, \; \gamma^{-1} \,_3N_r^{\cdot} \gamma)).
\]

We can rewrite expression (\ref{VKLYUCH_YADER}) in such form
  \begin{equation} \label{AUTOTOP_EQU_2}
N_r(Q, \cdot)  \subseteq SN_l(Q, \circ)S^{-1}.
\end{equation}
Comparing (\ref{EQUT_NUCLEI}) and (\ref{AUTOTOP_EQU_2}) we obtain $N_r(Q, \cdot) = SN_l(Q, \circ)S^{-1}$.
Therefore
  $\alpha\,{}_1N_l^{\circ}\alpha^{-1} =  \, {}_2N^{\cdot}_r$,  $\,{}_1N_l^{\circ} = \alpha^{-1} \, {}_2N^{\cdot}_r\alpha$, $\,{}_3N_l^{\circ} = \gamma^{-1} \, {}_3N^{\cdot}_r\gamma$.
\end{proof}

\smallskip

All other analogs of Lemma \ref{ISOSTR_AUTOT_GR_1vsp} are proved in the similar way.

\smallskip

In Table 3    connections between components of A-nuclei of a quasigroup $(Q,\cdot)$  and its isostrophic images of the form $(Q, \circ) = (Q,\cdot) ((\sigma)(\alpha, \beta, \gamma))$, where $\sigma \in S_3$, $\alpha, \beta, \gamma \in S_Q$  are collected.

%\rotatebox{90}{
{\hfill Table 3 \label{Table_3}}
\[
\large{
\begin{array}{|c||c| c| c| c| c| c|}
\hline
  & (\varepsilon, T)  & ((12), T) & ((13), T) & ((23), T) & ((132), T) & ((123), T) \\
\hline\hline
_1N_l^{\circ}  & \alpha^{-1}{}_1N_l^{\cdot} \alpha & \alpha^{-1}{}_2N_r^{\cdot} \alpha & \alpha^{-1}{} _3N_l^{\cdot} \alpha & \alpha^{-1}{}_1N_m^{\cdot} \alpha & \alpha^{-1}{}_2N_m^{\cdot} \alpha & \alpha^{-1}{}_3N_r^{\cdot} \alpha\\
\hline
_3N_l^{\circ}  & \gamma^{-1}{}_3N_l^{\cdot}\gamma & \gamma^{-1}{}_3N_r^{\cdot}\gamma & \gamma^{-1}{}_1N_l^{\cdot} \gamma & \gamma^{-1}{} _2N_m^{\cdot} \gamma &   \gamma^{-1}{} _1N_m^{\cdot} \gamma & \gamma^{-1}{}_2N_r^{\cdot} \gamma \\
\hline
_2N_r^{\circ}  & \beta^{-1}{}_2N_r^{\cdot}\beta  & \beta^{-1}{}_1N_l^{\cdot}\beta  &  \beta^{-1}{}_2N_m^{\cdot} \beta & \beta^{-1}{}_3N_r^{\cdot} \beta &  \beta^{-1}{}_3N_l^{\cdot} \beta     & \beta^{-1}{}_1N_m^{\cdot} \beta \\
\hline
_3N_r^{\circ} &  \gamma^{-1}{}_3N_r^{\cdot} \gamma & \gamma^{-1}{}_3N_l^{\cdot} \gamma & \gamma^{-1}{}_1N_m^{\cdot} \gamma & \gamma^{-1}{}_2N_r^{\cdot}\gamma &   \gamma^{-1}{}_1N_l^{\cdot} \gamma  &  \gamma^{-1}{}_2N_m^{\cdot} \gamma\\
\hline
_1N_m^{\circ}  & \alpha^{-1}{}_1N_m^{\cdot} \alpha & \alpha^{-1}{}_2N_m^{\cdot} \alpha &  \alpha^{-1}{}_3N_r^{\cdot} \alpha & \alpha^{-1}{}_1N_l^{\cdot} \alpha &  \alpha^{-1}{}_2N_r^{\cdot}\alpha    & \alpha^{-1}{}_3N_l^{\cdot} \alpha \\
\hline
_2N_m^{\circ}  &  \beta^{-1}{}_2N_m^{\cdot} \beta & \beta^{-1}{}_1N_m^{\cdot} \beta &  \beta^{-1}{}_2N_r^{\cdot} \beta & \beta^{-1}{}_3N_l^{\cdot} \beta &  \beta^{-1}{}_3N_r^{\cdot} \beta     & \beta^{-1}{}_1N_l^{\cdot} \beta \\
\hline
\end{array}}
\]

%%\label{ISOSTR_AUTOT_GR}

If  $(Q, \circ) = (Q,\cdot) ((123)(\alpha, \beta, \gamma))$, then from  Table 3 we have   $_2N_r^A(Q, \circ) = \beta^{-1}{} _1N_m^A \beta (Q,\cdot)$.

\begin{theorem}
Isostrophic quasigroups have isomorphic A-nuclei.
\end{theorem}
\begin{proof}
The proof follows from Table 3.
\end{proof}

\begin{remark}
From  the proof of Lemma \ref{ISOSTR_AUTOT_GR_1vsp} it follows that Table 3 can be used for the finding of connections between concrete components of  A-nuclei. For example, if  $(Q, \circ) = (Q,\cdot) ((23)(\alpha, \beta, \gamma))$, then from  Table 3 we have   $_2N_r^A(Q, \circ) = \beta^{-1}{} _3N_r^A \beta (Q,\cdot)$. Therefore, if $\rho \in {} _3N_r^A (Q,\cdot)$, then there exists an element $\mu \in {}_2N_r^A(Q, \circ)$ such that $\mu = \beta^{-1} \rho \beta$ and vice versa, if $\rho \in {}_2N_r^A(Q, \circ)$, then there exists an element $\mu \in {} _3N_r^A (Q,\cdot)$ such that $\rho = \beta^{-1} \mu \beta$.
\end{remark}

\subsection{Components of A-nuclei and identity elements}

There exist some connection between components of A-nuclei and local identity elements in any quasigroup.

\begin{lemma} \label{QUASIGR_A_NUCL_AND_ID}
\begin{enumerate}
  \item If $(Q,\cdot)$ is a quasigroup and  $(\alpha, \varepsilon, \gamma )\in Avt(Q,\cdot)$, then  $\alpha f(x) \cdot x = \gamma x$, $\alpha x \cdot e(x) = \gamma x$, $\alpha x \cdot x = \gamma s (x)$ for all $x\in Q$.
  \item If $(Q,\cdot)$ is a quasigroup and  $(\varepsilon, \beta,  \gamma )\in Avt(Q,\cdot)$, then  $f(x) \cdot \beta x = \gamma x$, $x \cdot \beta e(x) = \gamma x$, $ x \cdot \beta x = \gamma s (x)$ for all $x\in Q$.
  \item If $(Q,\cdot)$ is a quasigroup and  $(\alpha, \beta,  \varepsilon)\in Avt(Q,\cdot)$, then $\alpha f(x) \cdot \beta x = x$, $\alpha x \cdot \beta e(x) = x$, $\alpha x \cdot \beta x = s (x)$ for all $x\in Q$.
\end{enumerate}
\end{lemma}
\begin{proof}
Case 1. If we put $x=f(y)$, then we obtain $\alpha f(y) \cdot y = \gamma y$.
If we put $y=e(x)$, then we obtain $\alpha (x) \cdot e(x) = \gamma x$. If we put $x=y$, then we obtain $\alpha x  \cdot x = \gamma s(x)$.

Cases 2 and 3 are proved similarly.
\end{proof}

\begin{corollary} \label{ID_QUASIGR_A_NUCL_AND_ID}
\begin{enumerate}
  \item If $(Q,\cdot)$ is an idempotent  quasigroup and  $(\alpha, \varepsilon, \gamma )\in Avt(Q,\cdot)$, then  $\alpha x \cdot x = \gamma x$, $\alpha x= P^{-1}_{\gamma x}$ for all $x\in Q$.
      \item If $(Q,\cdot)$ is an idempotent quasigroup and  $(\varepsilon, \beta,  \gamma )\in Avt(Q,\cdot)$, then  $x \cdot \beta x = \gamma x$, $\beta x = P_{\gamma x} x$ for all $x\in Q$.
  \item  If $(Q,\cdot)$ is an idempotent quasigroup and  $(\alpha, \beta,  \varepsilon)\in Avt(Q,\cdot)$, then $\alpha x \cdot \beta x = x$, $P_{x} \alpha x  =  \beta x$ for all $x\in Q$.
\end{enumerate}
\end{corollary}
\begin{proof}
The proof follows from Lemma \ref{QUASIGR_A_NUCL_AND_ID} and the fact that in an idempotent quasigroup $(Q, \cdot)$ the maps $f, e, s$  all are equal to identity permutation of the set $Q$.
Further we have   $P_{\gamma x}\alpha x  =  x$, $\alpha x= P^{-1}_{\gamma x}$ (Case 1). From  $x \cdot \beta x = \gamma x$ we have $P_{\gamma x} x  =  \beta x$ in Case 2. From  $\alpha x \cdot \beta x = x$ we have $P_{x} \alpha x  =  \beta x$ in Case 3.
\end{proof}

\begin{remark}
 We can rewrite also:

 equality $\alpha x \cdot x = \gamma x$  in the form $x \cdot \alpha^{-1} x = \gamma \alpha^{-1} x$,   $\alpha \gamma^{-1} x \cdot \gamma^{-1} x =  x$ (Case 1);

equality  $x \cdot \beta x = \gamma x$  in the form  $\beta^{-1} x \cdot  x = \gamma \beta^{-1}x$, $\gamma^{-1} x \cdot \beta \gamma^{-1} x = x$ (Case 2);

equality $\alpha x \cdot \beta x = x$  in the form $ x \cdot \beta \alpha^{-1} x = \alpha^{-1} x$, $\alpha \beta^{-1} x \cdot  x = \beta^{-1} x$ (Case 3).
\end{remark}

\begin{corollary} \label{LEFT_RIGHT_MIDDLE_AUTOT}
\begin{enumerate}
\item Let  $(Q,\cdot)$ be a left loop.
\begin{enumerate}
  \item  If  $(\alpha, \varepsilon, \gamma )\in Avt(Q,\cdot)$, then  $\alpha 1 \cdot x = \gamma x$ for all $x\in Q$, i.e. $\gamma = L_{\alpha 1}$.
  \item If   $(\varepsilon, \beta,  \gamma )\in Avt(Q,\cdot)$, then  $\beta  = \gamma $.
  \item If   $(\alpha, \beta,  \varepsilon)\in Avt(Q,\cdot)$, then $\alpha 1 \cdot \beta x = x$ for all $x\in Q$, i.e. $\beta = L^{-1}_{\alpha 1}$.
\end{enumerate}
 \item Let  $(Q,\cdot)$ be a right loop.
\begin{enumerate}
  \item If  $(\alpha, \varepsilon, \gamma )\in Avt(Q,\cdot)$, then   $\alpha   = \gamma$.
  \item If  $(\varepsilon, \beta,  \gamma )\in Avt(Q,\cdot)$, then   $x \cdot \beta 1 = \gamma x$  for all $x\in Q$, i.e. $\gamma = R_{\beta 1}$.
  \item If   $(\alpha, \beta,  \varepsilon)\in Avt(Q,\cdot)$, then  $\alpha x \cdot \beta 1 = x$ for all $x\in Q$, i.e. $\alpha = R^{-1}_{\beta 1}$.
\end{enumerate}
\item Let $(Q,\cdot)$ be  an unipotent  quasigroup.
\begin{enumerate}
  \item If $(\alpha, \varepsilon, \gamma )\in Avt(Q,\cdot)$, then   $\alpha x \cdot x = \gamma 1$ for all $x\in Q$, i.e. $\alpha = P^{-1}_{\gamma 1}$.
  \item If   $(\varepsilon, \beta,  \gamma )\in Avt(Q,\cdot)$, then  $ x \cdot \beta x = \gamma 1$ for all $x\in Q$, i.e. $\beta = P_{\gamma 1}$.
  \item If  $(\alpha, \beta,  \varepsilon)\in Avt(Q,\cdot)$, then  $  \alpha = \beta$.
\end{enumerate}
\end{enumerate}
\end{corollary}
\begin{proof}
The proof follows from Lemma \ref{QUASIGR_A_NUCL_AND_ID} and Remarks \ref{LEFT_RIGHT_LOOP} and
 \ref{IDEMPOT_QUAS}.

 In Case 3(c) we have.  If $(Q,\cdot)$ is an unipotent quasigroup and  $(\alpha, \beta,  \varepsilon)\in Avt(Q,\cdot)$, then  $\alpha x \cdot \beta x = 1$ for all $x\in Q$, i.e. $\beta = P_{1}\alpha $. But in unipotent quasigroup $P_{1} = \varepsilon $.
\end{proof}

\begin{theorem} \label{A-NUCLEI_OF_LOOPS}
\begin{enumerate}
  \item
  Let  $(Q,\cdot)$ be a  loop.
   \begin{enumerate}
\item    If  $(\alpha, \varepsilon, \gamma )\in Avt(Q,\cdot)$, then $\alpha = \gamma = L_{\alpha 1}$.
  \item If  $(\varepsilon, \beta,  \gamma )\in Avt(Q,\cdot)$, then  $\beta  = \gamma = R_{\beta 1} $.
  \item If  $(\alpha, \beta,  \varepsilon)\in Avt(Q,\cdot)$, then  $\alpha = R_{\beta 1}$,   $\beta = L^{-1}_{\beta 1}$.
\end{enumerate}
  \item Let  $(Q,\cdot)$ be an unipotent left loop.
    \begin{enumerate}
\item
    If   $(\alpha, \varepsilon, \gamma )\in Avt(Q,\cdot)$, then    $\alpha = P_{\alpha 1}$, $\gamma = L_{\alpha 1}$.
  \item If   $(\varepsilon, \beta,  \gamma )\in Avt(Q,\cdot)$, then  $\beta = \gamma = P_{\gamma 1}$.
  \item If   $(\alpha, \beta,  \varepsilon)\in Avt(Q,\cdot)$, then  $\alpha = \beta = L^{-1}_{\alpha 1}$.
\end{enumerate}
\item Let  $(Q,\cdot)$ be  an unipotent right loop.
\begin{enumerate}
\item If   $(\alpha, \varepsilon, \gamma )\in Avt(Q,\cdot)$, then    $\alpha = \gamma  = P^{-1}_{\gamma 1}$.
  \item If   $(\varepsilon, \beta,  \gamma )\in Avt(Q,\cdot)$, then   $\beta = P^{-1}_{\beta 1}$, $\gamma = R_{\beta 1}$.
  \item If  $(\alpha, \beta,  \varepsilon)\in Avt(Q,\cdot)$, then  $\alpha = \beta = R^{-1}_{\beta 1}$.
\end{enumerate}
\end{enumerate}
\end{theorem}
\begin{proof}
Case 1(c) is well known. See for example  \cite{KEED_SCERB}.  From Corollary \ref{LEFT_RIGHT_MIDDLE_AUTOT} it follows that in loop case $\alpha = R^{-1}_{\beta 1}$ and  $\beta = L^{-1}_{\alpha 1}$ for any autotopy  of the form $(\alpha, \beta, \varepsilon)$. If $R^{-1}_a\in {}_1N_m^A(Q,\cdot)$, then $R_a\in {}_1N_m^A(Q,\cdot)$, since ${}_1N_m^A(Q,\cdot)$ is a group. Therefore in loop case any autotopy of the kind $(\alpha, \beta, \varepsilon)$ takes the following form
\begin{equation}  \label{Autotopy_form_MIDDLE_NUCLEUS_LOOP}
(R_{\beta 1}, L_{\alpha 1}, \varepsilon)
\end{equation}
Taking into consideration equality  (\ref{Autotopy_form_MIDDLE_NUCLEUS_LOOP}) and the fact that  ${}_1N_m^A(Q,\cdot)$ is a group  we have:  if $R_a, R_b \in {}_1N_m^A(Q,\cdot)$, then $R_aR_b = R_c$. Thus $R_aR_b 1= R_c 1$, $c = b\cdot a$.

   We find now the form of inverse element to the   element  $R_a\in {}_1N_m^A(Q,\cdot)$.
      Notice,  if $R^{-1}_a = R_b$ for some $b\in Q$, then $R_aR_b = R_{b\cdot a}= \varepsilon = R_1$. Then $b\cdot a = 1$ and since any right (left, middle) quasigroup translation  is defined in an unique way by its index element we have $b={}^{-1} a$. From the other side $R_b R_a = R_{a\cdot b}= \varepsilon = R_1$, $b = a^{-1}$.
       Therefore in this situation $a^{-1} = {}^{-1}a$ for any suitable element $a\in Q$, $R^{-1}_a = R_{(^{-1}a)} = R_{a^{-1}}$.

Similarly we can obtain that $L^{-1}_a = L_{(^{-1}a)} = L_{a^{-1}}$ for any $L_a \in {}_2N_m^A(Q,\cdot)$, where $(Q, \cdot)$ is a loop.

In   loop $(Q,\cdot)$ from  equality
$\alpha x \cdot \beta y = x\cdot y$
by $x=y = 1$  we have $\alpha 1 \cdot \beta 1 =  1$. Then $\beta 1  = (\alpha 1)^{-1}$, $\alpha 1 = {}^{-1}(\beta 1)$.
Moreover, we have that $\beta 1  = {}^{-1}(\alpha 1)$, $\alpha 1 = (\beta 1)^{-1}$, $R_{\beta 1} = R^{-1}_{\alpha 1}$.

Finally we obtain that any element of middle loop A-nucleus has the form
\begin{equation}  \label{Autotopy_form_MIDDLE_NUCLEUS_LOOP_1}
(R_{b}, L^{-1}_{b}, \varepsilon).
\end{equation}

It is easy to see that $(R_{\alpha 1}, L^{-1}_{\alpha 1}, \varepsilon )\in N_m^A \Leftrightarrow \alpha 1 \in N_m$.

Case 2(a).   From Corollary \ref{LEFT_RIGHT_MIDDLE_AUTOT} it follows that in unipotent left loop  $\alpha = P^{-1}_{\gamma 1} = P^{-1}_{\alpha 1\cdot 1}$, $\gamma = L_{\alpha 1}$.

In unipotent loop any left A-nucleus autotopy  $(\alpha, \varepsilon, \gamma)$ takes the following form
\begin{equation}  \label{Autotopy_form_LEFT_NUCLEUS_UNI_LOOP}
(P^{-1}_{\gamma 1}, \varepsilon, L_{\alpha 1})
\end{equation}

It is clear, if  $(P^{-1}_{\gamma 1}, \varepsilon, L_{\alpha 1})\in N_l^A$, then any element of the group $N_l^A$ has the form $(P_{\gamma 1}, \varepsilon, L^{-1}_{\alpha 1})$.
   We find now the form of inverse element to the  element  $P_a\in {}_1N_l^A(Q,\cdot)$.

At first we have:  if $P_a, P_b \in {}_1N_l^A(Q,\cdot)$, then $P_aP_b = P_c$ for some element $c \in Q$. Thus $P_aP_b 1= P_c 1$. If  $P_c 1 = d$, then $1\cdot d = c$, $c=d$. Therefore $P_a b = c$, $b\cdot c = a$,  $c = b\backslash a$.

      Notice,  if $P^{-1}_a = P_b$ for some $b\in Q$, then $P_aP_b = P_{b\backslash  a}= \varepsilon = P_1$ since $(Q,\cdot)$ is unipotent quasigroup.  Since any right (left, middle) quasigroup translation  is defined in an unique way by its index element, then $b\backslash  a = 1$,  $a = b\cdot 1$, $b = a\slash 1$. From the other side $P_b P_a = P_{a\backslash b} = \varepsilon = P_1$, $a\backslash b = 1$, $b = a \cdot 1$.
       Thus, if  $(Q, \cdot)$ is a left unipotent loop, then  $a\slash 1  = a \cdot 1$ for any  element $a\in Q$ such that $P_a \in {}_1N_l^A(Q, \cdot)$. Then   $P^{-1}_a = P_{a\slash 1} = P_{a \cdot 1}$, $P_a = P^{-1}_{a \cdot 1}$, $\alpha = P^{-1}_{\alpha 1\cdot 1} = P_{\alpha 1}$.
And  we obtain that any element of  left A-nucleus of an unipotent left loop has the form
\begin{equation}  \label{Autotopy_form_LEFT_NUCLEUS_UNI_LOOP_FIN}
(P_{\alpha 1}, \varepsilon, L_{\alpha 1})
\end{equation}

Case 3(b).   From Corollary \ref{LEFT_RIGHT_MIDDLE_AUTOT} it follows that in unipotent right  loop  $\beta = P_{1\cdot \beta 1}$, $\gamma = R_{\beta 1}$. Further we can do as in Case 2(a) but we shall use  parastrophic ideas.

It is easy to see that any unipotent right  loop $(Q, \circ)$ is $(12)$-parastrophic image  of a unipotent left loop $(Q, \cdot)$ (Table 2). In this case ${}_1N_l^A (Q,\circ) = {}_2N_r^A (Q,\cdot)$, ${}_3N_l^A (Q,\circ) = {}_3N_r^A (Q,\cdot)$ (Table 3).

From formula (\ref{Autotopy_form_LEFT_NUCLEUS_UNI_LOOP_FIN}) it follows that  any element of  unipotent right loop right A-nucleus has the form
\begin{equation}  \label{Autotopy_form_RIGHT_NUCLEUS_RIGHT_UNI_LOOP}
(\varepsilon, P^{\cdot}_{\beta 1},  L^{\cdot}_{\beta 1})
\end{equation}

Finally, the use of Table 1 gives us that in terms of translations of unipotent right loop $(Q, \cdot)$ we have
that  any element of   right A-nucleus of an unipotent right loop $(Q,\circ)$ has the form
\begin{equation}  \label{Autotopy_form_RIGHT_NUCLEUS_RIGHT_UNI_LOOP_FINAL}
(\varepsilon, P^{-1}_{\beta 1},  R_{\beta 1})
\end{equation}
\end{proof}

\begin{remark}
In fact from Theorem \ref{A-NUCLEI_OF_LOOPS} it follows that in loop case Belousov's concept of regular permutations   coincides with the  concepts of left, right, and middle A-nucleus \cite{VD, 1a, HOP}.
\end{remark}

\subsection{A-nuclei of loops by isostrophy}

Taking into consideration the importance of loops  in the class of all quasigroups we give some information on A-nuclear components of isostrophic images of a loop.
From Table 2 it follows the following
\begin{remark} \label{REM_LOOP_PARASTR}
The  $(12)$-parastrophe of a loop $(Q, \cdot)$ is a loop, $(13)$-parastrophe of a loop $(Q, \cdot)$ is an unipotent right  loop, $(23)$-parastrophe of a loop $(Q, \cdot)$ is an unipotent left  loop, $(123)$-parastrophe of a loop $(Q, \cdot)$ is an unipotent left  loop, $(132)$-parastrophe of a loop $(Q, \cdot)$ is an unipotent right  loop.
\end{remark}

Taking into consideration Theorem \ref{A-NUCLEI_OF_LOOPS}, Tables 1 and 3 we can give more detailed  connections between components of A-nuclei of a loop  $(Q,\cdot)$  and its isostrophic images of the form $(Q, \circ) = (Q,\cdot) ((\sigma)(\alpha, \beta, \gamma))$, where $\sigma \in S_3$, $\alpha, \beta, \gamma \in S_Q$.

\begin{example}
   If  $(Q, \circ) = (Q,\cdot) ((123)(\alpha, \beta, \gamma))$, then from  Table 3 we have   $_2N_r^A(Q, \circ) = \beta^{-1}{} _1N_m^A \beta (Q,\cdot)$. By Theorem \ref{A-NUCLEI_OF_LOOPS} any element of the group $_1N_m^A (Q,\cdot)$ is some right translation $R_a$  of the loop $(Q,\cdot)$.

If, additionally,  $\beta = \varepsilon$, then using Table 1 and Remark \ref{REM_LOOP_PARASTR} we can say that in the unipotent left loop $(Q, \circ)$ any element of the group $_2N_r^A(Q, \circ)$ is middle translation $P^{-1}_a$.
\end{example}

In Table 4 $(Q, \circ) = (Q,\cdot) ((\sigma)(\alpha, \beta, \gamma))$, where $(Q, \cdot)$ is a loop. We suppose that elements of left, right and middle  loop nucleus have the following forms $(L_a, \varepsilon, L_a)$, $(\varepsilon, R_b, R_b)$, and $(R_c, L^{-1}_c, \varepsilon)$, respectively.

%\rotatebox{90}{
{\hfill Table 4 \label{Table_4}}
\[
\large{
\begin{array}{|c||c| c| c| c| c| c|}
\hline
  & (\varepsilon, T)  & ((12), T) & ((13), T) & ((23), T) & ((132), T) & ((123), T) \\
\hline\hline
_1N_l^{\circ}  & \alpha^{-1}{}L_a^{\cdot} \alpha & \alpha^{-1}{}R_b^{\cdot} \alpha & \alpha^{-1}L_a^{\cdot} \alpha & \alpha^{-1}R_c^{\cdot} \alpha & \alpha^{-1}(L_c^{\cdot})^{-1} \alpha & \alpha^{-1}{}R_b^{\cdot} \alpha\\
\hline
_3N_l^{\circ}  & \gamma^{-1}L_a^{\cdot}\gamma & \gamma^{-1} R_b^{\cdot}\gamma & \gamma^{-1}{}L_a^{\cdot} \gamma & \gamma^{-1}{} L^{-1}_c{\cdot} \gamma &   \gamma^{-1}{} R_c{\cdot} \gamma & \gamma^{-1}{}R_b{\cdot} \gamma \\
\hline
_2N_r^{\circ}  & \beta^{-1}{}R_b^{\cdot}\beta  & \beta^{-1}{}L_a^{\cdot}\beta  &  \beta^{-1}(L^{-1}_c)^{\cdot} \beta & \beta^{-1}R_b^{\cdot} \beta &  \beta^{-1}L_a^{\cdot} \beta     & \beta^{-1}{}R_c^{\cdot} \beta \\
\hline
_3N_r^{\circ} &  \gamma^{-1}R_b^{\cdot} \gamma & \gamma^{-1}{}L_a^{\cdot} \gamma & \gamma^{-1}R_c^{\cdot} \gamma & \gamma^{-1}R_b^{\cdot}\gamma &   \gamma^{-1}L_a^{\cdot} \gamma  &  \gamma^{-1}(L^{-1}_c)^{\cdot} \gamma\\
\hline
_1N_m^{\circ}  & \alpha^{-1}R_c^{\cdot} \alpha & \alpha^{-1}{}(L^{-1}_c)^{\cdot} \alpha &  \alpha^{-1}R_b^{\cdot} \alpha & \alpha^{-1}L_a^{\cdot} \alpha &  \alpha^{-1}R_b^{\cdot}\alpha    & \alpha^{-1}L_a^{\cdot} \alpha \\
\hline
_2N_m^{\circ}  &  \beta^{-1}(L^{-1}_c)^{\cdot} \beta & \beta^{-1}R_c^{\cdot} \beta &  \beta^{-1}R_b^{\cdot} \beta & \beta^{-1}L_a^{\cdot} \beta &  \beta^{-1}R_b^{\cdot} \beta   & \beta^{-1}L_a^{\cdot} \beta \\
\hline
\end{array}}
\]

\subsection{Isomorphisms of A-nuclei}

\begin{lemma} \label{ISOMOR OF COMP}
In any quasigroup $(Q,\circ)$ we have
\begin{equation} \label{LEFT_NUCLEUS_EQ}
 N^A_l = \{ (\alpha, \varepsilon, R_c\alpha R^{-1}_c ) \, \mid \,  \, \textrm{for all} \:   c\in Q \} = \{ (R^{-1}_c\gamma R_c, \varepsilon, \gamma) \, \mid \,  \, \textrm{for all} \:    c\in Q \}
 \end{equation}
\begin{equation} \label{RIGHT_NUCLEUS_EQ}
 N^A_r = \{ (\varepsilon, \beta,  L_c\beta L^{-1}_c ) \, \mid \,  \, \textrm{for all} \: c\in Q \} =   \{ (\varepsilon, L^{-1}_c\gamma L_c,  \gamma ) \, \mid \,  \, \textrm{for all} \:   c\in Q \}
 \end{equation}
\begin{equation} \label{MIDDLE_NUCLEUS_EQ}
 N^A_m = \{ (\alpha, P_c \alpha P^{-1}_c,  \varepsilon) \, \mid \,  \, \textrm{for all} \:  c\in Q \} = \{ (P^{-1}_c \beta P_c,  \beta,  \varepsilon) \, \mid \,  \, \textrm{for all} \:    c\in Q \}
 \end{equation}
\end{lemma}
\begin{proof}
 Re-write left A-nuclear equality  $\alpha x \circ y = \gamma (x\circ y)$ in the following form $R_y \alpha x  = \gamma R_y x$, $\gamma = R_y \alpha R^{-1}_y$ for all $y\in Q$.

Re-write right A-nuclear equality  $x \circ \beta y = \gamma (x\circ y)$ in the following form $L_x \beta y  = \gamma L_x y$, $\gamma = L_x \beta L^{-1}_x$ for all $x\in Q$.

In middle A-nuclear equality  $\alpha x \circ \beta y = x\circ y$  we put $y = x^{\prime}$, where   $x \circ x^{\prime} = a$ and $a$ is a fixed element of the set $Q$, i.e.  $x^{\prime} = P_a x$, where  $P_a$ is a permutation of the set $Q$.

Then $\alpha x \circ \beta x^{\prime}  = x\circ x^{\prime} = a$ for all $x \in Q$, $P_a \alpha x = \beta x^{\prime} = \beta P_a x$. Therefore $\beta = P_a \alpha P^{-1}_a$ for all $a\in Q$.
\end{proof}

\begin{lemma} \label{ISOMORPHISMS OF COMPONENTS}
If $K$ is a subgroup of the group $_1N^A_l$, then   $K\cong H \subseteq \, _3N^A_l$, moreover $H = R_x K R^{-1}_x$ for any  $x \in Q$.

If $K$ is a subgroup of the group $_2N^A_r$, then   $K\cong H \subseteq \, _3N^A_r$, moreover $H = L_x K L^{-1}_x$ for any  $x \in Q$.

If $K$ is a subgroup of the group $_1N^A_m$, then   $K\cong H \subseteq \, _2N^A_m$, moreover $H = P_x K P^{-1}_x$ for any  $x \in Q$.
\end{lemma}
\begin{proof}
The proof follows from Lemma \ref{ISOMOR OF COMP}.
\end{proof}

\begin{lemma} \label{ISOMORPHISMS OF COMPONENTS_1} In any quasigroup $(Q,\circ)$ the groups $N^A_l$, $_1N^A_l$ and  $_3N^A_l$; $N^A_m$, $_1N^A_m$ and
$_2N^A_m$; $N^A_r$,  $_2N^A_r$ and $_3N^A_r$ are isomorphic in pairs.
 \end{lemma}
\begin{proof}
The proof follows from Lemma \ref{ISOMORPHISMS OF COMPONENTS}.
\end{proof}

%
%\begin{corollary} \label{Order_of COMPONENTS} In any quasigroup $(Q,\circ)$ we have  $|N^A_l| = |_1N^A_l| =  |_3N^A_l|$; $|N^A_m| = |_1N^A_m| = |_2N^A_m| $;  $|_2N^A_r| = |_2N^A_r| = |_3N^A_r|$.
% \end{corollary}
%\begin{proof}
%The proof follows from  Lemma \ref{ISOMORPHISMS OF COMPONENTS_1}.
%\end{proof}

\begin{theorem} \label{COINCIDENCE NUCLEI_COMPONENTS} In any quasigroup:
\begin{enumerate}
  \item $_1 N_l^A =  \, _3 N_l^A$ if and only if  $RM \subseteq N_{S_Q}(_1N_l^A)$;
  \item $_1 N_l^A =  \, _3 N_l^A$ if and only if  $RM \subseteq N_{S_Q}(_3N_l^A)$;
  \item $_2 N_r^A =  \, _3 N_r^A$ if and only if  $LM \subseteq N_{S_Q}(_2N_r^A)$;
  \item $_2 N_r^A =  \, _3 N_r^A$ if and only if  $LM \subseteq N_{S_Q}(_3N_r^A)$;
  \item $_1 N_m^A =  \, _2 N_m^A$ if and only if  $PM \subseteq N_{S_Q}(_1N_m^A)$;
  \item $_1 N_m^A =  \, _2 N_m^A$ if and only if  $PM \subseteq N_{S_Q}(_2N_m^A)$.
\end{enumerate}
 \end{theorem}
\begin{proof}
Case 1. If $_1 N_l^A =  \, _3 N_l^A$, then from Lemma  \ref{ISOMORPHISMS OF COMPONENTS} it follows that $R_c \in N_{S_Q}(_1N_l^A)$ for any $c\in Q$.  Since $N_{S_Q}(_1N_l^A)$ is a group, further we have that $RM \subseteq N_{S_Q}(_1N_l^A)$.

Converse. Let  $\alpha \in \,_1N_l^A$. Then by Lemma \ref{ISOMORPHISMS OF COMPONENTS} $R_c\alpha R^{-1}_c \in \,_3N_l^A$. But $RM \subseteq N_{S_Q}(_1N_l^A)$. Then  $R_c\alpha R^{-1}_c = \beta \in \,_1N_l^A$. Thus  $_3N_l^A \subseteq \, _1N_l^A$.

Let  $\gamma \in \,_3N_l^A$. Then by Lemma \ref{ISOMORPHISMS OF COMPONENTS} $R^{-1}_c\gamma R_c \in \,_1N_l^A$. But $RM \subseteq N_{S_Q}(_1N_l^A)$.  Then  $R_c R^{-1}_c\gamma R_c R^{-1}_c = \gamma \in \,_1N_l^A$.
  Thus  $_3N_l^A \subseteq \, _1N_l^A$.
 Therefore  $_1 N_l^A =  \, _3 N_l^A$.

Cases 2--6 are proved in the similar way.
\end{proof}

\medskip

Moreover we can specify Theorem \ref{COINCIDENCE NUCLEI_COMPONENTS}.

\begin{theorem} \label{CENTRALIZRS OF AUTOTOPY_COMPONENTS}
\begin{enumerate}
  \item  In the group
$N^A_l$ of a  quasigroup $(Q,\circ)$ the first and third components of any element $T\in N^A_l$  coincide if and only if (i)  $RM(Q,\circ) \subseteq C_{S_Q} ({}_1N_l^A)$.
  \item  In the group
$N^A_l$ of a  quasigroup $(Q,\circ)$ the first and third components of any element $T\in N^A_l$  coincide if and only if $RM(Q,\circ) \subseteq C_{S_Q} ({}_3N_l^A)$.
  \item  In the group
$N^A_r$ of a  quasigroup $(Q,\circ)$ the second  and third components of any element $T\in N^A_r$  coincide if and only if $LM(Q,\circ) \subseteq C_{S_Q} ({}_2N_r^A)$.
  \item  In the group
$N^A_r$ of a  quasigroup $(Q,\circ)$ the second  and third components of any element $T\in N^A_r$  coincide if and only if $LM(Q,\circ) \subseteq C_{S_Q} ({}_3N_r^A)$.
  \item  In the group
$N^A_m$ of a  quasigroup $(Q,\circ)$ the first   and second components of any element $T\in N^A_m$  coincide if and only if $PM(Q,\circ) \subseteq C_{S_Q} ({}_1N_m^A)$.
  \item  In the group
$N^A_m$ of a  quasigroup $(Q,\circ)$ the first   and second components of any element $T\in N^A_m$  coincide if and only if $PM(Q,\circ) \subseteq C_{S_Q} ({}_2N_m^A)$.
\end{enumerate}
 \end{theorem}
\begin{proof}

The proof follows from Lemma \ref{ISOMORPHISMS OF COMPONENTS}.
\end{proof}

\begin{corollary} \label{ISOMORPHISMS OF COMPONENTS_LOOPS}
\begin{enumerate}
  \item If  $(Q,\cdot)$ is a right loop,  then  first and third components of any element $T\in N^A_l$  coincide and  ${}_1N_l^A ={}_3N_l^A = C_{S_Q}(\mathbb R) = C_{S_Q}(RM)$.
  \item If  $(Q,\cdot)$ is a left loop, then second  and third components of any element $T\in N^A_r$  coincide and  ${}_2N_r^A ={}_3N_r^A = C_{S_Q}(\mathbb L) = C_{S_Q}(LM)$.
  \item If  $(Q,\cdot)$ is an unipotent quasigroup,  then the first and second  components of any element $T\in N^A_m$  coincide and   ${}_1N_m^A ={}_2N_m^A = C_{S_Q}(\mathbb P)  = C_{S_Q}(PM)$.
   \item
If  $(Q, \cdot)$ is  a  loop, then ${}_1N_l^A = {}_3N_l^A$ $ = C_{S_Q}(\mathbb R) $ $= C_{S_Q}(RM) = C_{M}(RM)$ $= C_{FM}(RM)$;  ${}_2N_r^A = {}_3N_r^A$ $ = C_{S_Q}(\mathbb L) $ $= C_{S_Q}(LM) = C_{M}(LM)$ $= C_{FM}(LM)$.
  \item
If  $(Q, \cdot)$ is  an  unipotent left  loop, then ${}_2N_r^A = {}_3N_r^A$ $ = C_{S_Q}(\mathbb L) $ $= C_{S_Q}(LM) = C_{PLM}(LM)$ $= C_{FM}(LM)$;  ${}_1N_m^A = {}_2N_m^A$ $ = C_{S_Q}(\mathbb P) $ $= C_{S_Q}(PM) = C_{PLM}(PM) = C_{FM}(PM)$.
  \item
If $(Q, \cdot)$ is  an unipotent right loop, then  ${}_1N_l^A = {}_3N_l^A$ $ = C_{S_Q}(\mathbb R) $ $= C_{S_Q}(RM) = C_{PRM}(RM)$ $= C_{FM}(RM)$; ${}_1N_m^A = {}_2N_m^A$ $ = C_{S_Q}(\mathbb P) $ $= C_{S_Q}(PM) = C_{PRM}(PM) = C_{FM}(PM)$.
 \end{enumerate}
 \end{corollary}
\begin{proof}
Case 1. The proof follows from Lemma \ref{ISOMORPHISMS OF COMPONENTS}. In a right loop $(Q, \circ)$ there is right identity element $1$. Then $R_c = \varepsilon$, if $c=1$. Thus $\gamma = R_c\alpha R^{-1}_c = \alpha$.    See also Corollaries \ref{LEFT_RIGHT_MIDDLE_AUTOT} and \ref{A-NUCLEI_OF_LOOPS}. From equality $R_c\alpha  = \alpha R_c$ that is true for all $c\in Q$ and for all $\alpha \in {}_1N_l^A$ we  conclude that
${}_1N_l^A \subseteq C_{S_Q}({\mathbb R})$.

Prove that  $C_{S_Q}(\mathbb R)    \subseteq {}_1N_l^A$.
We rewrite the proof from (\cite{DRAP_JEDL}, Lemma 2.6).
 Any element $\psi $ of the group $C_{S_Q}(\mathbb R)$ fulfils equality  $\psi R_y x = R_y \psi  x$, i.e.
$\psi (xy) =   \psi (x)\cdot y$, i.e. $(\psi, \varepsilon, \psi)  \in N_l^A$.  Therefore $C_{S_Q}(\mathbb R)    \subseteq {}_1N_l^A$,  ${}_1N_l^A = C_{S_Q}(\mathbb R)$. Equality $C_{S_Q}(\mathbb R) = C_{S_Q}(RM)$ is true  since $\left< \mathbb R\right> =  RM$.

Case 2 is mirror of Case 1 and we omit its proof.

Case 3.
From equality $P_c\alpha  = \alpha P_c$ (Lemma \ref{ISOMORPHISMS OF COMPONENTS}) that is true for all $c\in Q$ and for all $\alpha \in {}_1N_m^A$ we  conclude that
${}_1N_m^A \subseteq C_{S_Q}({\mathbb P})$.

Any element $\psi $ of the group $C_{S_Q}(\mathbb P)$ fulfils equality  $\psi P_y x = P_y \psi  x = z$. If $\psi P_y x = z$, then  $  x \cdot \psi^{-1} z = y$. If $P_y \psi  x = z$, then $\psi x \cdot z = y$. Then $  x \cdot \psi^{-1} z = \psi x \cdot z$,  $(\psi,\psi,  \varepsilon)  \in N_m^A$.  Therefore $C_{S_Q}({\mathbb P}) \subseteq {}_1N_m^A$.
Thus
${}_1N_m^A = C_{S_Q}({\mathbb P})$. Equality $C_{S_Q}(\mathbb P) = C_{S_Q}(PM)$ is true  since $\left< \mathbb P\right> =  PM$.

 Case 4. Equality $C_{S_Q}(\mathbb R)= C_{S_Q}(RM)$ follows from the fact that $\left<\mathbb R\right> = RM$.

Equality $C_{S_Q}(RM) = C_{M}(RM)$ follows from the fact that  ${}_1N_l^A\subseteq M$ and $RM\subseteq M$ since in loop case any element of the group $N_l^A$ has the form $(L_c, \varepsilon, L_c)$
(Corollary \ref{A-NUCLEI_OF_LOOPS}).

Equality $C_{M}(RM)= C_{FM}(RM)$ follows from the fact that $C_{M}(RM)= C_{S_Q}(RM)$,  $M\subseteq FM \subseteq S_Q$,   ${}_1N_l^A\subseteq FM$ and $RM\subseteq FM$.

The proofs of  Cases 5 and 6 are similar to the proof of Case 4.
\end{proof}
\begin{remark}
\begin{enumerate}
  \item  From equality $R_c\alpha  = \alpha R_c$ (Corollary \ref{ISOMORPHISMS OF COMPONENTS_LOOPS}, Case 1) that is true for all $c\in Q$ and for all $\alpha \in {}_1N_l^A$ we can do some additional conclusions:
 $ \mathbb{R} \subseteq C_{S_Q}(\alpha)$;
  $ \mathbb{R} \subseteq C_{S_Q}({}_1N_l^A)$;
 $ RM(Q) \subseteq C_{S_Q}(\alpha)$;
  $ RM(Q) \subseteq C_{S_Q}({}_1N_l^A)$.
  \item  From equality $L_c\beta  = \beta L_c$ (Corollary \ref{ISOMORPHISMS OF COMPONENTS_LOOPS}, Case 2) that is true for all $c\in Q$ and for all $\beta \in {}_2N_r^A$ we can do some additional conclusions, namely:
 $ \mathbb{L} \subseteq C_{S_Q}(\beta)$;
  $ \mathbb{L} \subseteq C_{S_Q}({}_2N_r^A)$;
 $ LM(Q) \subseteq C_{S_Q}(\beta)$;
  $ LM(Q) \subseteq C_{S_Q}({}_2N_r^A)$.
  \item  From equality $P_c\alpha  = \alpha P_c$ (Corollary \ref{ISOMORPHISMS OF COMPONENTS_LOOPS}, Case 3) that is true for all $c\in Q$ and for all $\alpha \in {}_1N_m^A$ we can do some additional conclusions, namely:
 $ \mathbb{P} \subseteq C_{S_Q}(\alpha)$;
  $ \mathbb{P} \subseteq C_{S_Q}({}_1N_m^A)$;
 $ PM(Q) \subseteq C_{S_Q}(\alpha)$;
  $ PM(Q) \subseteq C_{S_Q}({}_1N_m^A)$.
\end{enumerate}
\end{remark}

\subsection{A-nuclei by some isotopisms}

Mainly  we shall use Lemma \ref{ISOM_AUTOT_GR} \cite[Lemma 1.4]{1a}:  if quasigroups $(Q, \circ)$ and $(Q, \cdot)$ are isotopic with isotopy $T$,
i.e. $(Q,\circ) = (Q, \cdot)T$, then $ Avt(Q, \circ) = T^{-1} Avt  (Q, \cdot) T$.

\begin{lemma} \label{COINCIDENCE_OF_NUCLEI_COMP}
For any quasigroup $(Q, \cdot)$ there exists its isotopic  image $(Q, \circ)$ such that:
\begin{enumerate}
\item
any autotopy of the form $(\alpha, \varepsilon, \gamma)$ of quasigroup $(Q, \cdot)$ is transformed in autotopy of the form $(\gamma, \varepsilon, \gamma)$ of quasigroup $(Q, \circ)$;
\item
any autotopy of the form $(\alpha, \varepsilon, \gamma)$ of quasigroup $(Q, \cdot)$ is transformed in autotopy of the form $(\alpha, \varepsilon, \alpha)$ of quasigroup $(Q, \circ)$;
\item
any autotopy of the form $(\varepsilon, \beta, \gamma)$ of quasigroup $(Q, \cdot)$ is transformed in autotopy of the form $(\varepsilon, \beta, \beta)$ of quasigroup $(Q, \circ)$;
\item
any autotopy of the form $(\varepsilon, \beta, \gamma)$ of quasigroup $(Q, \cdot)$ is transformed in autotopy of the form $(\varepsilon, \gamma, \gamma)$ of quasigroup $(Q, \circ)$;
\item
any autotopy of the form $(\alpha, \beta, \varepsilon)$ of quasigroup $(Q, \cdot)$ is transformed in autotopy of the form $(\alpha, \alpha, \varepsilon)$ of quasigroup $(Q, \circ)$;
\item
any autotopy of the form $(\alpha, \beta, \varepsilon)$ of quasigroup $(Q, \cdot)$ is transformed in autotopy of the form $(\beta, \beta, \varepsilon)$ of quasigroup $(Q, \circ)$.
\end{enumerate}
\end{lemma}
\begin{proof}
Case 1.
If $(Q,\circ) = (Q,\cdot)(R^{-1}_a, \varepsilon, \varepsilon)$, where  element $a$ is a fixed element of the set $Q$, then  $(Q,\circ)$ is a right loop (Corollary \ref{ONE_COMPONT_ISOTOPY}).

By Lemma \ref{ISOM_AUTOT_GR}, if quasigroups $(Q, \circ)$ and $(Q, \cdot)$ are isotopic with isotopy, then $ Avt(Q, \circ) = T^{-1} Avt  (Q, \cdot) T$.
Then autotopy of the form $(\alpha, \varepsilon, \gamma)$ of $(Q, \cdot)$ passes in autotopy of the form $(R_a\alpha R^{-1}_a, \varepsilon, \gamma)$.

In any right loop  any element of the group  $N^A_l (Q, \circ)$ has equal  the first and the third components, i.e. it has the form $(\gamma, \varepsilon, \gamma)$   (Corollary \ref{LEFT_RIGHT_MIDDLE_AUTOT}).

Case 2.
We can take the following isotopy $(Q,\circ) = (Q,\cdot)(\varepsilon, \varepsilon, R_a)$. In this case   $(Q,\circ)$ is a right loop.

Case 3. We can take the following isotopy $(Q,\circ) = (Q,\cdot)(\varepsilon, \varepsilon, L_a)$. In this case   $(Q,\circ)$ is a left loop.

Case 4. We can take the following isotopy  $(Q,\circ) = (Q,\cdot)(\varepsilon, L^{-1}_a, \varepsilon)$. In this case $(Q,\circ)$ is a left loop.

Case 5. We can take the following isotopy  $(Q,\circ) = (Q,\cdot)(\varepsilon, P_a,  \varepsilon)$. In this case   $(Q,\circ)$ is an unipotent quasigroup.

Case 6. We can take the following isotopy  $(Q,\circ) = (Q,\cdot)( P^{-1}_a, \varepsilon, \varepsilon)$. In this case $(Q,\circ)$ is an unipotent quasigroup.
\end{proof}

\begin{corollary} \label{COINCIDENCE_OF_NUCLEI_COMP_1}
For any quasigroup $(Q, \cdot)$ there exists its isotopic  image $(Q, \circ)$ such that:
\begin{enumerate}
\item $_1N^A_l (Q, \circ)  = \, _3N^A_l (Q, \circ) = \, _3N^A_l (Q, \cdot)$;
\item
$_1N^A_l (Q, \circ)  = \, _3N^A_l (Q, \circ)  = \, _1N^A_l (Q, \cdot)$.
\item
$_2N^A_r (Q, \circ)  = \, _3N^A_r (Q, \circ) = \, _2N^A_r (Q, \cdot)$;
\item $_2N^A_r (Q, \circ)  = \, _3N^A_r (Q, \circ) = \, _3N^A_r (Q, \cdot)$;
\item
 $_1N^A_m (Q, \circ)  = \, _2N^A_m (Q, \circ) = {}_1N^A_m (Q, \cdot)$;
\item $_1N^A_m (Q, \circ)  = \, _2N^A_m (Q, \circ) = \, _2N^A_m (Q, \cdot)$.
\end{enumerate}
\end{corollary}
\begin{proof}
The proof follows from Lemma \ref{COINCIDENCE_OF_NUCLEI_COMP}.
\end{proof}

Isotopy of the form $(R^{-1}_x, L^{-1}_y, \varepsilon)$  is called LP-isotopy (Definition \ref{LP_ISOt_def}).

\begin{lemma}
\begin{enumerate}
\item Let $(Q, \cdot)$ be a loop.
If $(Q, \circ) = (Q, \cdot)(R^{-1}_a, L^{-1}_b, \varepsilon)$,   then
\begin{itemize}
\item[(i)] $N^A_l (Q, \circ) = N^A_l (Q, \cdot)$,
\item[(ii)] $N^A_r (Q, \circ) = N^A_r (Q, \cdot)$.
\end{itemize}
\item Let $(Q, \cdot)$ be an unipotent right loop.
If $(Q, \circ) = (Q, \cdot) (\varepsilon, P_{a}, R_{b})$,   then
\begin{itemize}
\item[(i)] $N^A_l (Q, \circ) = N^A_l (Q, \cdot)$,
\item[(ii)] $N^A_m (Q, \circ) = N^A_m (Q, \cdot)$.
\end{itemize}
\item Let $(Q, \cdot)$ be an unipotent left loop.
If $(Q, \circ) = (Q, \cdot)(P^{-1}_{a}, \varepsilon, L_b)$,   then
\begin{itemize}
\item[(i)] $N^A_r (Q, \circ) = N^A_r (Q, \cdot)$,
\item[(ii)] $N^A_m (Q, \circ) = N^A_m (Q, \cdot)$.
\end{itemize}
\end{enumerate}
\end{lemma}
\begin{proof}
In this case $(Q, \circ)$ is a loop.

Case 1, (i). Let   $(L_c, \varepsilon, L_c) \in N^A_l (Q, \cdot)$.
By the  isotopy any element of the group $N^A_l (Q, \circ)$ has the form $(R_aL_c R^{-1}_a, \varepsilon, L_c)$. Since $(Q, \circ)$ (or $(Q, \cdot)$) is a loop, then $R_aL_c R^{-1}_a = L_c$, $N^A_l (Q, \cdot) \subseteq N^A_l (Q, \circ)$. Inverse inclusion is proved in the similar way.
Therefore $N^A_l (Q, \cdot) = N^A_l (Q, \circ)$.

Case 1, (ii) is proved in the similar way with Case 1.

Case 2  is "$(13)$-parastrophe image" of Case 1.
In this case $(Q, \circ)$ is an unipotent right loop (Theorem \ref{LP_ISOT_AND_ANALOGS}).
The forms of components of left nucleus and middle nucleus of unipotent right loop are given in Theorem \ref{A-NUCLEI_OF_LOOPS}.

Further proof of Cases 2, (i) and 2, (ii) are similar with the proof of Case 1, (i) and we omit it.

Case 3  is "$(23)$-parastrophe image" of Case 1.
\end{proof}

Notice, also it is possible to see on Case 3  as on  "$(12)$-parastrophe image" of Case 2.

\subsection{Quasigroup bundle and nuclei}

\begin{definition}
Let $(Q,\circ)$ be a quasigroup.

Denote the set of all elements of the form $L_cL^{-1}_d$, where $L_c, L_d$ are left translations of $(Q,\circ)$, by $\cal M_L$.

Denote the set of all elements of the form $R_cR^{-1}_d$, where $R_c, R_d$ are right translations of $(Q,\circ)$, by $\cal M_R$.

Denote the set of all elements of the form $P_cP^{-1}_d$, where $P_c, P_d$ are middle  translations of $(Q,\circ)$,  by $\cal M_P$.

Further denote the following sets   ${\cal M}^{\ast}_{{\cal L}} =\{ L^{-1}_c L_d \, \mid \, a, b \in Q \} $,  ${\cal M}^{\ast}_{{\cal R}} =\{ R^{-1}_cR_d \, \mid \, a, b \in Q \} $,  ${\cal M}^{\ast}_{{\cal P}} =\{ P^{-1}_c P_d \, \mid \, a, b \in Q \} $.
\end{definition}

It is clear that in any quasigroup   ${\cal M_L} \subseteq  LM$, ${\cal M}^{\ast}_{{\cal L}}  \subseteq  LM$, ${\cal M_R} \subseteq  RM$, ${\cal M}^{\ast}_{{\cal R}}  \subseteq  RM$, ${\cal M_P} \subseteq  PM$, ${\cal M}^{\ast}_{{\cal P}}  \subseteq  PM$.

Notice that $(L_cL^{-1}_d)^{-1} = L_dL^{-1}_c$. Therefore $({\cal M_L})^{-1} = {\cal M_L}$, $({\cal M_R})^{-1} = {\cal M_R}$ and so on.

\begin{remark} \label{INVERSE_PERM_CENTRAL}
It is clear that ${\cal M_L} \subseteq LM(Q,\cdot)$. Then from Definition \ref{CENTRALIZER_DEF} it follows that $C_{S_Q}({\cal M_L}) \supseteq C_{S_Q}LM(Q, \cdot)$ and so on.
\end{remark}

In \cite{12} the set of all LP-isotopes  of a fixed quasigroup $(Q,\cdot)$ is called a bundle  of a quasigroup $(Q,\cdot)$.

\begin{remark}
If $|Q| = n$, then by Lemma \ref{LOOP_TRANSLATIONS} any of the sets $\cal M_L$,  $\cal M_R$,  $\cal M_P$, ${\cal M}^{\ast}_{{\cal L}}$, ${\cal M}^{\ast}_{{\cal R}}$, and ${\cal M}^{\ast}_{{\cal P}}$ can  contain  $n$  (not necessary different)  simply  transitive subsets  of order $n$, that correspond to  the sets of left,  right and middle translations of corresponding loops,  unipotent left loops and unipotent right loops of a quasigroup $(Q,\cdot)$.
\end{remark}

\begin{theorem} In any quasigroup $(Q, \cdot)$:
\[
\begin{array}{llllllll}
              \left(1\right)   &   {}_1N^A_l(Q,\cdot) & \subseteq & C_{S_Q}({\cal M^{\ast}_R}) & \left(2\right) & {}_3N^A_l(Q,\cdot) & \subseteq & C_{S_Q}({\cal M_R}) \\
\left(3\right) & {}_2N^A_r(Q,\cdot) & \subseteq & C_{S_Q}({\cal M^{\ast}_L}) & \left(4\right) & {}_3N^A_r(Q,\cdot) & \subseteq  & C_{S_Q}({\cal M_L}) \\
 \left(5\right) & {}_1N^A_m(Q,\cdot) & \subseteq & C_{S_Q}({\cal M^{\ast}_P}) & \left(6\right) & {}_2N^A_m(Q,\cdot) & \subseteq  & C_{S_Q}({\cal M_P})\\
\end{array}
\]
\end{theorem}
\begin{proof}
Case 1.  From Lemma \ref{TWO_AUTOTOPY_COMPONENTS}  it follows that in any autotopy $(\alpha, \varepsilon, \gamma)$  of a quasigroup $(Q, \circ)$   permutations $\alpha$ and $\varepsilon$   uniquely determine permutation $\gamma$. Therefore from Lemma \ref{ISOMORPHISMS OF COMPONENTS} in this case
$R_c\alpha R^{-1}_c = R_d\alpha R^{-1}_d$, $\alpha  = R^{-1}_cR_d\alpha R^{-1}_dR_c = R^{-1}_cR_d\alpha (R^{-1}_cR_d)^{-1}$  for all $c, d \in Q$.

Case 2 is proved in the similar way.

Case 3.
From Lemma \ref{TWO_AUTOTOPY_COMPONENTS}  it follows that in any autotopy $(\varepsilon, \beta,  \gamma)$  of a quasigroup $(Q, \circ)$  permutations $\varepsilon$ and $\beta$     uniquely determine permutation $\gamma$. Therefore from Lemma \ref{ISOMORPHISMS OF COMPONENTS}  in this case
$L_c\beta L^{-1}_c = L_d\beta L^{-1}_d$, $\beta  = L^{-1}_cL_d\beta L^{-1}_d L_c = L^{-1}_c L_d\beta (L^{-1}_c L_d)^{-1}$  for all $c, d \in Q$.

Case 4 is proved similarly.

Case 5.
From Lemma \ref{TWO_AUTOTOPY_COMPONENTS}  it follows that in any autotopy $(\alpha, \beta,  \varepsilon)$  of a quasigroup $(Q, \circ)$  permutations   $\beta$ and $\varepsilon$   uniquely determine permutation $\alpha$. Therefore from Lemma \ref{ISOMORPHISMS OF COMPONENTS}  in this case
$P_c\alpha  P^{-1}_c = P_d\alpha  P^{-1}_d$, $\alpha  = P^{-1}_cP_d\alpha  P^{-1}_d P_c = P^{-1}_c P_d\alpha  (P^{-1}_c P_d)^{-1}$  for all $c, d \in Q$.

Case 6 is proved similarly.
\end{proof}

\begin{corollary} \label{COROLL_FROM_BUNDLE_QUAS}
\begin{enumerate}
\item In any right loop  $(Q, \cdot)$
\begin{enumerate}
                  \item  ${}_1N^A_l(Q,\cdot) \subseteq  C_{S_Q}RM(Q,\cdot)$,
                  \item  $ {}_3N^A_l(Q,\cdot) \subseteq  C_{S_Q}RM(Q,\cdot)$.
                 \end{enumerate}
\item In any left loop  $(Q, \cdot)$
\begin{enumerate}
                   \item  $ {}_2N^A_r(Q,\cdot) \subseteq C_{S_Q}LM(Q,\cdot)$,
                   \item  $ {}_3N^A_r(Q,\cdot) \subseteq C_{S_Q}LM(Q,\cdot)$.
                 \end{enumerate}
\item In any unipotent quasigroup  $(Q, \cdot)$
\begin{enumerate}
                   \item   $ {}_1N^A_m(Q,\cdot) \subseteq  C_{S_Q}PM(Q,\cdot)$,
                   \item  $ {}_2N^A_m(Q,\cdot) \subseteq  C_{S_Q}PM(Q,\cdot)$.
                 \end{enumerate}
\end{enumerate}
\end{corollary}
\begin{proof} Case 1, (a).
Since any right  loop $(Q,\cdot)$ has right identity element $1$, then $R_1$ is identity permutation of the set $Q$. Therefore the set $\cal M^{\ast}_R$ contains any right translation of this right  loop $(Q,\cdot)$. It is clear that in this case $C_{S_Q} \mathbb{R} \supseteq  C_{S_Q}({\cal M_R}) \supseteq  C_{S_Q} RM$. But $C_{S_Q} \mathbb{R} = C_{S_Q} RM$ since $\left< \mathbb{R} \right> = RM$.   Then $C_{S_Q}({\cal M^{\ast}_R}) = C_{S_Q}RM$.

All other cases are proved in the similar way.
\end{proof}

\begin{corollary} Let  $x\circ y = R^{-1}_a x \cdot  y$ for all $x, y \in Q$. Then
                 ${}_1N^A_l(Q,\cdot) \subseteq  C_{S_Q}(\mathbb R(Q, \circ))$;
                 $ {}_3N^A_l(Q,\cdot) \subseteq  C_{S_Q}(\mathbb R(Q, \circ))$.

Let  $x\circ y =  x \cdot L^{-1}_b y$ for all $x, y \in Q$. Then
                 ${}_2N^A_r(Q,\cdot) \subseteq  C_{S_Q}(\mathbb L(Q, \circ))$;
                 $ {}_3N^A_r(Q,\cdot) \subseteq  C_{S_Q}(\mathbb L(Q, \circ))$.
\end{corollary}

\subsection{Action of A-nuclei}

\begin{theorem} \label{GROUP_ACTION_ON QUAS}
Any of the groups $_1N^A_l$,  $_3N^A_l$, $_1N^A_m$,
$_2N^A_m$,   $_2N^A_r$ and $_3N^A_r$ of a quasigroup $(Q, \cdot)$ acts  the set $Q$ free (semiregular) and in such manner that stabilizer of any element $x\in Q$ by this action is the identity group, i.e. $|\, _1N^A_l{}_{x}| = |\,_3N^A_l{}_{x}| = |\, _2N^A_r{}_{x} | =  |\, _3N^A_r{}_{x} | = |\, _1N^A_m{}_{x} | = |\, _3N^A_m{}_{x} | = 1$.
\end{theorem}
\begin{proof}
By Theorem  \ref{DEF_OF_COMP_BY_AUTOT_AND_EL} any autotopy $(\alpha, \beta, \gamma)$ ($(\alpha, \beta, \gamma) \equiv (\alpha_1, \alpha_2,\alpha_3)$ ) is uniquely defined   by any autotopy component $\alpha_i$, $i\in \{1, 2, 3 \}$,  and by element $b = \alpha_j a$, $i\neq j$. By definition of $_1N^A_l$ any element $\alpha$ of the group $_1N^A_l$ is the first component of an autotopy of the form $(\alpha, \varepsilon, R_c\alpha R^{-1}_c)$ (Lemma \ref{ISOMORPHISMS OF COMPONENTS}).

In order to apply Theorem  \ref{DEF_OF_COMP_BY_AUTOT_AND_EL} to the group $N^A_l$ we take $i =2$ (in this case $\beta = \varepsilon$ for any $T\in N^A_l$). And in fact any left nuclear autotopy $T$ is defined by the image $\alpha a$ of a fixed element $a\in Q$ by action of the permutation $\alpha \in {}_1N^A_l$.

Therefore, if $\alpha_1, \alpha_2 \in \, _1N^A_l$, then   $\alpha_1 x \neq  \alpha_2 x$ for any $x\in Q$, i.e. $_1N^A_l$ acts on the set $Q$ free (fixed point free), or equivalently, if $\alpha (x) = x$ for some $x\in Q$ then $\alpha  = \varepsilon$.

If $\alpha a = \beta a$ for some $\alpha, \beta \in \, _1N^A_l$, $a\in Q$, then $\beta^{-1}\alpha x =  x$, $\beta^{-1}\alpha = \varepsilon$. Therefore $|\, _1N^A_l{}_{x}| = |\,_3N^A_l{}_{x}| = |\, _2N^A_r{}_{x} | =  |\, _3N^A_r{}_{x} | = |\, _1N^A_m{}_{x} | = |\, _3N^A_m{}_{x} | = 1$.
 \end{proof}

\begin{corollary} \label{SUBGR_FREE_ACT_GR}
If $H$ is a subgroup of the group $_1N^A_l$ ($_3N^A_l$, $_1N^A_m$,
$_2N^A_m$,   $_2N^A_r$ and $_3N^A_r$), then the group $H$ acts on the set $Q$
free and in such manner that stabilizer of any element $x\in Q$ by this action is the identity group.
\end{corollary}
\begin{proof}
The proof follows from Theorem \ref{GROUP_ACTION_ON QUAS}.
\end{proof}

\begin{corollary} \label{ORDERS_OF ORBIT_IN_NUCLEI}
In any finite quasigroup $(Q, \cdot)$ the  orders of orbits of the group $\,_1N^A_l$ ($_3N^A_l$, $_1N^A_m$,
$_2N^A_m$,  $_2N^A_r$ and $_3N^A_r$) by its action on quasigroup $(Q, \cdot)$ are equal   with the order of the group $\,_1N^A_l$ ($_3N^A_l$, $_1N^A_m$,
$_2N^A_m$,  $_2N^A_r$ and $_3N^A_r$), respectively.
\end{corollary}
\begin{proof}
By Theorem \ref{GROUP_ACTION_ON QUAS} $|\, _1N^A_l{}_{x}| = 1$ for any $x$ in $Q$.
\end{proof}

\begin{remark} \label{ORDERS_OF ORBIT_IN_NUCLEI_BIJECTIONS}
For infinite case we can re-formulate Corollary \ref{ORDERS_OF ORBIT_IN_NUCLEI} in the language of bijections. For example, there exists a bijection between any orbit of the  group $\,_1N^A_l$ by its action on quasigroup $(Q, \cdot)$ and the group $\,_1N^A_l$. And so on.
\end{remark}

\begin{corollary} \label{NUCLEI_ORDER_DIV}
In any finite quasigroup $(Q, \cdot)$ the order of any subgroup   of the group $N^A_l$ ($N^A_m$, $N^A_r$,  $_1N^A_l$,  $_3N^A_l$, $_1N^A_m$, $_2N^A_m$,  $_2N^A_r$ and $_3N^A_r$) divide the order of the set $Q$.
\end{corollary}
\begin{proof}
We can use Corollaries  \ref{SUBGR_FREE_ACT_GR}, \ref{ORDERS_OF ORBIT_IN_NUCLEI} and Lagrange Theorem \cite{KM} about order of  subgroup of any finite group.
\end{proof}

\begin{corollary} \label{ORDERS}
In any finite quasigroup $(Q, \cdot)$ we have $|N^A_l| \leqslant |Q|$,  $|N^A_m| \leqslant |Q|$, $|N^A_r| \leqslant |Q|$,  $|_1N^A_l|\leqslant|Q|$,  $|_3N^A_l| \leqslant |Q|$, $|_1N^A_m|\leqslant |Q|$,
$|_2N^A_m|\leqslant |Q|$,  $|_2N^A_r| \leqslant |Q|$ and $|_3N^A_r|\leqslant |Q|$.
\end{corollary}
\begin{proof}
The proof follows from Corollary  \ref{NUCLEI_ORDER_DIV}.
\end{proof}

\begin{lemma} \label{NUCLEAR_QUAS}
If the order of a quasigroup $(Q, \cdot)$ is equal to the order of the group $N^A_l$, then the orders of all
groups $N^A_m$, $N^A_r$,  $_1N^A_l$,  $_3N^A_l$, $_1N^A_m$,
$_2N^A_m$,  $_2N^A_r$ and $_3N^A_r$ are equal to $|Q|$.
\end{lemma}
\begin{proof}
It is clear that the order of a quasigroup  $(Q, \cdot)$  is invariant relative to the parastrophy. I.e. $|(Q, \cdot)| = |(Q, \cdot)^{\sigma}|$, for any  $\sigma \in S_3$. From Table 3 it follows that $\left( N^A_l \right)^{(12)} = N^A_r$,
$\left( N^A_l \right)^{(23)} = N^A_m$. Therefore $|N^A_l| = |N^A_r| = |N^A_m| = |Q|$. Further we can apply Lemma \ref{ISOMORPHISMS OF COMPONENTS}.
\end{proof}

\begin{remark} \label{NUCLEAR_QUAS_EQUALITY}
Analogs of Lemma \ref{NUCLEAR_QUAS} are true for any group  from the following list: $N^A_m$, $N^A_r$,  $_1N^A_l$,  $_3N^A_l$, $_1N^A_m$, $_2N^A_m$,  $_2N^A_r$ and $_3N^A_r$.
\end{remark}

\begin{corollary} \label{ISOMORPHISMS OF COMPONENT_LOOPS}
\begin{enumerate}
  \item If  $(Q,\cdot)$ is finite  right loop,  then $|C_{S_Q}(FM)| \leq |C_{S_Q}(M)| \leq|C_{S_Q}(RM)| \leq |Q|$.
  \item If  $(Q,\cdot)$ is finite  left loop,  then $|C_{S_Q}(FM)| \leq |C_{S_Q}(M)|\leq |C_{S_Q}(LM)| \leq |Q|$.
  \item If  $(Q,\cdot)$ is finite  unipotent quasigroup, then $|C_{S_Q}(FM)| \leq |C_{S_Q}(PM)| \leq |Q|$.
\end{enumerate}
 \end{corollary}
\begin{proof}
Case 1. Inclusions  $FM \supseteq M \supseteq RM$ follow from definitions of these groups. Therefore $|C_{S_Q}(FM)| \leq |C_{S_Q}(M)| \leq|C_{S_Q}(RM)|$. Inequality  $|C_{S_Q}(RM)| \leq |Q|$  follows from Corollaries \ref{ISOMORPHISMS OF COMPONENTS_LOOPS} and \ref{ORDERS}.

Cases 2 and 3 are proved similarly.
\end{proof}

\begin{example} \label{MIDDLE_NUCLEUS_OF_LOOP}
We give an example of a loop with $N_m = \{ 0, 1 \}$. Notice this nucleus is not a normal subloop of loop $(Q, \ast)$. Indeed, $5 \{ 0 , 1 \} = \{ 5, 3 \}$, $\{ 0 , 1
\} 5  = \{ 5, 4 \}$.

Then  ${}_1N^A_m = \{\varepsilon, R_1\}$, where $R_1 = ( \, 0 \, 1 \, ) \, (\, 2\, 4\, )\,(\, 3\, 5\, )$, ${}_2N^A_m = \{\varepsilon, L^{-1}_1\}$, where $L_1 = L^{-1}_1 = ( \, 0 \, 1 \, ) \, (\, 2\, 3\, )\,(\, 4\, 5\,
)$. The group ${}_1N^A_m$ by  action on the set $Q$ has the following set of orbits $\{ \{0, 1\}, \{2, 4\}, \{ 3, 5\} \}$, and the group ${}_2N^A_m$
has the following set $\{ \{0, 1\}, \{2, 3\}, \{ 4, 5\} \}$.

%\begin{table}[HPTB]
 \centering
\begin{tabular}{r|rrrrrr}
$\ast$ & 0 & 1 & 2 & 3 & 4 & 5\\
\hline
    0 & 0 & 1 & 2 & 3 & 4 & 5 \\
    1 & 1 & 0 & 3 & 2 & 5 & 4 \\
    2 & 2 & 4 & 1 & 5 & 3 & 0 \\
    3 & 3 & 5 & 4 & 0 & 2 & 1 \\
    4 & 4 & 2 & 5 & 1 & 0 & 3 \\
    5 & 5 & 3 & 0 & 4 & 1 & 2
\end{tabular} \hspace{.5cm}
%\end{table}

\end{example}

\subsection{A-nuclear  quasigroups}

We start from  theorem which is a little generalization of Belousov Regular  Theorem \cite[Theorem 2.2]{VD}.
In \cite{kepka71, kepka75} Kepka studied  regular mappings of groupoids including   $n$-ary case.

\begin{theorem} \label{SUBQUAS_SUB_NUCL_TH}
\begin{enumerate}
  \item If an orbit $K$ by the action of a subgroup $H$ of the  group $_1N^A_l$ on a quasigroup $(Q, \cdot)$ is a subquasigroup of   $(Q, \cdot)$, then $(K, \cdot)$ is an isotope of the group $H$.
  \item If an orbit $K$ by the action of a subgroup $H$ of the group $_3N^A_l$ on a quasigroup $(Q, \cdot)$ is a subquasigroup of   $(Q, \cdot)$, then $(K, \cdot)$ is an isotope of the group $H$.
  \item If an orbit $K$ by the action of a subgroup $H$ of the group $_2N^A_r$ on a quasigroup $(Q, \cdot)$ is a subquasigroup of   $(Q, \cdot)$, then $(K, \cdot)$ is an isotope of the group $H$.
  \item If an orbit $K$ by the action of a subgroup $H$ of the group $_3N^A_r$ on a quasigroup $(Q, \cdot)$ is a subquasigroup of   $(Q, \cdot)$, then $(K, \cdot)$ is an isotope of the group $H$.
  \item If an orbit $K$ by the action of a subgroup $H$ of the group $_1N^A_m$ on a quasigroup $(Q, \cdot)$ is a subquasigroup of   $(Q, \cdot)$, then $(K, \cdot)$ is an isotope of the group $H$.
  \item If an orbit $K$ by the action of a subgroup $H$ of the group $_2N^A_m$ on a quasigroup $(Q, \cdot)$ is a subquasigroup of   $(Q, \cdot)$, then $(K, \cdot)$ is an isotope of the group $H$.
\end{enumerate}
\end{theorem}
\begin{proof} Sometimes in the proof  we shall denote by the symbol $\ast$ operation  of composition of permutations (of bijections) of the set $Q$.

Case 1.  Let $k\in K$.
 By Lemma  \ref{COINCIDENCE_OF_NUCLEI_COMP} for any quasigroup $(Q, \cdot)$ there exists its isotopic  image $(Q, \star) = (Q, \cdot)(\varepsilon, \varepsilon, R_k)$ such that:
$_1N^A_l (Q, \star)  = \, _3N^A_l (Q, \star)  = \, _1N^A_l (Q, \cdot)$. Notice by isotopy of such form ($k\in K$) subquasigroup $(K, \cdot)$  passes in subquasigroup  $(K, \star)$ of quasigroup $(Q, \star)$.

Moreover in right loop $(Q, \star)$ any autotopy of the form $(\alpha, \varepsilon, \gamma)$ of quasigroup $(Q, \cdot)$ takes the form $(\alpha, \varepsilon, \alpha)$. Therefore $H \subseteq \,{}_1N^A_l(Q, \star)$.

 We follow \cite[Theorem 2.2]{VD}. In this part of the proof we "are"  in right loop $(Q, \star)$.

 Any element  $l\in K$ it is possible to present in the form $l = \delta k$, where $\delta \in H \subseteq {}_1N^A_l$. Notice $k = \varepsilon k$.
  If $l = \delta k$, $r=\lambda k$, then $\delta k \star \lambda k \in K$ since $(K, \star)$ is a subquasigroup of quasigroup $(Q, \star)$. Thus there exists $\mu \in \ H$ such that $\delta k \star \lambda k = \mu k$ since $Hk=K$.

On the set $H$ we can define operation $\circ$ in the following way: $ \delta  \circ \lambda  = \mu $ if and only if $\delta k \star \lambda k = \mu k$.

Prove  that
\begin{equation} \label{ISOM_LEFT_nucl}
\, (H, \circ) \cong (K, \star)
\end{equation}
Define the  map $\varphi$ in the following way:   $\varphi:  \lambda \mapsto \lambda k$,  $\varphi (\delta \circ \lambda) = \varphi(\delta) \star \varphi (\lambda) =   \delta k \star \lambda k$. The  map $\varphi$ is  bijective, since  action of the group of permutations  $H\subseteq \, {}_1N^A_l$ on the set $K$ is regular (i.e. it is simply transitive). Therefore $(H, \circ) \cong (K, \star)$.

Let $\alpha, \lambda, \mu \, \in H$. Notice,  corresponding permutation to the permutation $\alpha$ from the set $_3N^A_l(Q,\star)$ also is permutation $\alpha$ (Lemma \ref{COINCIDENCE_OF_NUCLEI_COMP}). From definition of the set $K$ we have $\alpha K = K$. Indeed $\alpha K = \alpha (H k) = (\alpha H) k  = Hk = K$.

Then restriction of the action of the triple $(\alpha, \varepsilon, \alpha)\in N^A_l(Q,\star)$  on subquasigroup $(K, \star)\subseteq (Q, \star)$ is an autotopy of subquasigroup $(K, \star)$.
We have $\alpha ( (\lambda \circ \mu)k)  = \alpha( \lambda k  \star \mu k) = (\alpha  \lambda) k  \star \mu k = ((\alpha  \lambda)   \circ \mu) k$.

Therefore we obtain
\begin{equation} \label{nucl_IS_GR}
\alpha  (\lambda \circ \mu) = (\alpha \lambda) \circ \mu
\end{equation}

If we put in equality (\ref{nucl_IS_GR}) $\lambda = \varepsilon$, then we obtain
$\alpha(\varepsilon \circ \mu) =  \alpha \circ \mu.$

Since $(H, \circ)$ is a quasigroup, $\varepsilon \in H$,  then the mapping $L^{\circ}_{\varepsilon}$,  $L^{\circ}_{\varepsilon} \mu =  \varepsilon \circ \mu$ for all $\mu \in H$ is a permutation of the set $H$.
Then
\begin{equation} \label{GROUP_QUAS_ISOT}
 \alpha  \ast L^{\circ}_{\varepsilon} \mu  = \alpha \circ \mu
\end{equation}
where $\ast$ is  operation of the group $H$, i.e. it is composition of bijections of the set $Q$.

Then from  equality (\ref{GROUP_QUAS_ISOT}) we conclude  that quasigroup $(H, \circ)$ is isotopic to the group $(H, \ast)$, i.e. $(H, \circ) \sim (H, \ast)$, $(H, \circ) = (H, \ast)(\varepsilon, L^{\circ}_{\varepsilon}, \varepsilon)$.

Further we have
\begin{equation} \label{ISOM_LEFT_nl}
\, (H, \ast) \sim \, (H, \circ) \cong (K, \star) \sim (K, \cdot)
\end{equation}

Case 2. The proof of Case 2 is  similar to the proof of Case 1 and we omit it.

Case 3. In the proof of Case 3 instead of equality  (\ref{nucl_IS_GR})
we obtain the following equality
\begin{equation} \label{nucl_IS_GR_R}
\beta \ast (\lambda \circ \mu) = \lambda \circ (\beta \ast \mu)
\end{equation}
where $\beta, \lambda, \mu  \in H \subseteq \,_2N^A_r$. If we put $\mu = \varepsilon$, then we have $\beta \ast \lambda^{\, \prime}  = \lambda \circ \beta$. Notice, since $(H, \ast)$ is a group, then its $(12)$-parastrophe $(H, \overset{(12)}{\ast})$ also is a group. Therefore quasigroup $(H, \circ )$ is isotopic to a group.

Case 4. The proof of this case is  similar to the proof of Case 3.

Case 5. In this case equality  (\ref{nucl_IS_GR}) is transformed in the following equality
\begin{equation} \label{nucl_IS_GR_M}
(\alpha \ast \lambda) \circ (\alpha \ast \mu) = \lambda \circ \mu
\end{equation}
In this case  $H \subseteq \,_1N^A_m$. If we put in equality (\ref{nucl_IS_GR_M}) $\mu = \varepsilon$, then  we have $(\alpha \ast \lambda) \circ \alpha = \lambda^{\, \prime}$, and using operation of left division of quasigroup $(H, \circ)$  further we obtain
$\lambda^{\, \prime} \slash   \alpha = \alpha \ast \lambda $.

If we take $(12)$-parastrophe of the operation $\slash$, then we obtain  $\alpha \overset{(12)}{\slash}   \lambda^{\, \prime} = \alpha \ast \lambda$.
In fact the operation $\overset{(12)}{\slash}$ is $(132)$-parastrophe of the operation $\circ$.

Therefore quasigroup $(H, \overset{(132)}{\circ})$ is a group isotope. Then quasigroup $(H, \circ)$ also is group isotope as a parastrophe of a group isotope  (Lemma \ref{PARASTR_ISOT_GR_ISOT}).

Case 6. The proof of this case  is  similar to the proof of Case 5.
\end{proof}

From Theorem \ref{SUBQUAS_SUB_NUCL_TH} we obtain  the following

\begin{theorem} \label{SUBQUAS_NUCL_TH}
  If an orbit $K$ by the action of the group $_1N^A_l$ ($_3N^A_l$, $_2N^A_r$, $_3N^A_r$, $_1N^A_m$, $_2N^A_m$) on a quasigroup $(Q, \cdot)$ is a subquasigroup of   $(Q, \cdot)$, then $(K, \cdot)$ is an isotope of the group $_1N^A_l$ ($_3N^A_l$, $_2N^A_r$, $_3N^A_r$, $_1N^A_m$, $_2N^A_m$), respectively.
\end{theorem}
\begin{proof}
This theorem is partial case of Theorem \ref{SUBQUAS_SUB_NUCL_TH}. In Case 1 $H = \,{}_1N^A_l$ and so on.
\end{proof}

\begin{theorem}\label{QUAS_NUCL_TR_TH}
 The  group $_1N^A_l$ ($_3N^A_l$, $_2N^A_r$, $_3N^A_r$, $_1N^A_m$, $_2N^A_m$) of a quasigroup $(Q, \cdot)$ acts on the set $Q$ simply transitively if and only if  quasigroup $(Q, \cdot)$ is a group isotope.
 \end{theorem}
\begin{proof} If the group $_1N^A_l$ of a quasigroup $(Q, \cdot)$ acts on the set $Q$ simply transitively, then by Theorem \ref{SUBQUAS_NUCL_TH} $(Q, \cdot)\sim \, _1N^A_l$.

Let $(Q, \cdot)\sim (Q,+)$.
Prove, that in this case the group  $\, _1N^A_l (Q, \cdot)$ acts on the set $Q$ simply transitively.
Any element of left A-nucleus of the group $(Q,+)$ has the form $(L^+_a, \varepsilon, L^+_a)$ for all $a\in Q$. The group  $\, _1N^A_l (Q, +)$ acts on the set $Q$ simply transitively.

If $(Q, \cdot)  = (Q,+) (\alpha, \beta, \gamma)$, then $\, _1N^A_l (Q, \cdot) = \alpha^{-1} \, _1N^A_l (Q, +) \alpha$ (Lemma \ref{ISOM_AUTOT_GR}).

Let $a, b \in Q$. Prove that there exists permutation $\psi \in \, _1N^A_l (Q, \cdot)$ such that $\psi a = b$. We can write permutation $\psi$ in the form $\alpha^{-1} L^+_x \alpha$.
Since the group $\, _1N^A_l (Q, +)$ is transitive on the set $Q$, we can take element $x$ such that $L^+_x \alpha a = \alpha b$.
Then $\psi a = \alpha^{-1} L^+_x \alpha a = \alpha^{-1} \alpha b = b$.

The fact that action of the group $\, _1N^A_l$ on the set $Q$ is semiregular follows from Theorem \ref{GROUP_ACTION_ON QUAS}.

Other cases  are proved in the similar way.
\end{proof}

\begin{definition} \label{TRANSITIVITY_DEF}
   \textit{A-nuclear quasigroup}  is a quasigroup with transitive  action of at least one from its components of A-nuclei.
\end{definition}

Definition \ref{TRANSITIVITY_DEF}  is a generalization of  corresponding definition from  \cite{VD}.

\begin{corollary}
If at least one component of a quasigroup  A-nucleus is transitive, then all  components of quasigroup A-nuclei are transitive.
\end{corollary}
\begin{proof}
The proof follows from Theorem \ref{QUAS_NUCL_TR_TH}.
\end{proof}

We can reformulate Theorem \ref{QUAS_NUCL_TR_TH} in the following form.
\begin{theorem} \label{group_isotope_trans}
A quasigroup is A-nuclear  if and only if it is group isotope.
\end{theorem}

We give slightly other proof of Lemma \ref{NUCLEAR_QUAS}.
\begin{lemma} \label{NUCLEAR_QUAS_I}
If the order of a finite quasigroup $(Q, \cdot)$ is equal to the order of the group ${}_1N^A_l$, then the orders of
groups $N^A_m$, $N^A_r$,  $N^A_l$,  $_3N^A_l$, $_1N^A_m$,
$_2N^A_m$,  $_2N^A_r$ and $_3N^A_r$ are equal to $|Q|$.
\end{lemma}
\begin{proof}
By Theorem \ref{GROUP_ACTION_ON QUAS} the group $N^A_l$ acts on the set $Q$ free (semiregular). From condition of the lemma it follows that the group $N^A_l$ acts on the set $Q$ regular (simply transitively).  Further we can apply  Theorem \ref{QUAS_NUCL_TR_TH}.
\end{proof}

\subsection{Identities with permutation  and group isotopes}

Conditions when a quasigroup is a group isotope were studied in classical article of V.D.~Belousov \cite{vdb1}.
Functional equations on quasigroups  are studied in \cite{FOUR_QUAS_TH, KRAPEZ_80, SOH_05_1, Koval_05, KRAPEZ_10}.
Linearity, one-sided linearity, anti-linearity and one sided anti-linearity of group and abelian group isotopes is studied  in the articles of V.D.~Belousov \cite{vdb1}, T.~Kepka and P.~Nemec \cite{pntk, tkpn}, G.B. Belyavskaya and A.Kh. Tabarov \cite{gbb_TAB_91, Bel_TAB_92, Bel_TAB_94, gbb, TABAR, TABAR_2000}, F.M. Sokhatskii \cite{SOH_99}, J.D.H. Smith \cite{SM_00} and many other mathematicians.

\begin{definition} \label{PERM_IDENTITY}
Let $(Q, \Omega)$ be an algebra. We shall name an identity of algebra  $(Q, \Omega)$ with incorporated in this identity fixed  permutations of the set $Q$  as identity with permutations (permutation identity) \cite{KEED_SCERB}.
\end{definition}

Identity with permutations can be  obtained from "usual" quasigroup identity  by rewriting these identities using quasigroup translations. Notice any permutation of the set $Q$ can be viewed as special kind of unary operation.

Permutation identities in explicit or implicit form are in works of V.D.~Belousov, G.B.~Be\-lyav\-s\-kaya, A.A.~Gvaramiya, A.D.~Keedwell, A.~Krapez,  F.N. Sokhatsky,  A.Kh. Tabarov and many other mathematicians that study  quasigroup identities.

%%%The author thanks Prof. Aleco Alekseevich Gvaramiya that pays his attention to these identities.

We give the following based on Theorem \ref{group_isotope_trans} procedure to answer  the following question: is
a quasigroup with an identity  a group isotope?

\begin{procedure} \label{PROCEDURE}\begin{enumerate}
\item If we can write a permutation  identity  of a quasigroup $(Q, \cdot)$ in the form $A x \cdot y = \Gamma (x\cdot y)$, where $A, \Gamma$ are the sets of permutations of the set $Q$,  and can prove that the set $A$ or the set  $\Gamma$ acts on the set $Q$ transitively, then quasigroup $(Q, \cdot)$ is a group isotope.
\item If we can write a permutation  identity  of a quasigroup $(Q, \cdot)$ in the form $ x \cdot B y = \Gamma (x\cdot y)$, where $B, \Gamma $ are the sets of permutations of the set $Q$, and can prove that the set $B$ or the set $\Gamma$ acts on the set $Q$ transitively, then quasigroup $(Q, \cdot)$ is a group isotope.
\item If we can write a permutation   identity  of a quasigroup $(Q, \cdot)$ in the form $A x \cdot B y = x\cdot y$, where $A, B$ are the sets of permutations of the set $Q$,  and can prove that the set $A$ or the set $B$ acts on the set $Q$ transitively, then quasigroup $(Q, \cdot)$ is a group isotope.
\end{enumerate}
\end{procedure}

A procedure similar to Procedure \ref{PROCEDURE} is given in \cite{SOH_95_I, SOH_95_II, SOH_96_III}.

\begin{remark}
Procedure \ref{PROCEDURE} shows that it is possible to generalize  Definition \ref{PERM_IDENTITY} changing the words "fixed  permutations of the set $Q$" by the words "fixed sets of permutations of the set $Q$".
\end{remark}

We give Belousov criteria, when a group isotope is an abelian group isotope.

\begin{lemma} \label{COMMUT_VDB} Belousov criteria.
If in a group $(Q, +)$  the equality $\alpha x + \beta y = \gamma y + \delta x$ holds for all $x, y \in Q$,
where $\alpha,  \beta, \gamma, \delta$ are some fixed permutations of $Q$, then $(Q, +)$ is an abelian group
\cite{vdb1}.
\end{lemma}
\begin{proof}
From equality $\alpha x + \beta y = \gamma y + \delta x$ we have
\begin{equation} \label{BEL_CRITERIA}
\alpha \delta^{-1} x + \beta \gamma^{-1} y =  y +  x
\end{equation}
 If we put in equality (\ref{BEL_CRITERIA}) $x=0$, then $\beta \gamma^{-1} = L^+_{k}$, where $k = - \alpha \delta^{-1} 0$.

  If we put in equality (\ref{BEL_CRITERIA}) $y=0$, then $\alpha \delta^{-1} = R^+_{d}$, where $d = - \beta \gamma^{-1} 0$.

  We can rewrite equality (\ref{BEL_CRITERIA}) in the form
\begin{equation} \label{BEL_CRITERIA_1}
R^+_{d} x + L^+_{k}y =  y +  x
\end{equation}
If we put $x = y = 0$ in equality (\ref{BEL_CRITERIA_1}), then $d + k = 0$ and equality (\ref{BEL_CRITERIA_1}) takes the form $ x + y =  y +  x$.
\end{proof}

There exists  also  the following corollary from results of F.N.~So\-kha\-ts\-kii
(\cite{SOH_07}, Theorem 6.7.2). Notice up to isomorphism any isotope is principal.

\begin{corollary} \label{SOHA_REZ}
If in a principal group isotope $(Q, \cdot)$ of a group $(Q,+)$    the equality $\alpha x \cdot \beta y = \gamma
y \cdot \delta x$ holds for all $x, y \in Q$, where $\alpha,  \beta, \gamma, \delta$ are some fixed permutations
of $Q$, then $(Q, +)$ is an abelian group.
\end{corollary}
\begin{proof}
If $x\cdot y = \xi x + \chi y$, then we can re-write the equality $\alpha x \cdot \beta y = \gamma y \cdot
\delta x$ in the form $\xi \alpha x +\chi  \beta y = \xi \gamma y + \chi \delta x$. Now we can apply Belousov
criteria (Lemma \ref{COMMUT_VDB}).
\end{proof}

\begin{lemma}
A quasigroup $(Q, \cdot)$ with identity \begin{equation} \alpha_1 x\cdot \alpha_2 (y \cdot  z) = y\cdot \alpha_3(x \cdot \alpha_4 z), \label{(8)}\end{equation}
where $\alpha_1, \dots, \alpha_4$ are permutations of the set $Q$, is abelian group isotope.
\end{lemma}
\begin{proof}
We can re-write identity (\ref{(8)}) in the form $L_{\alpha_1 x}\alpha_2(y \cdot z) = y\cdot \alpha_3 L_x \alpha_4 z $. Since $(Q,\cdot)$  is a quasigroup, then for any fixed elements $a, b \in Q$ the exists an element $t$ such that $ta=b$, i.e.  $L_t a = b$.  Then the set of translations $\{ L_x \, | \, x\in Q  \}$ acts on the set $Q$ transitively.
Therefore the set of permutations of the form $\alpha_3 L_x \alpha_4$ also acts transitively on $Q$.
By Theorem \ref{group_isotope_trans} quasigroup $(Q, \cdot)$ is group isotope.

We can re-write identity (\ref{(8)}) in the form $\alpha_1 x\cdot \alpha_2 R_z y  = y\cdot \alpha_3 R_{\alpha_4 z} x $. From  Corollary \ref{SOHA_REZ} if follows that quasigroup $(Q, \cdot)$ is abelian group isotope.
\end{proof}

\begin{definition}
Let $(Q, \cdot)$ be a groupoid.  Identity of the form
\begin{equation} \label{PERM_MED_ID}
\alpha_1(\alpha_2 x\cdot \alpha_3 y) \cdot \alpha_4 (\alpha_5 u \cdot \alpha_6 v) = \alpha_7(xu) \cdot \alpha_8(yv)
 \end{equation}
 where $\alpha_1, \dots , \alpha_8$ are fixed permutations of the set $Q$, we shall name permutation medial identity.

Identity of the form
\begin{equation}\label{PERM_PARAMED_ID}
\alpha_1(\alpha_2 x\cdot \alpha_3 y) \cdot \alpha_4 (\alpha_5 u \cdot \alpha_6 v) = \alpha_7(vy) \cdot \alpha_8(ux)
 \end{equation}
 where $\alpha_1, \dots , \alpha_8$ are fixed permutations of the set $Q$, we shall name permutation paramedial identity.
\end{definition}

\begin{theorem} \label{PERM_MED_ID_TH}
\begin{enumerate}
\item Permutation medial quasigroup $(Q, \cdot)$ is an abelian group isotope.
\item Permutation paramedial quasigroup $(Q, \cdot)$ is an abelian group isotope.
\end{enumerate}
\end{theorem}
\begin{proof}
Case 1. \label{MED_EX}
 Using language of translations we can rewrite  identity (\ref{PERM_MED_ID}) in the form $\beta_1 x  \cdot \beta_2 v = \beta_3 x \cdot \beta_4 v$, where $\beta_1 = \alpha_1 R_{\alpha_3 y} \alpha_2$, $\beta_2 = \alpha_4 L_{\alpha_5 u} \alpha_6$, $\beta_3 = \alpha_7 R_u$, $\beta_4 = \alpha_8 L_y $.

Then $\beta_1 \beta^{-1}_3 \in \, _1N^A_m$ and the set of permutations of the form $\beta_1 \beta^{-1}_3$ acts on the set $Q$ transitively (we can take $u = a$, where the element $a$ is a fixed element of the set $Q$). Therefore by Theorem
\ref{group_isotope_trans} medial quasigroup is a group isotope.

 We can write identity (\ref{PERM_MED_ID})  in the form   $\gamma_1 y\cdot \gamma_2 u = \gamma_3 u\cdot \gamma_4 y$, where $\gamma_1 = \alpha_1 L_{\alpha_2 x} \alpha_3$, $\gamma_2 = \alpha_4 R_{\alpha_6 v} \alpha_5$, $\gamma_3 = \alpha_7 L_x$, $\gamma_4 = \alpha_8 R_v$.
From   Corollary \ref{SOHA_REZ} it follows that any permutation medial quasigroup is an isotope of abelian group.

Case 2 is proved similarly to Case 1.
\end{proof}

\begin{corollary}
 Medial quasigroup $(Q, \cdot)$ is an abelian group isotope \cite[p.33]{VD}.
 Paramedial quasigroup $(Q, \circ)$ is an abelian group isotope  \cite{pntk}.
\end{corollary}
\begin{proof}
The proof follows from Theorem \ref{PERM_MED_ID_TH}.
\end{proof}

Any permutation identity of the form
\begin{equation}
\alpha_1 (\alpha_2(\alpha_3 x \cdot \alpha_4 y) \cdot \alpha_5 z) = \alpha_6 x  \cdot \alpha_7 (\alpha_8 y \cdot \alpha_9 z)
\end{equation}
on a quasigroup $(Q, \cdot)$, where all $\alpha_i$ are some fixed permutations of the set $Q$,  it is possible to reduce to the following identity \cite{BEL_2007}
\begin{equation} \label{GB_IDENT}
\alpha_1 (\alpha_2(x \cdot y) \cdot z) = \alpha_3 x  \cdot (\alpha_4 y \cdot \alpha_5 z)
\end{equation}

In \cite{BEL_2007} using famous Four quasigroups theorem \cite{AC_BEL_HOS, 2} it is proved:
if a quasigroup $(Q, \cdot)$ satisfies identity  (\ref{GB_IDENT})
then this quasigroup is a group isotope and vice versa.

We can rewrite equality (\ref{GB_IDENT}) in the form
\begin{equation}
\alpha_1 \alpha_2 R_z (x \cdot y) = \alpha_3 x  \cdot R_{\alpha_5 z} \alpha_4 y\end{equation}
If $\alpha_3 = \varepsilon$, then by Procedure \ref{PROCEDURE} Case 2 quasigroup $(Q, \cdot)$ with equality $\alpha_1 \alpha_2 R_z (x \cdot y) =  x  \cdot R_{\alpha_5 z} \alpha_4 y$ is a group isotope.

\section{A-centers of a quasigroup}

\subsection{Classical definitions of center of a loop and a quasigroup}

 R.H.~Bruck defined a center of a loop $(Q,\cdot)$ in the following way  \cite{RHB}.
\begin{definition} \label{BRUCK_DEF_CENTER}
Let $(Q, \cdot)$ be a loop. Then center $Z$ of loop $(Q, \cdot)$ is the following set $Z (Q,\cdot) = N \cap C$, where $C = \{a \in Q \, \mid \, a\cdot x = x\cdot a \,\, \forall x \in Q\}$ \cite{RHB, HOP}.
\end{definition}

We follow \cite{HOP} in the denoting of loop center by the letter $Z$.

Notice,  that in  a loop (even in  a groupoid)
 $Z = N_l \cap N_r \cap C = N_l \cap N_m \cap C = N_m \cap N_r \cap C$  \cite{SCERB_03}.

It is well known that in loop case $Z(Q)$ is normal abelian subgroup of the loop $Q$ \cite{RHB}.

J.D.H.~Smith \cite{SM, JDH_86,  CPS, JDH_2007}
has given definition of center of quasigroup  in the language of universal algebra (Congruence Theory). J.D.H.~Smith  defined central congruence in a quasigroup. Center of a quasigroup is a coset class of central congruence. Notice that any congruence of quasigroup $Q$ defines a subquasigroup of $Q^{\,2}$ \cite{SM}.
G.B.~Belyavskaya and J.D.H.~Smith definitions are close.

A quasigroup $(Q,\cdot)$ is a T-quasigroup if and only if there exists an abelian group $(Q,+)$,  its automorphisms $\varphi$ and $\psi$  and a fixed element $a\in Q$ such that $x\cdot y = \varphi x + \psi y + a$ for all $x, y \in Q$ \cite{pntk}.

A quasigroup is central, if it coincides with its center.

\begin{theorem} \label{GBEL_CENTRAL_TH}
(Be\-lyav\-s\-kaya Theorem). A quasigroup is central (in  Belyavskaya and Smith sense) if and only if it is a T-quasigroup \cite{Bel_94}.
\end{theorem}

An overview of various definitions of quasigroup center is in \cite{SCERB_03}.
\iffalse
Therefore we do not loss in generality, if we shall study only loop case, since any quasigroup is LP-isotopic to a loop.
\fi

\subsection{A-nuclei and autotopy group}

We recall that isotopic quasigroups have isomorphic autotopy groups \cite{VD} (see Lemma \ref{ISOM_AUTOT_GR} of this paper).

It is well known that any quasigroup $(Q, \cdot) $ that is an isotope of a group $Q$  has the following structure of its autotopy group \cite{IVL, vs2}
\[
Avt (Q, \cdot) \cong (Q \times Q) \leftthreetimes Aut(Q)
\]

\begin{theorem} \label{NORM_SUB_AUTOT_GR}
For any  loop  $Q$ we have

\begin{equation*}
 (N^A_l \times N^A_r)  \leftthreetimes Aut(Q) \cong H  \subseteq  Avt (Q)
\end{equation*}
\begin{equation*}
(N^A_l \times N^A_m)  \leftthreetimes Aut(Q) \cong H  \subseteq  Avt (Q)
\end{equation*}
\begin{equation*}
   (N^A_m \times N^A_r)  \leftthreetimes Aut(Q) \cong H  \subseteq  Avt (Q)
\end{equation*}
\end{theorem}
\begin{proof}
Case 1.
It is clear that the groups $N^A_l$, $N^A_r$ and $Aut(Q)$ are subgroup of the group $Avt (Q)$.

Further proof is quit standard for group theory \cite{KM}.  Let $H  = N^A_l \cdot  N^A_r \cdot Aut(Q)$. Here the operation "$\cdot$" is the operation of  multiplication  of triplets of the group $Avt(Q)$.

We demonstrate that $N^A_l \cap  N^A_r = (\varepsilon, \varepsilon,\varepsilon)$.

 Any element of the group $N^A_l \cap  N^A_r$ has the form $(\varepsilon, \varepsilon, \gamma)$. By Corollary \ref{COROLL_OF_LEAKH_TH} $\gamma = \varepsilon$ in this case.

It is easy to see that $N^A_l \unlhd Avt(Q)$. Then $N^A_l \unlhd H$. Similarly $N^A_r \unlhd H$.

Prove that $N^A_l \cdot  N^A_r = N^A_r \cdot  N^A_l$. Let $(\delta, \varepsilon, \lambda) \in N^A_l$, $(\varepsilon,   \mu, \psi) \in N^A_r$.  Any element of the set $N^A_l \cdot  N^A_r$ has the form $(\delta, \mu, \lambda\psi)$ and any element of the set $N^A_r \cdot  N^A_l$ has the form $(\delta, \mu, \psi\lambda)$. By Lemma  \ref{TWO_AUTOTOPY_COMPONENTS} any two autotopy components define the third component in the unique way, therefore $\psi\lambda =  \lambda\psi$ in our case.
Thus $N^A_l \cdot  N^A_r = N^A_l \times  N^A_r$.

Prove that $N^A_l \times  N^A_r \unlhd Avt(Q)$.
It is easy to see that $N^A_l \unlhd Avt(Q)$, $ N^A_r \unlhd Avt(Q)$.
Indeed, in loop case we have
\[
(\alpha^{-1}, \beta^{-1}, \gamma^{-1})  (\delta_1, \varepsilon, \lambda_1)(\alpha, \beta, \gamma) = (\delta_2, \varepsilon, \lambda_2)\] where $(\alpha, \beta, \gamma) \in Avt (Q)$,  $(\delta_1, \varepsilon, \lambda_1), (\delta_2, \varepsilon, \lambda_2) \in N^A_l$.

Similarly
\[
(\alpha^{-1}, \beta^{-1}, \gamma^{-1})  (\varepsilon, \mu_1, \psi_1)(\alpha, \beta, \gamma) = (\varepsilon, \mu_2, \psi_2)\]
Further  we have
\begin{equation}\label{EQUAL_29}
\begin{split}
& (\alpha^{-1}, \beta^{-1}, \gamma^{-1})  (\delta, \mu, \lambda\psi)(\alpha, \beta, \gamma) =  \\
& = (\alpha^{-1}, \beta^{-1}, \gamma^{-1})  (\delta, \varepsilon, \lambda) (\varepsilon, \mu, \psi)(\alpha, \beta, \gamma)=  \\
& = (\alpha^{-1}, \beta^{-1}, \gamma^{-1}) (\delta, \varepsilon, \lambda) (\alpha, \beta, \gamma) (\alpha^{-1}, \beta^{-1}, \gamma^{-1})  (\varepsilon, \mu, \psi) (\alpha, \beta, \gamma)= \\
& = (\delta_1, \varepsilon, \lambda_1) (\varepsilon, \mu_1, \psi_1) = (\delta_1, \mu_1, \lambda_1 \psi_1) \in N^A_l \times  N^A_r
\end{split}
\end{equation}

Therefore $N^A_l \times  N^A_r \unlhd Avt(Q)$ in any quasigroup $Q$.

Prove  that $(N^A_l \times  N^A_r) \cap  Aut(Q) = (\varepsilon, \varepsilon,\varepsilon)$. In this place we shall use
the fact that $Q$ is a loop. In a loop any element of the group $(N^A_l \times  N^A_r)$ has the form $(L_a, R_b, R_b L_a)$. Since any automorphism is an autotopy with equal components we have  $L_a = R_b L_a$, $R_b=\varepsilon$, $R_b =  R_b L_a$, $L_a = \varepsilon$.

Therefore we have  that  $H \cong (N^A_l \times  N^A_r) \leftthreetimes Aut(Q)$.

Cases 2 and 3 are proved similarly to Case 1.
\end{proof}

\begin{remark}
The fact  $N^A_l \cap  N^A_r = (\varepsilon, \varepsilon,\varepsilon)$  demonstrates that there is difference between Garrison left nucleus  $N_l$  and left A-nucleus  $N^A_l$,  between Garrison right nucleus  $N_r$  and right A-nucleus  $N^A_r$, and so on.
\end{remark}

\begin{remark}
Theorem \ref{NORM_SUB_AUTOT_GR} is true for any quasigroup $(Q, \cdot)$  but by isotopy ($(Q, \cdot)\sim (Q, +)$), in general, the group $Aut(Q, +)$ passes in a subgroup of the group $Avt(Q, \cdot)$ \cite{vs2}.
\end{remark}

\begin{corollary} \label{NORM_DIR_PROD}
For any  quasigroup $Q$ we have
\begin{equation*}
 (N^A_l \times N^A_r)  \unlhd  Avt (Q)
\end{equation*}
\begin{equation*}
(N^A_l \times N^A_m)  \unlhd  Avt (Q)
\end{equation*}
\begin{equation*}
   (N^A_m \times N^A_r)  \unlhd Avt (Q)
\end{equation*}
\end{corollary}
\begin{proof}
The proof follows from the proof of Theorem \ref{NORM_SUB_AUTOT_GR} (equality (\ref{EQUAL_29})).
\end{proof}

\begin{corollary} \label{coroll_Perm_el}
For any  quasigroup $Q$ we have:
\begin{itemize}
\item [1.] if   $\mu \in {}_1N^A_l$, $\nu \in {}_1N^A_m$, then $\mu \nu = \nu \mu$;
\item [2.] if   $\mu \in {}_2N^A_r$, $\nu \in {}_2N^A_m$, then $\mu \nu = \nu \mu$;
\item [3.] if   $\mu \in {}_3N^A_l$, $\nu \in {}_3N^A_r$, then $\mu \nu = \nu \mu$.
\end{itemize}
\end{corollary}
\begin{proof}
The proof follows from Theorem \ref{NORM_SUB_AUTOT_GR}.
\end{proof}

\begin{corollary}
In any  quasigroup
\begin{itemize}
\item[1.] ${}_1N^A_l \cdot {}_1N^A_m = {}_1N^A_m \cdot {}_1N^A_l$;
\item[2.] $ {}_2N^A_r \cdot {}_2N^A_m = {}_2N^A_m \cdot {}_2N^A_r$;
\item[3.] $ {}_3N^A_r \cdot {}_3N^A_l = {}_3N^A_l \cdot {}_3N^A_r$.
\end{itemize}
\end{corollary}
\begin{proof}
The proof follows from Corollary  \ref{coroll_Perm_el}.
\end{proof}

\begin{example}
Loop $(Q, \cdot)$ has   identity group $Aut(Q)$ \cite{SHCHUKIN_08}, the  identity groups $(N^A_l \times  N^A_r) \leftthreetimes Aut(Q)$,
$(N^A_l \times  N^A_m) \leftthreetimes Aut(Q)$,$(N^A_r \times  N^A_r) \leftthreetimes Aut(Q)$. Its autotopy group $Avt(Q)$ is isomorphic to the alternating group $A_4$ \cite{HALL}.
\[
\begin{array}{c|ccccc}
\cdot & 1 & 2 & 3 & 4 & 5 \\
\hline
1 & 1 & 2 & 3 & 4 & 5 \\
2 & 2 & 3 & 1 & 5 & 4 \\
3 & 3 & 4 & 5 & 1 & 2 \\
4 & 4 & 5 & 2 & 3 & 1 \\
5 & 5 & 1 & 4 & 2 & 3
\end{array}
\]
We recall  $|A_4| = 12$ and this number  is  divisor of the number $5! \cdot 5 = 600$.
\end{example}

\subsection{A-centers of a loop}

\begin{definition} \label{AUTOT_CENTER_LOOP}
Let $Q$ be a loop.

Autotopy  of the form $\{ (L_a,  \varepsilon, L_a) \, | \,  a \in Z(Q)\}$ we shall name left central autotopy. Group of all left central autotopies  we shall denote  by $Z^A_l$.

Autotopy  of the form $\{ (\varepsilon, L_a,  L_a) \, | \,  a \in Z(Q)\}$ we shall name right  central autotopy. Group of all  right  central autotopies  we shall denote by $Z^A_r$.

 Autotopy  of the form $\{ (L_a, L^{-1}_a, \varepsilon) \, | \,  a \in Z(Q)\}$ we shall name  middle central autotopy. Group of all middle central autotopies  we shall denote  by $Z^A_m$.
\end{definition}

\begin{lemma} \label{ISOM_OF_CENTERS_OF_LOOP}
In any loop the groups $Z^A_l$, $Z^A_r$, $Z^A_m$,  ${}_1 Z^A_l$, ${}_3 Z^A_l$, ${}_2 Z^A_r$, ${}_3 Z^A_r$, ${}_1 Z^A_m$, ${}_2 Z^A_m$, and $Z$ are isomorphic. In more details  $Z \cong Z^A_l \cong Z^A_r \cong Z^A_m \cong {}_1 Z^A_l = {}_3 Z^A_l = {}_2 Z^A_r = {}_3 Z^A_r = {}_1 Z^A_m = {}_2 Z^A_m$.
\end{lemma}
\begin{proof}
The proof follows from definition  of center of  a loop (Definition \ref{BRUCK_DEF_CENTER}) and Definition  \ref{AUTOT_CENTER_LOOP}. The map $\xi: a \mapsto L_a$, where $a\in Z$, gives necessary isomorphism of the group $Z$ and ${}_1 Z^A_l$. And so on.
\end{proof}

\begin{lemma}
If $Q$ is a loop, then ${}_1Z^A_l 1 = {}_3Z^A_l =  {}_2Z^A_r 1 = {}_3Z^A_r = {}_1Z^A_m = {}_2Z^A_m 1 = Z$.
\end{lemma}
\begin{proof}
It is possible to use   Definition  \ref{AUTOT_CENTER_LOOP}.
\end{proof}

\begin{corollary}
In any loop $Q$ the orbit ${}_1Z^A_l 1$ (${}_3Z^A_l$, ${}_2Z^A_r 1$, ${}_3Z^A_r$, ${}_1Z^A_m$, ${}_2Z^A_m 1$) coincides with the set $Q$ if and only if $Q$ is an abelian group.
\end{corollary}

\begin{theorem} \label{INTERS_PAIR_AND_THE_THIRD}
Let $Q$ be a loop. Then
\begin{itemize}
\item [1.]   $(N^A_l \times N^A_r) \cap N^A_m = \{ (L_a, L^{-1}_a, \varepsilon) \, | \,  a \in Z(Q)\} = Z^A_m$;
\item [2.]   $(N^A_l \times N^A_m) \cap N^A_r = \{ ( \varepsilon, L_a, L_a) \, | \,  a \in Z(Q)\} = Z^A_r$;
\item [3.]   $(N^A_r \times N^A_m) \cap N^A_l = \{ (L_a, \varepsilon, L_a) \, | \,  a \in Z(Q)\} = Z^A_l$.
\item [4.]   $Z^A_l\trianglelefteq Avt(Q), Z^A_r\trianglelefteq Avt(Q), Z^A_m\trianglelefteq Avt(Q)$.
\item [5.]   The groups $Z^A_l, Z^A_r$ and $Z^A_m$ are abelian subgroups of the group $ Avt(Q)$.
\end{itemize}
\end{theorem}
\begin{proof} Case 1.
Any element  of the group $N^A_l \times N^A_r$ has the form
\begin{equation} \label{LEFT_RIGHT_AUTT}
(L_a, R_b, L_aR_b)
\end{equation}
 any element  of the group $N^A_m$ has the form
\begin{equation} \label{MIDDL_AUTT}
(R_c, L^{-1}_c, \varepsilon)
\end{equation}
 (Theorem \ref{A-NUCLEI_OF_LOOPS}). Then $R_b L_a = \varepsilon$, $R_b = L^{-1}_a$. Expression (\ref{LEFT_RIGHT_AUTT}) takes the form $(L_a, L^{-1}_a, L_a L^{-1}_a) = (L_a, L^{-1}_a, \varepsilon)$. Comparing it with expression (\ref{MIDDL_AUTT}) we have  that $L_a=R_a$, i.e. $ax = xa$ for any element $x\in Q$. In addition taking into consideration that $(L_a, L^{-1}_a, \varepsilon) \in N^A_m, (L_a, \varepsilon, L_a) \in N^A_l, (\varepsilon, L^{-1}_a, L^{-1}_a) \in N^A_r$, we obtain that the element $a$ is central, i.e. $a\in Z(Q)$.

Therefore  $(N^A_l \times N^A_r)\cap N^A_m \subseteq  \{ (L_a, L^{-1}_a, \varepsilon) \, | \,  a \in Z(Q)\}$.

Converse. It is easy to see that any element from the center $Z(Q)$
 of a loop $Q$ "generate" autotopies of left, right, middle A-nucleus of the loop $Q$, i.e.
 $\{ (L_a, L^{-1}_a, \varepsilon) \, | \,  a \in Z(Q)\} \subseteq (N^A_l \times N^A_r)\cap N^A_m $.

 And finally we have  $(N^A_l \times N^A_r)\cap N^A_m = \{ (L_a, L^{-1}_a, \varepsilon) \, | \,  a \in Z(Q)\} $.

Case 2. The proof of Case 2 is similar to the proof of Case 1 and we omit some details.
Any element  of the group $N^A_l \times N^A_m$ has the form
\begin{equation} \label{LEFT_MIDDLL_AUTT}
(L_a R_b, L^{-1}_b, L_a)
\end{equation}
 any element  of the group $N^A_r$ has the form
\begin{equation} \label{RIGHT_AUTT}
( \varepsilon, R_c, R_c)
\end{equation}
 (Theorem \ref{A-NUCLEI_OF_LOOPS}). Since expressions (\ref{LEFT_MIDDLL_AUTT}) and (\ref{RIGHT_AUTT}) in the intersection should be componentwise equal, we have the following system of equations
 \[
 \left\{
 \begin{array}{l}
L_a R_b  = \varepsilon \\
L^{-1}_b = R_c\\
L_a = R_c
\end{array}
\right.
 \]

We remember that in any loop for any nuclear translation we have $L^{-1}_a = L_{a^{-1}} $, $R^{-1}_a = R_{a^{-1}} $  (Theorem \ref{A-NUCLEI_OF_LOOPS}).

From equality $L^{-1}_b = L_a$ we have that $a = b^{-1}$ and $b = a^{-1}$. Then we from  equality
$L_a R_b  = \varepsilon$ it follows that  $L_b = R_b$ for all elements of the group $(N^A_l \times N^A_m) \cap N^A_r$.

  Expression (\ref{LEFT_MIDDLL_AUTT}) takes the form $(L_a L^{-1}_a, L_a,  L_a) = ( \varepsilon, L_a, L_a)$. Expression (\ref{MIDDL_AUTT}) takes the same form.

Taking into consideration that $L_a = R_a$,  $(L^{-1}_a, L_a, \varepsilon) \in N^A_m, (L_a, \varepsilon, L_a) \in N^A_l, (\varepsilon, L_a, L_a) \in N^A_r$, we obtain that the element $a$ is central, i.e. $a\in Z(Q)$.

Therefore  $(N^A_l \times N^A_r)\cap N^A_m \subseteq  \{ (L_a, L^{-1}_a, \varepsilon) \, | \,  a \in Z(Q)\}$.

Converse inclusion is proved similarly with the proof in Case 1.

Finally we have  $(N^A_l \times N^A_m)\cap N^A_r = \{ (\varepsilon, L_a, L_a) \, | \,  a \in Z(Q)\} $.

Case 3. The proof of Case 3 is similar to the proof of Cases 1, 2.
Any element  of the group $N^A_r \times N^A_m$ has the form
\begin{equation} \label{RIGHT_MIDDLL_AUTT}
(R_b, R_a L^{-1}_b, R_a)
\end{equation}
 any element  of the group $N^A_l$ has the form
\begin{equation} \label{LEFT_AUTT}
(L_c,  \varepsilon, L_c)
\end{equation}
 (Theorem \ref{A-NUCLEI_OF_LOOPS}). Since expressions (\ref{RIGHT_MIDDLL_AUTT}) and (\ref{LEFT_AUTT}) in the intersection should be equal componentwise, we have the following system of equations
 \[
 \left\{
 \begin{array}{l}
R_b  = L_c \\
R_a L^{-1}_b = \varepsilon\\
R_a = L_c
\end{array}
\right.
 \]

From equality $R_b = R_a$ we have that $a = b$. Then we from  equality
$R_a L^{-1}_b  = \varepsilon$ it follows that  $L_b = R_b$ for all elements of the group $(N^A_r \times N^A_m) \cap N^A_l$.

  Expression (\ref{RIGHT_MIDDLL_AUTT}) takes the form $(L_a,   L_a L^{-1}_a,  L_a) = (L_a,  \varepsilon, L_a)$.

Taking into consideration that $L_a = R_a$,  $(L_a, L^{-1}_a, \varepsilon) \in N^A_m, (L_a, \varepsilon, L_a) \in N^A_l, (\varepsilon, L_a, L_a) \in N^A_r$, we obtain that the element $a$ is central, i.e. $a\in Z(Q)$.

Therefore  $(N^A_r \times N^A_m)\cap N^A_l =  \{ (L_a,  \varepsilon, L_a) \, | \,  a \in Z(Q)\}$.

Case 4. The proof follows from Corollary \ref{NORM_DIR_PROD} and the fact that intersection of normal subgroups of a group is normal subgroup.

Case 5. The proof follows from Lemma \ref{ISOM_OF_CENTERS_OF_LOOP}. Indeed, in this case, if we denote the loop operation by $+$, then we have $L_aL_b = L_{a+b} = L_{b+a} = L_bL_a$.
\end{proof}

\begin{theorem}
Let $Q$ be a loop, $K= N^A_l \cdot  N^A_r \cdot N^A_m$. Then
\begin{itemize}
\item[(i)] $K/Z^A_l \cong (N^A_r\times N^A_m)/Z^A_l\times N^A_l/Z^A_l$;
\item[(ii)] $K/Z^A_r \cong (N^A_l\times N^A_m)/Z^A_r\times N^A_r/Z^A_r$;
\item[(iii)] $K/Z^A_m \cong (N^A_l\times N^A_r)/Z^A_m\times N^A_m/Z^A_m$.
\end{itemize}
\end{theorem}
\begin{proof} Case (i).
The groups $N^A_l$, $N^A_r$ and $N^A_m$ are normal subgroups of the group $Avt(Q)$.
From Theorem \ref{NORM_SUB_AUTOT_GR} we have that  the groups $N^A_l$, $N^A_r$ and $N^A_m$ are intersected in pairs by identity group.

From Lemma \ref{NORMALITY_OF_A_NUCL} and  Corollary \ref{NORM_DIR_PROD} we have that  $N^A_l  \unlhd  Avt (Q)$, $(N^A_r \times N^A_m)  \unlhd  Avt (Q)$. Then $N^A_l  \unlhd  K$, $(N^A_r \times N^A_m)  \unlhd K$.

From Theorem \ref{INTERS_PAIR_AND_THE_THIRD} and Theorem 4.2.1 \cite[p. 47]{KM} we have that  groups $(N^A_r\times N^A_m)/Z^A_l$ and $N^A_l/Z^A_l$ are normal subgroups of the group $K/Z^A_l$ that are intersected by the identity subgroup. Therefore $K/Z^A_l \cong (N^A_r\times N^A_m)/Z^A_l\times N^A_l/Z^A_l$.

Cases (ii) and (iii) are proved in the similar way.
\end{proof}

\subsection{A-centers of a quasigroup}

\begin{definition}
Let $Q$ be a quasigroup. We shall name
\begin{itemize}
\item[(i)] the group   $(N^A_r \times N^A_m) \cap N^A_l$  as left   A-center of  $Q$ and shall denote by $Z^A_l$;
\item[(ii)] the group  $(N^A_l \times N^A_m) \cap N^A_r$ as  right  A-center of  $Q$ and shall denote by $Z^A_r$;
\item[(iii)] the group   $(N^A_l \times N^A_r) \cap N^A_m$ as  middle A-center of $Q$ and shall denote by $Z^A_m$.
\end{itemize}
\end{definition}

\begin{lemma} \label{IZOTOPES_OF_QUS_CENTER}
Left, right, and middle A-centers and their components  of isotopic  quasigroups $(Q, \circ)$ and $(Q, \cdot)$ are isomorphic, i.e., if $(Q, \circ) \sim (Q, \cdot)$,
then
\[
\begin{array}{llllll}
1. & Z^A_l(Q, \cdot) \cong Z^A_l(Q, \circ); & 2. & {}_1Z^A_l(Q, \cdot) \cong {}_1Z^A_l(Q, \circ); & 3. & {}_3Z^A_l(Q, \cdot) \cong {}_3Z^A_l(Q, \circ);\\
4. & Z^A_r(Q, \cdot) \cong Z^A_r(Q, \circ); & 5. & {}_2Z^A_r(Q, \cdot) \cong {}_2Z^A_r(Q, \circ); & 6. & {}_3Z^A_r(Q, \cdot) \cong {}_3Z^A_r(Q, \circ);\\
7. & Z^A_l(Q, \cdot) \cong Z^A_l(Q, \circ); & 8. & {}_1Z^A_m(Q, \cdot) \cong {}_1Z^A_m(Q, \circ); & 9. & {}_2Z^A_m(Q, \cdot) \cong {}_2Z^A_m(Q, \circ);\\
\end{array}
\]
\end{lemma}
\begin{proof}
Case 1. Suppose that  $(Q, \circ) = (Q, \cdot)T$, where $T$ is an isotopy. By Lemma \ref{ISOM_AUTOT_GR}  $Avt(Q, \circ) = T^{-1} Avt (Q, \cdot) T$.

Notice,  $T^{-1}Z^A_l(Q, \cdot)T =  T^{-1}((N^A_r (Q, \cdot) \times N^A_m(Q, \cdot)) \cap N^A_l(Q, \cdot))T$.
For short below we shall omit denotation of quasigroups $(Q, \cdot)$ and $(Q, \circ)$.

Denote $(N^A_r \times N^A_m)$ by $A$, $N^A_l$ by $B$, the
conjugation isomorphism by the letter $\varphi$.
Prove that $\varphi (A\cap B) = \varphi A \cap \varphi B$.

Let $\varphi x\in \varphi(A\cap B)$. Then $x \in A\cap B$ since $\varphi$ is a bijective map. If $x\in A\cap B$, then $x\in A$, $x\in B$,
$\varphi (x) \in \varphi(A) $,  $\varphi(x)\in \varphi(B)$, $\varphi(x) \in \varphi(A)\cap \varphi(B)$, $\varphi(A\cap B) \subseteq \varphi
A\cap \varphi B$.

Let $\varphi x\in \varphi A\cap \varphi B$. Then $\varphi x \in \varphi A$, $\varphi x \in \varphi B$. Since $\varphi$ is a bijective map, then
$x\in A$, $x\in B$,  $x\in A \cap B$, $\varphi x\in \varphi(A\cap B)$,   $\varphi A\cap \varphi B \subseteq \varphi(A\cap B)$.

Finally, $\varphi
A\cap \varphi B = \varphi(A\cap B)$, i.e. $T^{-1}((N^A_r \times N^A_m) \cap N^A_l)T = T^{-1}(N^A_r \times N^A_m)T \cap T^{-1}N^A_l T$.

The fact, that $T^{-1}(N^A_r \times N^A_m)T = T^{-1}N^A_rT \times T^{-1}N^A_mT$ is well known.  Indeed, in the group of all triplets $S_Q\times S_Q \times S_Q$ isomorphic image of the direct product of two subgroups  is the  direct product of their isomorphic images. Also it is possible to prove this fact  using equality  (\ref{EQUAL_29}).

We have the following chain of equalities
\begin{equation*}
\begin{split}
& T^{-1}Z^A_l(Q, \cdot)T =  T^{-1}((N^A_r \times N^A_m) \cap N^A_l)T = \\ & T^{-1}(N^A_r \times N^A_m)T \cap T^{-1} N^A_lT  = (( T^{-1}N^A_r T) \times (T^{-1} N^A_m T)) \cap T^{-1} N^A_lT = \\ & (N^A_r (Q, \circ) \times N^A_m (Q, \circ)) \cap N^A_l(Q, \circ)  = Z^A_l(Q, \circ)
\end{split}
\end{equation*}

Therefore $  Z^A_l(Q, \cdot)\cong Z^A_l(Q, \circ)$.

Cases 2, 3 follows from Case 1. Cases 4 and 7 are proved in the similar way with Case 1. Cases 5, 6 and 8, 9 follow from Cases 2 and 3, respectively.
\end{proof}

\begin{corollary} \label{COR_Z_GROUPS_ABEL}
In any quasigroup $(Q, \cdot)$  the groups $Z^A_l$, $Z^A_r$, $Z^A_m$,  ${}_1 Z^A_l$, ${}_3 Z^A_l$, ${}_2 Z^A_r$, ${}_3 Z^A_r$, ${}_1 Z^A_m$, ${}_2 Z^A_m$ are isomorphic  abelian groups, i.e.
\begin{equation*}
\begin{split}
& Z^A_l(Q, \cdot) \cong Z^A_r(Q, \cdot) \cong Z^A_m(Q, \cdot) \cong \\& {}_1 Z^A_l(Q, \cdot)  \cong {}_3 Z^A_l(Q, \cdot) \cong {}_2 Z^A_r(Q, \cdot) \cong \\& {}_3 Z^A_r(Q, \cdot) \cong {}_1 Z^A_m(Q, \cdot) \cong {}_2 Z^A_m(Q, \cdot) \end{split}
\end{equation*}
\end{corollary}
\begin{proof}
Suppose that quasigroup $(Q,\cdot)$ is isomorphic  to a loop $(Q, \circ)$. Then by Lemma   \ref{IZOTOPES_OF_QUS_CENTER}   $Z^A_l(Q, \cdot) \cong Z^A_l(Q, \circ)$;   $Z^A_r(Q, \cdot) \cong Z^A_r(Q, \circ)$;
 $Z^A_m(Q, \cdot) \cong Z^A_m(Q, \circ)$ and so on.

 By Lemma \ref{ISOM_OF_CENTERS_OF_LOOP} the groups $Z^A_l(Q, \circ)$, ${}_1Z^A_l(Q, \circ)$, ${}_3Z^A_l(Q, \circ)$,  $ Z^A_r(Q, \circ)$, ${}_2Z^A_r(Q, \circ)$, ${}_3Z^A_r(Q, \circ)$, $Z^A_m(Q, \circ)$, ${}_1Z^A_m(Q, \circ)$, ${}_2Z^A_m(Q, \circ)$ are isomorphic  abelian groups. Therefore
the same is true for A-centers of any quasigroup.
\end{proof}

\begin{theorem} \label{SUBQUAS_SUB_ZENTER_TH}
   If an orbit $K$ by the action of a subgroup $H$ of the  group $_1Z^A_l(Q, \cdot)$ ($_3Z^A_l(Q, \cdot)$, $_2Z^A_r(Q, \cdot)$, $_3Z^A_r(Q, \cdot)$, $_1Z^A_m(Q, \cdot)$, $_2Z^A_m(Q, \cdot)$) on the set $Q$ is a subquasigroup of quasigroup  $(Q, \cdot)$, then $(K, \cdot)$ is an isotope of abelian group $H$.
\end{theorem}
\begin{proof} From Corollary \ref{COR_Z_GROUPS_ABEL} it follows that the group $_1Z^A_l$ is abelian. Then any subgroup $H$ of group $_1Z^A_l$ also is abelian.
The fact that $(K, \cdot)$ is  group isotope follows from definition of left A-center $Z^A_l$ and Theorem   \ref{SUBQUAS_SUB_NUCL_TH}.

Other cases  are proved in the similar way.
\end{proof}

\begin{corollary} \label{ACTION_OF_ZENTER}
    If an orbit $K$ by the action of   group $_1Z^A_l(Q, \cdot)$ ($_3Z^A_l(Q, \cdot)$, $_2Z^A_r(Q, \cdot)$, $_3Z^A_r(Q, \cdot)$, $_1Z^A_m(Q, \cdot)$, $_2Z^A_m(Q, \cdot)$)  on the set $Q$ is a subquasigroup of quasigroup  $(Q, \cdot)$, then $(K, \cdot)$ is an isotope of abelian group $_1Z^A_l(Q, \cdot)$ ($_3Z^A_l(Q, \cdot)$, $_2Z^A_r(Q, \cdot)$, $_3Z^A_r(Q, \cdot)$, $_1Z^A_m(Q, \cdot)$, $_2Z^A_m(Q, \cdot)$, respectively).
\end{corollary}
\begin{proof}
This corollary is partial case of Theorem \ref{SUBQUAS_SUB_ZENTER_TH}. For example in Case 1 $ H= {}_1Z^A_l(Q, \cdot)$.
\end{proof}

\begin{theorem}\label{QUAS_ZENTR_TH}
 The  group $_1Z^A_l$ ($_3Z^A_l$, $_2Z^A_r$, $_3Z^A_r$, $_1Z^A_m$, $_2Z^A_m$) of a quasigroup $(Q, \cdot)$ acts on the set $Q$ simply transitively if and only if  quasigroup $(Q, \cdot)$ is abelian group isotope.
\end{theorem}
\begin{proof}
If the group $_1Z^A_l$ of a quasigroup $(Q, \cdot)$ acts on the set $Q$  transitively, then the orbit by the action of group $_1Z^A_l$ coincides with the set $Q$, and by Corollary \ref{ACTION_OF_ZENTER} $(Q, \cdot)\sim \, _1Z^A_l$.

Let $(Q, \cdot)\sim (Q,+)$, where $(Q, +)$ is abelian group.
Prove, that in this case the group  $\, _1Z^A_l (Q, \cdot)$ acts on the set $Q$ simply transitively.
Any element of left A-center  of the group $(Q,+)$ has the form $(L^+_a, \varepsilon, L^+_a)$ for all $a\in Q$.

It is clear that  group  $\, _1Z^A_l (Q, +)$ acts on the set $Q$ simply transitively.

If $(Q, \cdot)  = (Q,+) (\alpha, \beta, \gamma)$, then $\, _1Z^A_l (Q, \cdot) = \alpha^{-1} \, _1Z^A_l (Q, +) \alpha$ (Corollary \ref{COROLL_VDB_LEMMA}).

Prove that action of the  group $\, _1Z^A_l (Q, \cdot)$  on the set $Q$ is transitive.
Let $a, b \in Q$. Prove that there exists permutation $\psi \in \, _1Z^A_l (Q, \cdot)$ such that $\psi a = b$. We can write permutation $\psi$ in the form $\alpha^{-1} L^+_x \alpha$.
Since the group $\, _1Z^A_l (Q, +)$ is transitive on the set $Q$, we can take element $x$ such that $L^+_x \alpha a = \alpha b$.
Then $\psi a = \alpha^{-1} L^+_x \alpha a = \alpha^{-1} \alpha b = b$.

The fact that action of the group $\, _1Z^A_l$ on the set $Q$ has identity stabilizers for any element $x\in Q$  follows from Theorem \ref{GROUP_ACTION_ON QUAS} since in any abelian group $\, _1N^A_l = \, _1Z^A_l$.

Other cases are proved in the similar way.
\end{proof}

\begin{corollary}
If at least one component of quasigroup  A-center is transitive, then all  components of A-centers are transitive.
\end{corollary}
\begin{proof}
The proof follows from Theorem \ref{QUAS_ZENTR_TH}.
\end{proof}

\begin{definition} \label{A_CENTRA_DEF}
   \textit{A-central quasigroup}  is a quasigroup $Q$ with transitive  action of at least one from its  components of A-centers on $Q$.
\end{definition}

Using this definition we formulate main theorem of this paper
\begin{theorem} \label{group_isotope_trans_CENTR}
A quasigroup is A-central  if and only if it is abelian group isotope.
\end{theorem}
\begin{proof}
We can use Theorem \ref{QUAS_ZENTR_TH}.
\end{proof}

\begin{corollary}
Any central quasigroup in G.B.~Belyavskaya and J.D.H.~Smith sense is A-central quasigroup.
\end{corollary}
\begin{proof}
The proof follows from Theorems \ref{GBEL_CENTRAL_TH} and \ref{group_isotope_trans}. Any $T$-quasigroup is an isotope of  abelian group.
\end{proof}

\section{Normality of A-nuclei and A-centers}

\subsection{Normality of  equivalences in quasigroups}

A binary relation on a set $Q$ is any subset of the set $Q\times Q$.

\begin{definition} A relation $\theta$ on a set $S$ satisfying the reflexive ($x\theta x$), symmetric ($x\theta y \Leftrightarrow y\theta x $), and
transitive ($x\theta y, y\theta z \Longrightarrow  x\theta z$)
properties is an equivalence relation on $S$. Each cell $\bar a = \theta (a)$ in the natural partition given by an
equivalence  relation is an equivalence class \cite{FR}.
\end{definition}

\begin{definition} \label{UNIV_AKGEBR_QUAS_DEF}
 Let $A$ be an algebra of type $\mathbb F$ and let $\theta$ an equivalence.  Then $\theta$ is a congruence
on A if $\theta$  satisfies the following compatibility property:
 For each n-ary function symbol $f\in \mathbb F$ and elements $a_i, b_i \in A$, if $a_i\theta b_i$ holds for
$1\leq i \leq n$ then $f^A(a_1,  \dots a_n)\theta f^A(b_1,  \dots b_n)$
holds \cite{BURRIS}.
\end{definition}
"The compatibility property is an obvious condition for introducing an algebraic structure
on the set of equivalence classes $A\slash \theta$  an algebraic structure which is inherited from the
algebra $A$" \cite{BURRIS}.

Notice there exists a quasigroup $(Q,\cdot)$  in the sense of Definition  \ref{def2} ("existential quasigroup") and  its congruence $\theta$  in
the  sense of Definition \ref{UNIV_AKGEBR_QUAS_DEF} such that $Q\slash \theta$ ("an algebraic structure which is inherited from the algebra
$A$") is a division groupoid which, in general, is not  a quasigroup \cite{BK, SHCH_BR_BL_05}. Therefore for quasigroups congruence and homomorphism theory has a little non-standard character from the point of view of universal-algebraic approach \cite{BURRIS}.

\begin{definition} \label{DEF_CONGR_QUAS}
An equivalence $\theta$ is a left congruence of a groupoid  $(Q, \circ)$, if the following implication is  true for
all  $x, y, z \in Q$: $x \, \theta \, y  \Longrightarrow (z\circ x)\, \theta \, (z\circ y)$.
In other words equivalence  $\theta$ is stable relative to any left translation of $(Q,\circ)$.

An equivalence $\theta$ is a right  congruence of a groupoid  $(Q, \circ)$, if the following implication is  true for
all  $x, y, z \in Q$: $ x\, \theta \, y
\Longrightarrow (x\circ z)\, \theta \, (y\circ z)$.
In other words equivalence  $\theta$ is stable relative to any right translation  of $(Q,\circ)$ \cite{PC}.
\end{definition}

\begin{definition}\label{NORM_CONGR_QUAS} An equivalence  $\theta$ of a groupoid  $(Q,\circ)$ is called cancellative from the left, if the following implication is
true for all  $x, y, z \in Q$: $(z\circ x) \, \theta \, (z\circ y)  \Longrightarrow x\, \theta \, y$.

An equivalence  $\theta$ of a groupoid  $(Q,\circ)$ is called cancellative from the right, if the following implication is true for all  $x, y, z \in Q$: $(z\circ x) \, \theta \, (z\circ y)  \Longrightarrow x\, \theta \, y, (x\circ z)
\, \theta \ (y\circ z) \Longrightarrow x\, \theta \, y$ \cite{VD, 1a}.

An equivalence  $\theta$ of a groupoid  $(Q,\circ)$ is called normal from the right (right normal), if it is stable and cancellative from the right.

An equivalence  $\theta$ of a groupoid  $(Q,\circ)$ is called normal from the left (left normal), if it is stable and cancellative from the left.

An equivalence  $\theta$ of a groupoid  $(Q,\circ)$ is called normal, if it is left and right normal.
\end{definition}

\begin{remark}
In  \cite{SCERB_07} there is slightly other approach to normality of  groupoid congruences.
\end{remark}

\begin{definition}
 If $\theta$ is a binary relation on a  set $Q$, $\alpha$ is a permutation of the set $Q$ and from $x\theta y$
it follows  $\alpha x \theta \alpha y$  and  $\alpha^{-1} x \theta \alpha^{-1} y$ for all $(x,y) \in \theta$,
then we shall say  that the permutation $\alpha$ is an \textit{admissible} permutation relative to the binary
relation $\theta$ \cite{VD}.
\end{definition}
\index{Binary relation!admissible}

Moreover,  we shall say that a binary relation $\theta$ admits a permutation $\alpha$.

One from the most important properties of  e-quasigroup  $(Q, \cdot, \backslash, /)$  is the following property.
\begin{lemma} \label{CONGR_PRIM_QUAS}
Any congruence of a quasigroup $(Q, \cdot, \backslash, /)$  is a normal congruence of quasigroup $(Q, \cdot)$;
any normal congruence of a quasigroup $(Q, \cdot)$ is a congruence of  quasigroup $(Q, \cdot, \backslash, /)$
\cite{BIRKHOFF, MALTSEV,  VD, 1a}.
\end{lemma}

\begin{lemma} \label{NORM_CONGR_ADMISSBLE_REL_TRANS}
Any normal quasigroup congruence of a quasigroup $(Q, \cdot)$ is admissible relative to  any left, right and middle  quasigroup translation
\cite{BELAS}.
\end{lemma}
\begin{proof}
The fact that any normal quasigroup congruence is admissible relative to  any left and right  quasigroup
translation follows from Definitions \ref{DEF_CONGR_QUAS} and \ref{NORM_CONGR_QUAS}.

Let $\theta$ be a normal congruence of a quasigroup $(Q,\cdot)$. Prove the following implication
\begin{equation} \label{IMPLICATION_1}
a\theta b \rightarrow P_c a \, \theta \, P_c b
\end{equation}
 If $P_c a = k$, then $a\cdot k = c$, $k = a\backslash c$, $k =
R^{\backslash}_c a$. Similarly  if $P_c b = m$, then $b\cdot m = c$, $m = b\backslash c$, $m = R^{\backslash}_c
b$. Since $\theta$ is a congruence of quasigroup $(Q, \cdot, \backslash, /)$ (Lemma \ref{CONGR_PRIM_QUAS}), then
implication (\ref{IMPLICATION_1}) is true.

Implication \begin{equation}\label{IMPLICATION_2}
 a\theta b \rightarrow P^{-1}_c a \, \theta \, P^{-1}_c b
\end{equation} is proved in the similar way.
 If $P^{-1}_c a = k$, then $k \cdot a = c$, $k = c\slash a$, $k =
L^{\slash}_c a$. Similarly  if $P^{-1}_c b = m$, then $m\cdot b = c$, $m = c\slash b$, $m = L^{\slash}_c b$.
Since $\theta$ is a congruence of quasigroup $(Q, \cdot, \backslash, /)$ (Lemma \ref{CONGR_PRIM_QUAS}), then
implication (\ref{IMPLICATION_2}) is true.
\end{proof}
\begin{corollary} \label{PARASTROPHY_INVAR_NORM_CONG}
If $\theta$ is a normal quasigroup congruence of a quasigroup $Q$, then $\theta$ is a normal congruence of any
parastrophe  of $Q$  \cite{BELAS}.
\end{corollary}
\begin{proof}
The proof follows from Lemma \ref{NORM_CONGR_ADMISSBLE_REL_TRANS} and Table 1.
\end{proof}

From Corollary \ref{PARASTROPHY_INVAR_NORM_CONG} it follows that any element of the group $FM(Q, \cdot)$ of a quasigroup $(Q, \cdot)$ is admissible relative to any normal congruence of the quasigroup $(Q, \cdot)$.

In the following lemma we give information on behavior of a left quasigroup congruence relative to quasigroup parastrophy.

\begin{lemma} \label{LEFT_CONGR_RELAT_PARASTROPHY}
The following propositions are equivalent.
 \begin{enumerate}
   \item An equivalence $\theta$ is stable relative to  translation $L_c$  of   a quasigroup $(Q,\cdot)$.
   \item An equivalence $\theta$ is stable relative to   translation $R_c$  of   a quasigroup $(Q,\ast)$.
   \item An equivalence $\theta$ is stable relative to  translation  $P^{-1}_c$ of   a quasigroup $(Q,\slash)$.
   \item An equivalence $\theta$ is stable relative to  translation  $L^{-1}_c$ of   a quasigroup $(Q,\backslash)$.
   \item An equivalence $\theta$ is stable relative to   translation $R^{-1}_c$  of   a quasigroup $(Q,\backslash\backslash)$.
   \item An equivalence $\theta$ is stable relative to   translation $P_c$ of   a quasigroup $(Q,\slash\slash)$.
 \end{enumerate}
\end{lemma}
\begin{proof}
From Table 1 it follows that $L^{\ast}_c = R_c$, $L^{\slash}_c = P^{-1}_c$.  And so on.
\end{proof}

\begin{remark}
It is easy to see that similar equivalences (Lemma \ref{LEFT_CONGR_RELAT_PARASTROPHY}) are true for the other five kinds of translations and its combinations of a quasigroup $(Q,\cdot)$.
\end{remark}

\begin{lemma} \label{NORMALITY_INVAR_CONGR_RELAT_TRANSL}
\begin{enumerate}
  \item If an equivalence $\theta$ of a quasigroup $(Q, \cdot)$ is admissible relative to any left and right translation of this quasigroup, then $\theta$ is a normal congruence.
      \item If an equivalence $\theta$ of a quasigroup $(Q, \cdot)$ is admissible relative to any left and middle translation of this quasigroup, then $\theta$ is a normal congruence.
  \item If an equivalence $\theta$ of a quasigroup $(Q, \cdot)$ is admissible relative to any right and middle  translation of this quasigroup, then $\theta$ is a normal congruence.
\end{enumerate}
\end{lemma}
\begin{proof}
Case 1. Definition \ref{NORM_CONGR_QUAS}.

Case 2. Equivalence $\theta$ is a normal congruence of quasigroup $(Q, \backslash)$ (Table 1, Definition \ref{NORM_CONGR_QUAS}) and we apply Corollary  \ref{PARASTROPHY_INVAR_NORM_CONG}.

Case 3. Equivalence $\theta$ is a normal congruence of quasigroup $(Q, \slash)$ (Table 1, Definition \ref{NORM_CONGR_QUAS})  and we apply Corollary  \ref{PARASTROPHY_INVAR_NORM_CONG}.
\end{proof}

\subsection{Other conditions of normality of equivalences}

We find additional conditions when an equivalence is  left, right or  "middle" quasigroup congruence.
We  use  H. Thurston \cite{10, THURSTON} and A.I.~Mal'tsev \cite{7} approaches.
Results similar to the results from this section are in  \cite{SHCH_BR_BL_05, SCERB_TAB_PUSH_09}.

The set of all left and  right  translations of a  quasigroup $(Q,\cdot)$ will be denoted by
$\mathbb T(Q,\cdot)$.
If $ \varphi $ and $ \psi $ are  binary relations on $Q $, then  their
product is defined in the following way: $ (a, b) \in \varphi \circ \psi $ if there is an element $c\in Q $ such
that $ (a, c) \in \varphi $ and $ (c, b) \in \psi $. If $ \varphi $ is a  binary relation on $Q $, then
$\varphi^{-1} = \left\{ (y, x) \, \mid \, (x, y) \in \varphi \right\}$. The operation of the product of binary
relations is associative \cite{5, II, SM,  RIGE}.

\begin{remark}
Translations of a quasigroup can be considered as binary relations:  $ (x, y) \in L_a, $ if and only if $y =
a\cdot x $; $ (x, y) \in R_b, $ if and only if $ y = x\cdot b $; $ (x, y) \in P_c, $ if and only if $ y = x\backslash c $ \cite{90, SCERB_91, SHCH_BR_BL_05}.
\end{remark}

\begin{remark} \label{NREM_1}
To coordinate the multiplication of translations with  their multiplication as binary relations, we use the
following  multiplication of translations: if  $ \alpha, \beta $ are translations, $x $ is an  element of the
set $Q$, then $ (\alpha \beta) (x) = \beta (\alpha (x))$, i.e. $(\alpha \beta) x = \beta\alpha \, x$.
\end{remark}

\begin{proposition} \label{NP1}
\begin{enumerate}
  \item An equivalence $ \theta $ is a
left congruence of a quasigroup $(Q, \cdot) $ if and only if $ \theta \omega \subseteq \omega\theta $ for all $ \omega \in
{\mathbb L}$.
  \item An equivalence $ \theta $ is a
right congruence of a quasigroup $(Q, \cdot)$ if and only if $ \theta \omega \subseteq \omega\theta $ for all $ \omega \in
{\mathbb R}$.
 \item An equivalence $ \theta $ is a
"middle" congruence of a quasigroup $(Q, \cdot)$ if and only if $ \theta \omega \subseteq \omega\theta $ for all $ \omega \in
{\mathbb P}$.
  \item An equivalence $ \theta $ is a
congruence of a quasigroup $(Q, \cdot)$ if and only if $ \theta \omega \subseteq \omega\theta $ for all $ \omega \in
{\mathbb T}$ \cite{90, SCERB_91, SHCH_BR_BL_05}.
\end{enumerate}
\end{proposition}
\begin{proof} Case 1. Let $ \theta $ be an  equivalence, $ \omega = L_a $. It is clear that   $ (x, z) \in
\theta L_a $
 is equivalent to that there exists an element $y\in Q $ such
that $ (x, y) \in \theta $ and $ (y, z) \in L_a $. But if $ (y, z) \in L_a $, $z=ay $, then $y = L ^ {-1} _a z
$. Therefore, from the relation  $ (x, z) \in \theta L_a $ it follows that $ (x, L ^ {-1} _a z) \in \theta $.

Let us prove that from $ (x, L ^ {-1} _a z) \in \theta $  it follows $ (x, z) \in \theta L_a $. We have  $ (x, L
^{-1} _a z) \in \theta $ and $ (L ^ {-1} _a z, z) \in L_a $, $ (x, z) \in \theta L_a $. Thus $ (x, z) \in \theta
L_a $ is equivalent to  $ (x, L ^ {-1} _az) \in \theta $.
\smallskip

Similarly, $ (x, z) \in L_a\theta $ is equivalent to $ (ax, z) \in \theta $. Now we can say that the inclusion $
\theta \omega \subseteq \omega\theta $ by  $ \omega = L_a $ is equivalent to  the following implication:
$$
(x, L ^ {-1} _a z) \in \theta \Longrightarrow (ax, z) \in \theta
$$
for all suitable  $a, x, z \in Q $.

If we replace in the last implication $z $ with  $L_a z $, we shall obtain the following  implication:
$$
(x, z) \in \theta \Longrightarrow (ax, az) \in \theta
$$
for all $a\in Q $.

Thus, the inclusion $ \theta L_a \subseteq L_a \theta $ is equivalent to the stability of the relation  $ \theta
$ from
 the left relative to  an element $a $. Since the element $a $ is an arbitrary element of the set $Q$, we have that
the inclusion $ \theta \omega \subseteq \omega \theta $ by $ \omega \in {\mathbb L} $  is equivalent to the
stability of the relation $ \theta $ from  the left.

Case 2. Similarly, the  inclusion $ \theta \omega \subseteq \omega \theta $ for any $ \omega \in {\mathbb R} $  is
equivalent to the stability  from  the right of relation  $ \theta $.

Case 3. Applying Case 1 to the quasigroup $(Q, \slash\slash) = (Q, \cdot)^{(132)}$ we obtain that an equivalence
$ \theta $ is a left congruence of a quasigroup $(Q, \cdot)^{(132)} $ if and only if $ \theta \omega \subseteq \omega\theta $ for all $ \omega \in {\mathbb L^{(132)}}$.

Using   parastrophic equivalence of translations (Table 1 or Lemma \ref{LEFT_CONGR_RELAT_PARASTROPHY}) we conclude that  an equivalence $ \theta $ is stable relative to any middle translation of a quasigroup $(Q, \cdot) $ if and only if $ \theta \omega \subseteq \omega\theta $ for all $ \omega \in {\mathbb P}$.

In other words an equivalence $ \theta $ is a
"middle" congruence of a quasigroup $(Q, \cdot)$ if and only if $ \theta \omega \subseteq \omega\theta $ for all $ \omega \in {\mathbb P}$.

Case 4. Uniting Cases 1 and 2  we obtain required  equivalence.
\end{proof}

Let us remark  Proposition \ref{NP1} can be deduced from results of H.~Thurston \cite{THURSTON}.

\begin{proposition} \label{NP1_1}
\begin{enumerate}
  \item An equivalence $ \theta $ is  left-cancellative  congruence of a quasigroup $(Q, \cdot) $ if and only if $ \omega\theta  \subseteq  \theta \omega   $ for all $ \omega \in
{\mathbb L}$.
  \item An equivalence $ \theta $ is  right-cancellative  congruence of a quasigroup $(Q, \cdot) $ if and only if  $ \omega\theta   \subseteq     \theta \omega $ for all $ \omega \in
{\mathbb R}$.
 \item An equivalence $ \theta $ is  middle-cancellative  congruence of a quasigroup $(Q, \cdot) $ if and only if $ \omega\theta  \subseteq \theta \omega   $ for all $ \omega \in
{\mathbb P}$.
  \end{enumerate}
\end{proposition}
\begin{proof} As it is proved in Proposition  \ref{NP1}, the inclusion $ \theta L_a \subseteq L_a \theta
$ is equivalent to the implication $x\theta y \Longrightarrow ax\theta ay $.

 Let us check up  that the inclusion $L_a \theta \subseteq \theta L_a $ is equivalent to the implication
$$
ax\theta ay \Rightarrow x\theta y.
$$

Indeed, as it is proved in Proposition \ref{NP1}, $ (x, z) \in \theta L_a $ is equivalent with  $ (x, L ^ {-1}
_az) \in \theta $. Similarly, $ (x, z) \in L_a\theta $ is equivalent with  $ (ax, z) \in \theta $. The inclusion
$ \omega\theta \subseteq \theta \omega $ by  $ \omega = L_a $ has the form  $L_a\theta \subseteq \theta L_a $
and it is equivalent to the following  implication:
$$
(ax, z) \in \theta \Longrightarrow (x, L ^ {-1} _a z) \in \theta
$$
for all  $a, x, z \in Q $. If we change in the last implication the  element $z $ by the element $L_a z $, we
shall obtain that the inclusion $ \theta L_a \supseteq L_a \theta $ is equivalent to  the implication $ax\theta
ay \Rightarrow x\theta y $. Therefore, the equivalence $ \theta $ is cancellative   from  the left.

Similarly, the inclusion $R_b \theta \subseteq \theta R_b $ is equivalent to the implication:
$$
(xa, za) \in \theta \Longrightarrow (x, z) \in \theta.
$$

Case 3 is  proved similarly. It is possible to use identity (\ref{(63)}).
\end{proof}

Below by the symbol $\left< {\mathbb L}, {\mathbb P} \right>$ we shall denote the group generated by all left and middle translations of a quasigroup.

\begin{corollary} \label{NP1_1_1}
\begin{enumerate}
  \item An equivalence $ \theta $ is  left normal  congruence of a quasigroup $(Q, \cdot) $ if and only if $ \omega\theta  =  \theta \omega   $ for all $ \omega \in
{\mathbb L}$.
  \item An equivalence $ \theta $ is  right normal  congruence of a quasigroup $(Q, \cdot) $ if and only if  $ \omega\theta   =     \theta \omega $ for all $ \omega \in
{\mathbb R}$.
 \item An equivalence $ \theta $ is  middle normal   congruence of a quasigroup $(Q, \cdot) $ if and only if $ \omega\theta  = \theta \omega   $ for all $ \omega \in
{\mathbb P}$.
  \item An equivalence $ \theta $ is   normal  congruence of a quasigroup $(Q, \cdot) $ if and only if $ \omega\theta  =  \theta \omega   $ for all $ \omega \in
{\mathbb L} \cup {\mathbb P}$.
  \item An equivalence $ \theta $ is   normal  congruence of a quasigroup $(Q, \cdot) $ if and only if $ \omega\theta  =  \theta \omega   $ for all $ \omega \in
{\mathbb R} \cup {\mathbb P}$.
 \item An equivalence $ \theta $ is   normal   congruence of a quasigroup $(Q, \cdot) $ if and only if $ \omega\theta  = \theta \omega   $ for all $ \omega \in
\left< {\mathbb L}, {\mathbb R} \right> = M(Q, \cdot)$.
  \item An equivalence $ \theta $ is   normal  congruence of a quasigroup $(Q, \cdot) $ if and only if $ \omega\theta  =  \theta \omega   $ for all $ \omega \in
\left< {\mathbb L}, {\mathbb P} \right>$.
  \item An equivalence $ \theta $ is   normal  congruence of a quasigroup $(Q, \cdot) $ if and only if $ \omega\theta  =  \theta \omega   $ for all $ \omega \in
\left< {\mathbb R}, {\mathbb P} \right>$.
   \end{enumerate}
\end{corollary}
\begin{proof}
The proof follows from Propositions  \ref{NP1}, \ref{NP1_1} and Lemma \ref{NORMALITY_INVAR_CONGR_RELAT_TRANSL}.

In the proving of Cases 6--8 we can use the following fact: if $\omega\theta = \theta\omega$, then $\omega^{-1}\theta = \theta\omega^{-1}$.
\end{proof}

\iffalse
\begin{remark}
In Corollary \ref{NP1_1_1} conditions of stability of an equivalence  relative to any left (right, middle) quasigroup translation and its inverse are given.
\end{remark}
\fi

The following proposition is  almost obvious corollary of  Theorem 5 from \cite{7}.

\begin{proposition} \label{NP2}
\begin{enumerate}
  \item
 An equivalence $ \theta $
is a left congruence of a quasigroup $Q $ if and only if  $ \omega \theta (x) \subseteq \theta (\omega x) $ for all
$x\in Q $, $ \omega \in {\mathbb L}$.
  \item
 An equivalence $ \theta $
is a right congruence of a quasigroup $Q $ if and only if  $ \omega \theta (x) \subseteq \theta (\omega x) $ for all
$x\in Q $, $ \omega \in {\mathbb R}$.
 \item
 An equivalence $ \theta $
is a "middle" congruence of a quasigroup $Q $ if and only if  $ \omega \theta (x) \subseteq \theta (\omega x) $ for all
$x\in Q $, $ \omega \in {\mathbb P}$.
 \item An equivalence $ \theta $
is left-cancelative  congruence of a quasigroup $Q $ if and only if  $ \theta (\omega x)    \subseteq \omega \theta (x)$ for all
$x\in Q $, $ \omega \in {\mathbb L}$.
 \item An equivalence $ \theta $
is right-cancelative  congruence of a quasigroup $Q $ if and only if  $ \theta (\omega x)    \subseteq \omega \theta (x)$ for all
$x\in Q $, $ \omega \in {\mathbb R}$.
 \item An equivalence $ \theta $
is middle-cancelative  congruence (i.e. $P_x a \, \theta \, P_x b \Rightarrow a \theta b $) of a quasigroup $Q $  if and only if  $ \theta (\omega x)    \subseteq \omega \theta (x)$ for all
$x\in Q $, $ \omega \in {\mathbb P}$.
\end{enumerate}
\end{proposition}
\begin{proof} Case 1. Let $ \theta $ be  an equivalence relation and for all $ \omega \in {\mathbb L} $, $
\omega \theta (x) \subseteq \theta (\omega x) $. We shall prove that  from $a \theta b $ follows $ca \theta cb
$ for all $c\in Q $.

By definition of the equivalence $\theta$, $a\theta b $  is equivalent to $a\in \theta (b) $. Then $ ca \in c
\theta (b) {\subseteq} \theta (cb),$   $ca\theta cb $.

Converse.  Let $ \theta $ be a left congruence. We shall prove that $c \, \theta (a) \subseteq \theta (ca) $ for all
$c, a \in Q $. Let $x\in c\, \theta (a) $. Then $x = c y $, where $y\in \theta (a) $, that is $y\theta a $.
Then, since  $ \theta $ is  a left congruence, we obtain $cy\theta ca $. Therefore $x = cy \in \theta (ca) $. Thus,
$L_c\theta \subseteq \theta (ca) $.

Case 2 is proved similarly.

Case 3. We can use approach similar to the approach used by the proof of Case 4 of Proposition \ref{NP1}.

Cases 4--6 are proved in the similar way with Cases 1--3.
\end{proof}

\begin{proposition} \label{EQUIV_NORM_CONG}
\begin{enumerate}
  \item An equivalence  $ \theta $ of a quasigroup $Q $ is left normal if and only if
 $ \theta (\omega x) = \omega \theta (x) $ for all $ \omega \in {\mathbb L} $, $x\in Q $;
  \item A equivalence  $ \theta $ of a quasigroup $Q $ is right normal if and only if  $ \theta (\omega x) = \omega \theta (x) $ for all $ \omega \in {\mathbb R} $, $x\in Q $.
  \item An equivalence  $ \theta $ of a quasigroup $Q $ is middle normal if and only if  $ \theta (\omega x) = \omega \theta (x) $ for all $ \omega \in {\mathbb P} $, $x\in Q $.
\end{enumerate}
\end{proposition}
\begin{proof} The proof follows from Proposition \ref{NP2}.
\end{proof}

\begin{proposition}
\begin{enumerate}
  \item An equivalence  $ \theta $  is  normal congruence of a quasigroup $Q $  if and only if
 $ \theta (\omega x) = \omega \theta (x) $ for all $ \omega \in {\mathbb L} \cup {\mathbb R}$, $x\in Q $;
  \item An equivalence  $ \theta $  is  normal congruence of a quasigroup $Q $  if and only if
 $ \theta (\omega x) = \omega \theta (x) $ for all $ \omega \in {\mathbb L} \cup {\mathbb P}$, $x\in Q $;
  \item An equivalence  $ \theta $  is  normal congruence of a quasigroup $Q $  if and only if
 $ \theta (\omega x) = \omega \theta (x) $ for all $ \omega \in {\mathbb R} \cup {\mathbb P}$, $x\in Q $.
 \item An equivalence  $ \theta $  is  normal congruence of a quasigroup $Q $  if and only if
 $ \theta (\omega x) = \omega \theta (x) $ for all $ \omega \in \left< {\mathbb L},  {\mathbb R}\right>$, $x\in Q $;
  \item An equivalence  $ \theta $  is  normal congruence of a quasigroup $Q $  if and only if
 $ \theta (\omega x) = \omega \theta (x) $ for all $ \omega \in \left< {\mathbb L}, {\mathbb P}\right>$, $x\in Q $;
  \item An equivalence  $ \theta $  is  normal congruence of a quasigroup $Q $  if and only if
 $ \theta (\omega x) = \omega \theta (x) $ for all $ \omega \in \left<{\mathbb R}, {\mathbb P}\right>$, $x\in Q $.
\end{enumerate}
\end{proposition}
\begin{proof} The proof follows from Proposition \ref{EQUIV_NORM_CONG} and Lemma \ref{NORMALITY_INVAR_CONGR_RELAT_TRANSL}.

In the proving of Cases 4--6 we can use the following fact. If $\theta (\omega x) = \omega\theta (x)$, then  $\theta (\omega\omega^{-1} x) = \omega \theta (\omega^{-1} x) = \omega\omega^{-1} \theta (x)$. Thus   $\omega \theta (\omega^{-1} x) = \omega\omega^{-1} \theta (x)$, $\theta (\omega^{-1} x) = \omega^{-1} \theta (x)$.
\end{proof}

\subsection{A-nuclei and congruences}

\begin{definition}
We define the following binary relation on a quasigroup $(Q, \cdot)$ that correspond to a subgroup $H$ of the  group $_1N^A_l$:
 $a \, (_1\theta^A_l)\, b$ if and only if there exists a permutation $\alpha \in H \subseteq \, _1N^A_l  $ such that $b = \alpha a$.
\end{definition}

For a subgroup $H$ of the group  $_3N^A_l$ ($_1N^A_m$, $_2N^A_m$,   $_2N^A_r$ and $_3N^A_r$) binary relation $_3\theta^A_l$ ($_1\theta^A_m$, $_2\theta^A_m$,   $_2\theta^A_r$ and $_3\theta^A_r$, respectively) is defined in the similar way.

\begin{lemma} \label{BIN_REL_IS_EQIVIV}
The binary relation $_1\theta^A_l$ is an equivalence of the set $Q$.
\end{lemma}
\begin{proof}
Since binary relation $_1\theta^A_l$  is  defined  using  orbits by the action on the set $Q$ of a subgroup of the group $_1N^A_l$, we conclude that  this binary relation is an  equivalences. It is clear that any orbit defines an equivalence class. Notice, all these equivalence classes are of equal order (Corollary \ref{SUBGR_FREE_ACT_GR}).
\end{proof}

Analogs of Lemma \ref{BIN_REL_IS_EQIVIV} are true for any subgroup of the groups  $_3\theta^A_l$, $_1\theta^A_m$, $_2\theta^A_m$,   $_2\theta^A_r$ and $_3\theta^A_r$.

These equivalences have additional properties.

\begin{lemma} \label{Pairs_of_Eqiviv} Let $(Q, \cdot)$ be a quasigroup.
\begin{enumerate}
\item Let $_1\theta^A_l$ be an equivalence   that is defined by a subgroup $K$ of the group $_1N^A_l$  and equivalence  $_3\theta^A_l$ is defined by isomorphic to $K$ subgroup  $R_x K R^{-1}_x$ of the group $_3N^A_l$ (Lemma \ref{ISOMORPHISMS OF COMPONENTS}). Then we have the following implications:
    \begin{enumerate}
\item    if $a (_1\theta^A_l) b$,  then $(a \cdot c) (_3\theta^A_l) (b \cdot c)$ for all $a, b, c \in Q$;
\item  if $(a \cdot c) (_1\theta^A_l) (b \cdot c)$,  then $a (_3\theta^A_l) b$ for all $a, b, c \in Q$;
\end{enumerate}

\item Let $_3\theta^A_l$ be an equivalence   that is defined by a subgroup $K$ of the group $_3N^A_l$  and equivalence  $_1\theta^A_l$ is defined by isomorphic to $K$ subgroup  $R^{-1}_x K R_x$ of the group $_1N^A_l$ (Lemma \ref{ISOMORPHISMS OF COMPONENTS}). Then we have the following implications:
      \begin{enumerate}
\item if $a (_3\theta^A_l) b$,  then $(a \cdot c) (_1\theta^A_l) (b \cdot c)$ for all $a, b, c \in Q$;
\item   if $(a \cdot c) (_3\theta^A_l) (b \cdot c)$,  then $a (_1\theta^A_l) b$ for all $a, b, c \in Q$;
 \end{enumerate}
 \item Let $_2\theta^A_r$ be an equivalence   that is defined by a subgroup $K$ of the group $_2N^A_r$  and equivalence  $_3\theta^A_r$ is defined by isomorphic to $K$ subgroup  $L_x K L^{-1}_x$ of the group $_3N^A_r$ (Lemma \ref{ISOMORPHISMS OF COMPONENTS}). Then we have the following implications:
      \begin{enumerate}
 \item if $a (_2\theta^A_r) b$,  then $(c \cdot a) (_3\theta^A_r) (c \cdot b)$ for all $a, b, c \in Q$;
\item  if $(c \cdot a) (_2\theta^A_r) (c \cdot b)$,  then $a (_3\theta^A_r) b$ for all $a, b, c \in Q$;
 \end{enumerate}

 \item Let $_3\theta^A_r$ be an equivalence   that is defined by a subgroup $K$ of the group $_3N^A_r$  and equivalence  $_2\theta^A_r$ is defined by isomorphic to $K$ subgroup  $L^{-1}_x K L_x$ of the group $_3N^A_r$ (Lemma \ref{ISOMORPHISMS OF COMPONENTS}). Then we have the following implications:
  \begin{enumerate}
\item if $a (_3\theta^A_r) b$,  then $(c \cdot a) (_2\theta^A_r) (c \cdot b)$ for all $a, b, c \in Q$;
\item if $(c \cdot a) (_3\theta^A_r) (c \cdot b)$,  then $a (_2\theta^A_r) b$ for all $a, b, c \in Q$;
 \end{enumerate}

 \item Let $_1\theta^A_m$ be an equivalence   that is defined by a subgroup $K$ of the group $_1N^A_m$  and equivalence  $_2\theta^A_m$ is defined by isomorphic to $K$ subgroup  $P_x K P^{-1}_x$ of the group $_2N^A_m$ (Lemma \ref{ISOMORPHISMS OF COMPONENTS}). Then we have the following implications:
      \begin{enumerate}
\item  if $a (_1\theta^A_m) b$,  then $P_c a (_2\theta^A_m) P_c b$ for all $a, b, c \in Q$;
\item  if $ P_c a (_1\theta^A_m) P_c b$,  then $a (_2\theta^A_m) b$ for all $a, b, c \in Q$;
 \end{enumerate}

 \item Let $_2\theta^A_m$ be an equivalence   that is defined by a subgroup $K$ of the group $_2N^A_m$  and equivalence  $_1\theta^A_m$ is defined by isomorphic to $K$ subgroup  $P^{-1}_x K P_x$ of the group $_1N^A_m$ (Lemma \ref{ISOMORPHISMS OF COMPONENTS}). Then we have the following implications:
      \begin{enumerate}
 \item if $a (_2\theta^A_m) b$,  then $P^{-1}_c a  (_1\theta^A_m) P^{-1}_c b$ for all $a, b, c \in Q$;
\item if $P^{-1}_c a  (_2\theta^A_m) P^{-1}_c b $,  then $a (_1\theta^A_m) b$ for all $a, b, c \in Q$.
  \end{enumerate}
 \end{enumerate}
\end{lemma}
\begin{proof}
Case 1, (a). Expression  $a (_1\theta^A_l) b$ means that there exists a permutation  $\alpha \in K \subseteq \,  _1 N^A_l$ such that  $b = \alpha a$.
Expression  $(a \cdot c) (_3\theta^A_l) (b \cdot c)$ means that there exists a permutation  $\gamma \in R_x K R^{-1}_x \subseteq \,_3 N^A_l$ such that $b \cdot c  = \gamma (a \cdot c)$. We can take $\gamma = R_c K R^{-1}_c$, i.e. can take $x=c$.  Then $b \cdot c  = \alpha a \cdot c$.

Then we obtain the following implication $b = \alpha a \Longrightarrow b \cdot c  = \alpha a \cdot c$ for all $a, b, c \in Q$. It is clear that this implication is true for all  $a, b, c \in Q$, since $(Q, \cdot)$ is a quasigroup.
Therefore, implication
\[
\text{if} \; a (_1\theta^A_l) b, \: \text{then} \; (a \cdot c) (_3\theta^A_l) (b \cdot c) \: \text{for all} \: a, b, c \in Q
\]
also is true.

Case 1, (b).
Expression  $(a \cdot c) (_1\theta^A_l) (b \cdot c)$ means that there exists a permutation  $\alpha \in K  \subseteq \, _1 N^A_l$ such that  $b \cdot c = \alpha (a \cdot c)$, i.e. $b  = R^{-1}_c\alpha R_c a $.
Expression $a (_3\theta^A_l) b$ means that there exists a permutation  $\gamma \in H = R_x K R^{-1}_x \subseteq \,  _3 N^A_l$ such that $b  = \gamma a $.
By Lemma \ref{ISOMORPHISMS OF COMPONENTS} we can take $\gamma = R^{-1}_c \alpha  R_c$.

Then we obtain the following implication
\[
b  = R^{-1}_c\alpha R_c a  \Longrightarrow b  =  R^{-1}_c\alpha R_c a \: \; \text{for all} \;  a, b, c \in Q
\]
 It is clear that this implication is true.

Cases 2--4 are proved in the similar way with  Case 1.

Case 5, (a). Expression  $a (_1\theta^A_m) b$ means that there exists a permutation  $\alpha \in K \subseteq \, _1 N^A_m$ such that  $b = \alpha a$.
Expression  $P_c a (_2\theta^A_m) P_c b$ means that there exists a permutation  $\beta \in H \subseteq \,  _2 N^A_m$ such that $P_c b  = \beta P_c a$.
By Lemma \ref{ISOMORPHISMS OF COMPONENTS} we can take $\beta = P_c \alpha P^{-1}_c$. Then
$P_c b  = \beta P_c a  = P_c \alpha P^{-1}_c P_c a = P_c \alpha a$.

Then we obtain the following implication $b = \alpha a \Longrightarrow P_c b  = P_c \alpha a$ for all $a, b, c \in Q$. It is clear that this implication is true for all  $a, b, c \in Q$.

Case 5, (b).
Expression $ P_c a (_1\theta^A_m) P_c b$ means that there exists a permutation  $\alpha \in H \subseteq \,  _1 N^A_m$ such that $P_c a  = \alpha P_c b$. By Lemma \ref{ISOMORPHISMS OF COMPONENTS} we can take $\alpha = P_c \beta P^{-1}_c$, where $\beta \in  \, _2 N^A_m$. Thus  $P_c a  = P_c \beta P^{-1}_c P_c b$, $a = \beta b$,   $a (_2\theta^A_m) b$.

Case 6 is proved in the similar way with Case 5.
\end{proof}

Cases 1 and 2 show that the pair of equivalences $_1\theta^A_l$ and $_3\theta^A_l$ is  normal from the right reciprocally.

Cases 3 and 4 show that the pair of equivalences $_2\theta^A_r$ and $_3\theta^A_r$ is  normal from the left  reciprocally.

Cases 5 and 6 show that the pair of equivalences $_1\theta^A_m$ and $_2\theta^A_m$ is stable relative to middle quasigroup translations and its inverse.

\begin{corollary} \label{Eqiviv_NORMALITY}
If in a quasigroup $(Q,\cdot)$:
\begin{itemize}
  \item [] $_1\theta^A_l = \, _3\theta^A_l = \theta^A_l$, then the equivalence $_1\theta^A_l$ ($_3\theta^A_l$) is normal from the right;
  \item [] $_2\theta^A_r = \, _3\theta^A_r  = \theta^A_r$, then the equivalence $_2\theta^A_r$  ($_3\theta^A_r$) is normal from the left;
  \item [] $_1\theta^A_m = \, _2\theta^A_m = \theta^A_m$, then the equivalence $_1\theta^A_m$ ($_2\theta^A_m$) is  stable relative to middle quasigroup translations and its inverse;
  \item []  $\theta^A_l =  \theta^A_r$, then equivalences  ${}_1\theta^A_l$,  ${}_3\theta^A_l$, ${}_2\theta^A_r$, ${}_3\theta^A_r$  are normal congruences;
  \item []  $\theta^A_l =  \theta^A_m$, then equivalences ${}_1\theta^A_l$,  ${}_3\theta^A_l$, ${}_1\theta^A_m$, ${}_2\theta^A_m$  are normal congruences;
  \item []  $\theta^A_r =  \theta^A_m$, then equivalences ${}_2\theta^A_r$,  ${}_3\theta^A_r$, ${}_1\theta^A_m$, ${}_2\theta^A_m$  are normal congruences.
\end{itemize}
\end{corollary}
\begin{proof}
We use Lemmas \ref{Pairs_of_Eqiviv} and \ref{NORMALITY_INVAR_CONGR_RELAT_TRANSL}.
\end{proof}

If in conditions of Lemma \ref{Pairs_of_Eqiviv} $ K= \, _1N^A_l$, then corresponding equivalence we shall denote by $_1\Theta^A_l$ and so on.

\begin{corollary} \label{ISMS OF COMPONENTS_LOOPS}
\begin{enumerate}
  \item If  $(Q,\circ)$ is a right loop, then  $_1\theta^A_l = \, _3\theta^A_l = \theta^A_l$ and equivalence $_1\theta^A_l$  is normal from the right. Any subgroup of the left nucleus $N_l$ (coset class $\theta^A_l(1)$) is normal from the right.
  \item If  $(Q,\circ)$ is a left loop, then  $_2\theta^A_r = \, _3\theta^A_r  = \theta^A_r$ and equivalence $_2\theta^A_r$ is normal from the left. Any subgroup of the right nucleus $N_r$ (coset class $\theta^A_r(1)$) is normal from the left.
  \item If  $(Q,\circ)$ is an unipotent quasigroup, then $_1\theta^A_m = \, _2\theta^A_m = \theta^A_m$ and equivalence $_1\theta^A_m$ is  stable relative to middle quasigroup translations and its inverse. Any subgroup of the middle nucleus $N_m$ (coset class $\theta^A_m(1)$) is stable relative to middle quasigroup translations and its inverse.
\end{enumerate}
 \end{corollary}
\begin{proof}
The proof follows from  Corollary \ref{ISOMORPHISMS OF COMPONENTS_LOOPS} and Corollary  \ref{Eqiviv_NORMALITY}.
\end{proof}

\begin{corollary} \label{Normality_OF_quasigroup_Nuclei}
\begin{enumerate}
  \item If  $(Q,\circ)$ is a commutative loop, then any from equivalences $\theta^A_l$ and $\theta^A_r$ is a normal congruence. Left nucleus $N_l$ (coset class $\Theta^A_l(1)$) and right nucleus $N_r$ (coset class
      $\Theta^A_r(1)$) are equal and $N_l$ is  a normal subgroup of $(Q,\circ)$.
      \item If  $(Q,\circ)$ is a unipotent right  loop with identity $x\circ (x\circ y) = y$, then any from  equivalences $\theta^A_l$ and $\theta^A_m$ is a normal congruence. Left nucleus $N_l$ (coset class $\Theta^A_l(1)$) and middle  nucleus $N_m$ (coset class $\Theta^A_m(1)$) are equal  and $N_l$ is a  normal subgroup of $(Q,\circ)$.
\item If  $(Q,\circ)$ is a unipotent left loop with identity $(x\circ y) \circ y = x$, then any from equivalences $\theta^A_r$ and $\theta^A_m$ is a normal congruence. Right  nucleus $N_r$ (coset class $\Theta^A_r(1)$) and middle  nucleus $N_m$ (coset class $\Theta^A_m(1)$) are equal  and $N_r$ is a  normal subgroup of $(Q,\circ)$.
 \end{enumerate}
 \end{corollary}
\begin{proof}
Case 1.  In any  right loop first and third component of any element of the group $N^A_l$ coincide,  in any  left loop second and third component of the group  $ N^A_r$  coincide (Corollary \ref{A-NUCLEI_OF_LOOPS}). In a commutative loop in additional  $L_a = R_a$ for any element $a\in Q$.
 From Corollary   \ref{Eqiviv_NORMALITY} it follows that in this case any  equivalence $\theta^A_l$ is a normal congruence and therefore any subgroup of left nucleus  $N_l$ is normal as coset class $\theta^A_l(1)$ of normal congruence $\theta^A_l$.

Cases 2 and 3 are proved using  passage to parastrophy images ($(23)$ and $(13)$, respectively) of the loop from Case 1.
\end{proof}

\begin{lemma} \label{SUBGR_CENTER_NORMALITY}
In a loop  $(Q,\cdot)$ any subgroup of  group $Z$ ($Z^A_l$, $Z^A_r$, $Z^A_m$, ${}_1Z^A_l$, ${}_3Z^A_l$, ${}_2Z^A_r$, ${}_3Z^A_r$, ${}_1Z^A_m$, ${}_2Z^A_m$) defines a normal congruence.
\end{lemma}
\begin{proof}
We use Lemma  \ref{ISOM_OF_CENTERS_OF_LOOP} and Corollary \ref{Eqiviv_NORMALITY}. If $H$ is a subgroup of the group $Z$,  then corresponding to $H$ subgroups of the groups ${}_1Z^A_l$, ${}_3Z^A_l$   are equal,  since $(Q, \cdot)$ is a loop. Therefore   corresponding equivalences ${}_1\zeta^A_{\,l}= {}_3\zeta^A_{\, l}$ are equal. Since ${}_1Z^A_l = {}_2Z^A_r = {}_3Z^A_r$ (Lemma  \ref{ISOM_OF_CENTERS_OF_LOOP}) we obtain that ${}_1\zeta^A_{\, l} = {}_2\zeta^A_{\, r} = {}_3\zeta^A_{\, r}$ and equivalence ${}_1\zeta^A_{\, l}$ is normal congruence (Corollary \ref{Eqiviv_NORMALITY}) of loop $(Q, \cdot)$.

Other cases are proved in the similar way.
\end{proof}

\begin{corollary} \label{SUBGR_CENTER_NORMALITY_COR}
In a loop  $(Q,\cdot)$ A-center defines a normal congruence.
\end{corollary}
\begin{proof}
We use Lemma  \ref{SUBGR_CENTER_NORMALITY}.
\end{proof}

\medskip

We give conditions of normality of A-nuclei of a loop. We also give conditions of normality of Garrison's nuclei, since  Garrison's nuclei are coset classes by action of corresponding A-nuclei.

\begin{theorem} \label{NORM_COND_LOOPS}
Let $(Q, \cdot)$ be a loop.
\begin{enumerate}
\item Let $H$ be  a subgroup of the group ${}_1N^A_l$. Denote by  ${}_1\theta^A_l$  corresponding equivalence to the group $H$.
\begin{enumerate}
\item If equivalence ${}_1\theta^A_l$ is admissible relative to any translation $L_x$ of loop $(Q, \cdot)$ and its inverse, then ${}_1\theta^A_l$ is normal congruence of  $(Q, \cdot)$.
\item If equivalence ${}_1\theta^A_l$ is admissible relative to any translation $P_x$ of loop $(Q, \cdot)$ and its inverse, then ${}_1\theta^A_l$ is normal congruence of  $(Q, \cdot)$.
\end{enumerate}
\item Let  $H$ be  a subgroup of the group ${}_2N^A_r$. Denote by  ${}_2\theta^A_r$  corresponding equivalence to the group $H$.
\begin{enumerate}
\item If equivalence ${}_2\theta^A_r$ is admissible relative to any translation $R_x$ of loop $(Q, \cdot)$ and its inverse, then ${}_2\theta^A_r$ is normal congruence of  $(Q, \cdot)$.
\item
If equivalence ${}_2\theta^A_r$ is admissible relative to any translation $P_x$ of loop $(Q, \cdot)$ and its inverse, then ${}_2\theta^A_r$ is normal congruence of  $(Q, \cdot)$.
\end{enumerate}
\item Let  $H$ be  a subgroup of the group $\,{}_1N^A_m \cap \, {}_2N^A_m$. Denote by  ${}_1\theta^A_m$  corresponding equivalence to the group $H$.
\begin{enumerate}
\item If equivalence ${}_1\theta^A_m$ is admissible relative to any translation $R_x$ of loop $(Q, \cdot)$ and its inverse, then ${}_1\theta^A_m$ is normal congruence of  $(Q, \cdot)$.
\item
If equivalence ${}_1\theta^A_m$ is admissible relative to any translation $L_x$ of loop $(Q, \cdot)$ and its inverse, then ${}_1\theta^A_m$ is normal congruence of  $(Q, \cdot)$.
\end{enumerate}
\end{enumerate}
\end{theorem}
\begin{proof} The proof follows from Lemma  \ref{NORMALITY_INVAR_CONGR_RELAT_TRANSL} and Corollary \ref{ISMS OF COMPONENTS_LOOPS}.
\end{proof}

\begin{theorem}
If in a quasigroup $(Q, \cdot)$ $H = {}_1N_l^A \cap {}_3N_l^A \cap {}_2N_r^A \cap {}_3N_r^A$, then $H$ induces  normal congruence of $(Q, \cdot)$.

If in a quasigroup $(Q, \cdot)$ $H = {}_1N_l^A \cap {}_3N_l^A \cap {}_1N_m^A \cap {}_2N_m^A$, then $H$ induces normal congruence of $(Q, \cdot)$.

If in a quasigroup $(Q, \cdot)$ $H = {}_2N_r^A \cap {}_3N_r^A \cap {}_1N_m^A \cap {}_2N_m^A$, then $H$ induces normal congruence of $(Q, \cdot)$.
\end{theorem}
\begin{proof}
The proof follows from Lemma \ref{Pairs_of_Eqiviv}.
\end{proof}

\subsection{Normality of A-nuclei of inverse quasigroups}

Mainly information   on the coincidence of A-nuclei and  nuclei of inverse quasigroups and loops are taken from \cite{KEED_SCERB}.
Also we give  results on normality of nuclei of some inverse quasigroups and loops.

\begin{definition}
A quasigroup $(Q,\circ )$ has the {\it $\lambda $-inverse-property} if there exist permutations $\lambda_1,
\lambda_2, \lambda_3$ of the set $Q$ such that \begin{equation} \lambda_1 x \circ  \lambda_2 (x\circ y) =
\lambda_3 y \label{eqno(2.8)}\end{equation} for all $x, y \in Q$ \cite{INVERS}.
\end{definition}

\begin{definition}
A quasigroup $(Q,\circ )$ has the {\it $\rho $-inverse-property} if there exist permutations $\rho_1, \rho_2$,
$\rho_3$ of the set $Q$ such that \begin{equation}  \rho_1 (x\circ y) \circ  \rho_2 y =  \rho_3 x
\label{eqno(2.9)}\end{equation} for all $x, y \in Q$ \cite{INVERS}.
\end{definition}

\begin{definition}
A  quasigroup $(Q, \circ)$ is an $(\alpha, \beta, \gamma
)$-inverse quasigroup if  there exist permutations $\alpha, \beta, \gamma $ of the set $Q$ such that
\begin{equation}
\alpha (x\circ y) \circ \beta x = \gamma y \label{eqno(2.7)}
\end{equation}
 for all $x, y \in Q$ \cite{ks3, RB,KS,STB}.
\end{definition}

\begin{definition}
A quasigroup $(Q,\circ )$ has the {\it $\mu  $-inverse-property} if there exist permutations $\mu _1, \mu _2$,
$\mu _3$ of the set $Q$ such that \begin{equation}  \mu _1 y\circ \mu _2 x  =  \mu _3 (x \circ y)
\label{eqno(2.10)}\end{equation} for all $x, y \in Q$.
\end{definition}

We give definition of  main classes of inverse quasigroups  using autostrophy \cite{INVERS, KEED_SCERB, SCERB_08}.
\begin{theorem} \label{INVERS_AUTOSTR_DEF}
 A quasigroup $(Q, \cdot)$ is:
 \begin{enumerate}
  \item  a  $\lambda $-inverse quasigroup  if and only if it has  $[(2\;3),(\lambda_1,\lambda_2,\lambda_3)]$ autostrophy;
 \item   a $\rho $-inverse quasigroup  if and only if it has  $[(1\;3),$ $(\rho_1,\rho_2,\rho_3)]$ autostrophy;
 \item an $(\alpha ,\beta , \gamma )$-inverse quasigroup if and only if it has
$[(1\;2\;3),(\alpha, \beta ,\gamma)]$ autostrophy;
\item    a \textit{$\mu $-inverse quasigroup} if and only if it has  $[(12), (\mu _1, \mu _2, \mu _3) ]$     autostrophy.
  \end{enumerate}
\end{theorem}
\begin{proof}
If $(\sigma, (\alpha_1, \alpha_2, \alpha_3))$ is an autostrophism of a quasigroup $(Q, A)$,
then from formula (\ref{composition_of_isostrophisms_LEFT_REC})
we obtain the following formula
\begin{equation}\label{autostrophism}
\begin{split}
A(x_1, x_2, x_3) =
 A(\alpha_{1} x_{\sigma^{-1}1}, \alpha_{2} x_{\sigma^{-1} 2}, \alpha_{3} x_{\sigma^{-1} 3})
\end{split}
\end{equation}

Case 1.
Using equality (\ref{autostrophism}) we find values of $\sigma$ and $\alpha_i$. Since $\alpha_{1} x_{\sigma^{-1}1} = \lambda_1 x_1$, we have  $\alpha_{1} = \lambda_1$, $x_{\sigma^{-1}1} = x_1$, $\sigma 1 = 1$. And so on.

Cases 2--4 are proved in similar way with Case 1.
\end{proof}

\begin{lemma} \label{INVERSE_AUTOSTRP}
  \begin{enumerate}
  \item  A  $\lambda $-inverse quasigroup  has  $((2\;3),(\lambda^{-1}_1,\lambda^{-1}_3,\lambda^{-1}_2))$ autostrophy;
 \item   a $\rho $-inverse quasigroup   has  $((1\;3),$ $(\rho^{-1}_3,\rho^{-1}_2,\rho^{-1}_1))$ autostrophy;
 \item an $(\alpha ,\beta , \gamma )$-inverse quasigroup has
$((1\;3\;2),(\beta^{-1}, \gamma^{-1}, \alpha^{-1} ))$ autostrophy, i.e. in any $(\alpha ,\beta , \gamma )$-inverse quasigroup $(Q, \circ)$ the following equality is true $\beta^{-1} y \circ \gamma^{-1}(x\circ y) = \alpha^{-1} x$;
 \item    a \textit{$\mu $-inverse quasigroup} has  $((12), (\mu _2^{-1}, \mu _1^{-1}, \mu _3^{-1}))$     autostrophy.
   \end{enumerate}
\end{lemma}
\begin{proof}
From Lemma \ref{inver_AUTOSTROPHY} we have.  If $(\sigma, T) = (\sigma, (\alpha_1, \alpha_2, \alpha_3))$ is an autostrophism, then
\begin{equation}\label{INVERSE_AUTOsTROPHISM}
(\sigma, T)^{-1} = (\sigma^{-1},(T^{-1})^{\sigma^{-1}}) = (\sigma^{-1}, (\alpha^{-1}_{\sigma 1}, \alpha^{-1}_{\sigma 2}, \alpha^{-1}_{\sigma 3}))
\end{equation}
\end{proof}

\subsubsection{$(\alpha, \beta, \gamma)$-inverse quasigroups}

\begin{lemma} \label{A_NUCLEI_OF_RST_INV}
In $(\alpha, \beta, \gamma)$-inverse quasigroup we have the following relations between components of A-nuclei:
\begin{equation} \label{EUAL_A_NUCLEI_1}
 {}_1N^A_l = \alpha^{-1}{}_3N^A_r \alpha = \beta \, {}_2N^A_m \beta^{-1}
\end{equation}
\begin{equation}\label{EUAL_A_NUCLEI_2}
 {}_3N^A_l = \gamma^{-1} {}_2N^A_r \gamma = \alpha \, {}_1N^A_m \alpha^{-1}
 \end{equation}
\begin{equation}\label{EUAL_A_NUCLEI_3}
 {}_2N^A_r = \beta^{-1}{}_1N^A_m \beta = \gamma \, {}_3N^A_l \gamma^{-1}
 \end{equation}
\begin{equation}\label{EUAL_A_NUCLEI_4}
 {}_3N^A_r = \gamma^{-1}{}_2N^A_m \gamma = \alpha \, {}_1N^A_l \alpha^{-1}
 \end{equation}
\begin{equation}\label{EUAL_A_NUCLEI_5}
 {}_1N^A_m = \alpha^{-1}{}_3N^A_l\alpha  = \beta \, {}_2N^A_r \beta^{-1}
 \end{equation}
\begin{equation}\label{EUAL_A_NUCLEI_6}
 {}_2N^A_m = \beta^{-1}{}_1N^A_l \beta = \gamma \, {}_3N^A_r\gamma.
\end{equation}
\end{lemma}
\begin{proof}
The proof follows from Theorem \ref{INVERS_AUTOSTR_DEF}, Lemma \ref{INVERSE_AUTOSTRP} and Table 3.
\end{proof}

\begin{theorem} \label{A_NUCL_OF_INV_LOOPS}
\begin{enumerate}
  \item In  $(\varepsilon; \beta; \gamma)$-inverse loop $(Q, \circ)$ ${}_1N^A_l = {}_3N^A_l= {}_2N^A_r= {}_3N^A_r = {}_1N^A_m$, $N_l =  N_r = N_m \unlhd Q$.
  \item In  $(\alpha; \varepsilon;  \gamma)$-inverse loop $(Q, \circ)$ ${}_1N^A_l = {}_3N^A_l=  {}_2N^A_m$, ${}_2N^A_r = {}_3N^A_r =  {}_1N^A_m$,  $N_l =  N_r = N_m$.
  \item In  $(\alpha; \beta; \varepsilon)$-inverse loop $(Q, \circ)$ ${}_1N^A_l = {}_3N^A_l= {}_2N^A_r= {}_3N^A_r = {}_2N^A_m$, $N_l =  N_r = N_m \unlhd Q$.
      \item In  $(\alpha; \alpha^{-1}; \gamma)$-inverse loop $(Q, \circ)$ ${}_1N^A_l = {}_3N^A_l= {}_2N^A_r= {}_3N^A_r = {}_2N^A_m$, $N_l =  N_r = N_m \unlhd Q$.
           \item In  $(\alpha; \beta; \alpha^{-1})$-inverse loop $(Q, \circ)$ ${}_1N^A_l = {}_3N^A_l = {}_2N^A_m$,
           ${}_2N^A_r = {}_3N^A_r = {}_1N^A_m$,  $N_l =  N_r = N_m$.
\item In  $(\alpha; \beta; \beta^{-1})$-inverse loop $(Q, \circ)$ ${}_1N^A_l = {}_3N^A_l = {}_2N^A_r = {}_3N^A_r = {}_1N^A_m$,  $N_l =  N_r = N_m \unlhd Q$.
\end{enumerate}
 \end{theorem}
\begin{proof}
The proof follows from Lemma \ref{A_NUCLEI_OF_RST_INV} and the fact that in any loop ${}_1N^A_l= {}_3N^A_l$, ${}_2N^A_r= {}_3N^A_r$. Normality of Garrison's nuclei follows from Corollary \ref{Eqiviv_NORMALITY}.
\end{proof}

\begin{theorem}\label{COINCIDENCE_NUCLEI}
If in $(\alpha; \beta; \gamma)$-inverse loop $(Q,\cdot)$ $\alpha\, {}_3N_r^{A} \alpha^{-1} \subseteq {}_3N_r^{A},$ $ \beta \, {}_1N_m^{A} \beta^{-1} \subseteq {}_1N_m^{A},$ $\alpha^{-1}  {}_3N_l^{A}\alpha  \subseteq {}_3N_l^{A}$,
then in this  loop $N_l = N_m = N_r$.
\end{theorem}
\begin{proof}
By proving this theorem we use standard way \cite{kk}.
Let $L_a\in {{}_1N_l^A},$ i.e. $a\in N_l$. Then using (\ref{EUAL_A_NUCLEI_1}) we have
$\alpha^{-1} L_a \alpha \in {}_3N_r^{A}$. Therefore $ L_a \in
\alpha\,\, {}_3N_r^{A} \alpha^{-1} \subseteq {}_3N_r^{A}.$  Thus
$L_a 1 = a \in {}_3N_r^{A} 1 = N_r$.

If $a\in N_r$, then $(\varepsilon, R_a, R_a )\in {N_r^A}$ and
using (\ref{EUAL_A_NUCLEI_3})  we see that $\beta^{-1} R_a \beta \in  {}_1 N_m^{A}$. Then
$ R_a \in \beta \; {}_1N_m^{A} \beta^{-1} \subseteq \, {}_1N_m^{A}$.
Therefore $ a = R_a 1 \in  {}_1N_m^{A} 1 = N_m$.

If $a\in N_m$, then $(R_a, L^{-1}_a, \varepsilon)\in {N_m^A}$ and using (\ref{EUAL_A_NUCLEI_5})
we see that $\alpha  R_a \alpha^{-1} \in {}_3N_l^{A}$. Then
$ R_a \in \alpha^{-1}  {}_3N_l^{A}\alpha  \subseteq {}_3N_l^{A}$.
Therefore $ a=R_a 1 \in  {}_3N_l^{A} 1 = N_l$.

We have obtained $N_l \subseteq  N_r \subseteq N_m \subseteq N_l$,
therefore $N_l = N_r = N_m.$
\end{proof}

\begin{definition}
A  quasigroup $(Q, \circ)$ is
\begin{enumerate}
\item an $(r, s, t)$-inverse quasigroup if  there exist a permutation $J$ of the set $Q$  such that
$J^r (x\circ y) \circ J^s x = J^t y$ \label{2.73}
 for all $x, y \in Q$ \cite{ks3, KS}.
\item
 an $m$-inverse quasigroup if  there exist a permutation $J$ of the set $Q$  such that
$J^m (x\circ y) \circ J^{m+1} x = J^m y$ \label{2.77}
 for all $x, y \in Q$ \cite{kk, ks3, KS}.
\item
   a  $WIP$-inverse quasigroup if  there exist a permutation $J$ of the set $Q$  such that
$J (x\circ y) \circ  x = J y $\label{2.87}
 for all $x, y \in Q$ \cite{VD, RB, STB, KS}.
\item
   a  $CI$-inverse quasigroup if  there exist a permutation $J$ of the set $Q$  such that
$ (x\circ y) \circ  J x =  y \label{2.98}
$
 for all $x, y \in Q$ \cite{RA0, RA, BEL_TSU, VD}.
\end{enumerate}
\end{definition}

It is easy to see, that classes of $(r, s, t)$-inverse, $m$-inverse, $WIP$-inverse, $CI$-inverse quasigroups are  included in the class of $(\alpha, \beta, \gamma)$-inverse quasigroups.

\begin{lemma}
An $(r, s, t)$-inverse quasigroup has $((1\, 2\, 3), (J^r, J^s, J^t))$ autostrophy.

An $m$-inverse quasigroup has $((1\, 2\, 3), (J^m, J^{m+1}, J^m))$ autostrophy.

A $WIP$-inverse quasigroup has $((1\, 2\, 3), (J, \varepsilon, J))$ autostrophy.

A $CI$-inverse quasigroup has $((1\, 2\, 3), (\varepsilon, J, \varepsilon))$ autostrophy.
\end{lemma}
\begin{proof}
The proof follows from Theorem \ref{INVERS_AUTOSTR_DEF}.
\end{proof}

\begin{theorem}
\begin{enumerate}
\item If in   m-inverse loop $(Q, \circ)$ $J = I_r$, where $x\circ I_r x = 1$, then  $N_l =  N_r = N_m$.
\item In WIP-loop   ${}_1N^A_l = {}_3N^A_l=  {}_2N^A_m$, ${}_2N^A_r = {}_3N^A_r =  {}_1N^A_m$,  $N_l =  N_r = N_m$.
\item In CI-loop   ${}_1N^A_l = {}_3N^A_l=  {}_2N^A_m = {}_2N^A_r = {}_3N^A_r =  {}_1N^A_m$,  $N_l =  N_r = N_m \unlhd Q$.
\end{enumerate}
\end{theorem}
\begin{proof}
Case 1. The proof follows from Theorem \ref{COINCIDENCE_NUCLEI}.

Cases 2 and 3.
The proof follows from Theorem \ref{A_NUCL_OF_INV_LOOPS} and Corollary \ref{ISMS OF COMPONENTS_LOOPS}. \end{proof}

\subsubsection{$\lambda$-, $\rho$-,  and $\mu $-inverse quasigroups}

\begin{lemma} \label{COINCIND_OF_A_NUCL_L_R_M}
\begin{enumerate}
  \item In  $\lambda$-inverse quasigroup
   \begin{enumerate}
\item ${}_1N^A_l = \lambda^{-1}_1 {}_1N^A_m \lambda_1  = \lambda_1 \, \, {}_1N^A_m \lambda^{-1}_1$;
\item ${}_3N^A_l = \lambda^{-1}_3 {}_2N^A_m \lambda_3  = \lambda_2 \, \, {}_2N^A_m \lambda^{-1}_2$;
\item  ${}_2N^A_r = \lambda^{-1}_2 {}_3N^A_r \lambda_2  = \lambda_3 \, \, {}_3N^A_r \lambda^{-1}_3$.
\end{enumerate}
\item In  $\rho$-inverse quasigroup
\begin{enumerate}
\item ${}_1N^A_l = \rho^{-1}_1 {}_3N^A_l \rho_1  = \rho_3 \, \, {}_3N^A_l \rho^{-1}_3$;
\item ${}_2N^A_r = \rho^{-1}_2 {}_2N^A_m \rho_2  = \rho_2 \, \, {}_2N^A_m \rho^{-1}_2$;
\item  ${}_3N^A_r = \rho^{-1}_3 {}_1N^A_m \rho_3  = \rho_1 \, \, {}_1N^A_m \rho^{-1}_1$.
\end{enumerate}
    \item In  $\mu $-inverse quasigroup
  \begin{enumerate}
\item ${}_1N^A_l = \mu^{-1}_1 {}_2N^A_r \mu_1  = \mu_2 \, \, {}_2N^A_r \mu^{-1}_2$;
\item ${}_3N^A_l = \mu^{-1}_3 {}_3N^A_r \mu_3  = \mu_3 \, \, {}_3N^A_r \mu^{-1}_3$;
\item  ${}_1N^A_m = \mu^{-1}_1 {}_2N^A_m \mu_1  = \mu_2 \, \, {}_2N^A_m \mu^{-1}_2$.
\end{enumerate}
 \end{enumerate}
\end{lemma}
\begin{proof}
We can use Theorem \ref{INVERS_AUTOSTR_DEF} Lemma \ref{INVERSE_AUTOSTRP} and Table 3.
\end{proof}

A quasigroup $(Q, \cdot)$ with identities $xy = yx$, $x\cdot xy = y$ is called TS-quasigroup. In TS-quasigroup any parastrophy is an autostrophy \cite{VD}, i.e. $(Q, \cdot)^{\sigma} = (Q, \cdot)$ for all parastrophies $\sigma \in S_3$.

\begin{corollary} \label{A_NUCLEI_LIP_QUAS}
 \begin{enumerate}
      \item In  $LIP$-quasigroup  ${}_1N^A_l = \lambda \; ({}_1N^A_m) \lambda$,  $\lambda^2 = \varepsilon$,  ${}_3N^A_l = {}_2N^A_m$,  $  {}_2N^A_r = {}_3N^A_r$.
      \item In  $RIP$-quasigroup  ${}_2N^A_r = \rho \; ({}_2N^A_m) \rho$,  $\rho^2 = \varepsilon$,  ${}_1N^A_l = {}_3N^A_l$,  $  {}_1N^A_m = {}_3N^A_r$.
   \item In  $IP$-quasigroup $\lambda^2 = \varepsilon$,  $\rho^2 = \varepsilon, \,  {}_1N^A_l = {}_3N^A_l = {}_2N^A_m  \cong  {}_1N^A_m = {}_2N^A_r =   {}_3N^A_r$.
       \item In  $TS$-quasigroup ${}_1N^A_l = {}_3N^A_l = {}_1N^A_m = {}_2N^A_m = {}_2N^A_r =   {}_3N^A_r$.
 \end{enumerate}
\end{corollary}
\begin{proof}
The proof follows from Lemma \ref{COINCIND_OF_A_NUCL_L_R_M}. In $LIP$-quasigroup $\lambda_2 = \lambda_3 = \varepsilon$. In $RIP$-quasigroup $\rho_1 = \rho_3 = \varepsilon$.
\end{proof}
\smallskip

\begin{theorem}
\begin{enumerate}
  \item In  $LIP$-quasigroup   ${}_2\Theta^A_r = {}_3\Theta^A_r$ and these equivalences are  normal from the left.
      \item In  $RIP$-quasigroup  ${}_1\Theta^A_l = {}_3\Theta^A_l$, these equivalences are  normal from the right.
 \item In TS-quasigroup we have ${}_1\Theta^A_l = {}_3\Theta^A_l = {}_2\Theta^A_r = {}_3\Theta^A_r = {}_1\Theta^A_m = {}_2\Theta^A_m$ and all these equivalences are normal congruences.
  \end{enumerate}
\end{theorem}
\begin{proof}
Cases 1 and 2. The proof follows from Lemma \ref{COINCIND_OF_A_NUCL_L_R_M},  Corollaries \ref{Eqiviv_NORMALITY} and \ref{Normality_OF_quasigroup_Nuclei}.

By proving of Case 3 it is possible to use Table 3 and Corollary \ref{Eqiviv_NORMALITY}.
\end{proof}
\begin{corollary}
 In  commutative LIP-loop, commutative RIP-loop  $N_l = N_r = N_m  \trianglelefteq Q$.
 In $(\varepsilon; \mu_2; \mu_3)$-, $(\mu_1; \varepsilon; \mu_3)$-, and $(\mu_1; \mu_2; \varepsilon)$-inverse loop $N_l = N_r \trianglelefteq Q$.
\end{corollary}
\begin{proof}
The proof follows from Corollary \ref{A_NUCLEI_LIP_QUAS}.
\end{proof}

\medskip

{\bf Acknowledgment.}
The author is grateful to  Prof. A.A.~Gvaramiya for fruitful discussions on the subject  of this paper.
The author thanks Head of the Department of Mathematics and its Applications of Central European University  Professor G.~Moro\c sanu for hospitality, support, and useful advice during writing this paper.

\addcontentsline{toc}{section}{\protect{\bf \hskip2mm References}}

\end{document}